\documentclass[11pt]{amsart}

\usepackage{amssymb,amsfonts,amsthm,amsmath}

\usepackage{graphicx, verbatim, centernot}
\usepackage{xcolor}
\usepackage[normalem]{ulem}

\numberwithin{equation}{section}

\usepackage{fancyhdr}
\newtheorem{theorem}{Theorem}
\newtheorem{lemma}{Lemma}[section]

\newtheorem{proposition}{Proposition}

\newtheorem{definition}{Definition}

\newtheorem{dugma}{Example}

 \DeclareMathOperator*{\sumsum}{\sum\sum}

\def\ep{{\varepsilon}}
\newcommand{\card}{\operatorname{card}}

\newcommand{\bR}{\mathbb R}
\newcommand{\bC}{\mathbb C}
\newcommand{\bZ}{\mathbb Z}
\newcommand{\bD}{\mathbb D}
\newcommand{\bT}{\mathbb T}
\newcommand{\bN}{\mathbb N}

\newcommand{\bP}{\mathbb P}

\newcommand{\bE}{\mathbb E}

\begin{document}

\title[Zero distribution of power series]{Zero distribution of power series and binary correlation of coefficients}

\author{Jacques Benatar} 
\thanks{The first author supported by ERC Advanced Grant 692616}   
\author{Alexander Borichev}
\thanks{The second author partially supported by the project ANR-18-CE40-0035}  
\author{Mikhail Sodin}
\thanks{The third author supported by ERC Advanced Grant 692616 and by ISF Grant 382/15} 

\begin{abstract}
We study
the distribution of zeroes of power series with infinite radius of convergence. 
The coefficients of the series have the form $\xi(n)a(n)$, where $a$
is a smooth sequence of positive numbers, and $\xi$ is a sequence of complex-valued
multipliers having binary correlations and no gaps in the spectrum.
We show that under certain assumptions on the smoothness of the sequence $a$ and on
the binary correlations of the multipliers $\xi$, the zeroes of the power series
are equidistributed with respect to a radial measure defined by the sequence $a$.

We apply our approach to several examples of the sequence $\xi$:
(i) IID sequences,
(ii) sequences $e(\alpha n^2)$ with Diophantine $\alpha$,
(iii) random multiplicative sequences, (iv) the Golay--Rudin--Shapiro sequence,
(v) the indicator function of the square-free integers, (vi) the Thue--Morse sequence.
\end{abstract}

\maketitle

\section{Introduction and main results}
In this work, we study the
following problem: {\em how does the sequence of multipliers $\xi\colon \bZ_+\to\bC$
affect the zero distribution of the entire function represented by the
power series} $F_\xi(z) = \sum_{n\ge 0} \xi (n) a(n) z^n$?
The theory of entire functions has no general results pertaining to this classical
question, there are only several case studies initiated by L\'evy, Littlewood, Offord,
and others. The most studied and the only relatively well understood instances
are the cases of IID sequences $\xi$ (see Littlewood--Offord~\cite{LO}, Offord~\cite{Off-Rand1, Off-Rand2}, Kabluchko--Zaporozhets~\cite{KZ}, Nazarov--Nishry--Sodin~\cite{NNS}) and of
lacunary sequences $a$ (see Hayman~\cite{Hayman}, Hayman--Rossi~\cite{HR}, and Offord~\cite{Off-Lac1, Off-Lac2}), in which case the sequence of multipliers $\xi$ plays no essential role.

\medskip
A new approach to this problem, which is based on spectral properties of the sequence $\xi$,
was launched in~\cite{BBS, BNS}. It appears that certain estimates for the
autocorrelations
\[
\frac1{B-A}\, \sum_{A\le n< B} \xi (n) \bar{\xi}(n+h), \qquad h\in\bZ_+\,,
\]
yield the angular equidistribution of zeroes of the entire functions of exponential type
\begin{equation}\label{eq-specialF_xi}
F_\xi (z) = \sum_{n\ge 0} \xi(n)\, \frac{z^n}{n!}\,.
\end{equation}
More precisely,
\[
n_{F_\xi} (r; \theta_1, \theta_2) = \frac{\theta_2-\theta_1+o(1)}{2\pi}\, r\,, \qquad r\to\infty,
\]
where $n_{F_\xi} (r; \theta_1, \theta_2)$ is the number of zeroes of $F_\xi$ (counted with multiplicities) in the closed sector $\{z\colon 0\le |z|\le r, \theta_1 \le \arg z \le \theta_2 \}$.
In~\cite{BBS} the authors proved that the same angular equidistribution 
holds under the assumption that $\xi$ is
a Wiener sequence without lacunas in its spectrum. Recall that the sequence $\xi$ is called
{\em a Wiener sequence}, if for every $h\in\bZ_+$, the limit
\[
r(h) = \lim_{N\to\infty}\, \frac1N\, \sum_{0\le n < N} \xi(n) \bar{\xi}(n+h)
\]
exists.
Extending the sequence $r$ to $\bZ$ by setting $r(-h)=\bar{r}(h)$, we get a positive-definite sequence,
which is, therefore, given by the sequence of Fourier coefficients of a non-negative measure
$\chi_\xi$ on the unit circle $\bT$, $r=\widehat{\chi}_\xi$, $\chi_\xi\in M_+(\bT)$.
The measure $\chi_\xi$ is called {\em the spectral measure} of the Wiener sequence $\xi$, and the
closed support of $\chi_\xi$ is called {\em the spectrum} of $\xi$.
This approach permitted us to prove the angular equidistribution of zeroes of the function $F_\xi$
defined in~\eqref{eq-specialF_xi} for many  sequences $\xi$ that were intractable using
the previously known techniques.

\medskip
In this work, we advance in several directions, studying the zero distribution of $F_\xi$
on local scales and replacing the sequence $a(n)=\tfrac1{n!}$ by a rather wide class of ``smooth sequences''. The uniform transportation distance provides a convenient set-up for this study.

It is worth mentioning that our interest in
the zero distribution on local scales was a result of a conversation on some aspects of the work Lester--Matom\"aki--Radziwi{\l}{\l}~\cite{LMR}, which one of the authors had
several years ago with Steve Lester and Maks Radziwi{\l}{\l}.

\subsection{Smooth sequences $a$}
A positive sequence $a$ will be called {\em smooth} if
\begin{equation}\label{eq:a_n}
a(n) = \exp\Bigl[ -\int_0^n \varphi \Bigr]\,,
\end{equation}
where $\varphi$ is a non-negative, increasing, concave $C^2$-function on $[0, \infty)$
satisfying
\[
\varphi (0) = 0, \quad \lim_{t\to\infty} \varphi (t) = \infty, \quad \lim_{t\to\infty} \varphi' (t)= 0\,,
\]
with some quantitative bounds on $|\varphi''|$.
Different sequences $\xi$ will require different bounds on $|\varphi''|$, but
the following condition~\eqref{eq:phi} will suffice for all instances of the sequence
$\xi$ considered in this work (except for the Thue--Morse sequence,
which requires a stronger restriction):
\begin{equation}\label{eq:phi}
(\varphi')^{2+\ep} \lesssim_\ep |\varphi''|
\lesssim_\ep (\varphi')^{2-\ep}
\qquad {\rm for\ every\ } \ep>0.
\end{equation}
These assumptions are not too restrictive, and the class
of entire functions with smooth Taylor coefficients contains functions of zero and infinite order of
growth.

In what follows, we will mostly use the  inverse $\psi=\varphi^{-1}$ rather than $\varphi$ itself.
This is a convex function, which, together with its derivative, grows to $+\infty$, and satisfies the
regularity condition
\begin{equation}\label{eq:psi1}
(\psi')^{1-\ep} \lesssim_\ep \psi''
\lesssim_\ep (\psi')^{1+\ep}
\qquad {\rm for\ every\ } \ep>0,
\end{equation}
which readily follows from~\eqref{eq:phi}. In turn,~\eqref{eq:psi1} yields similar bounds for the
function $\psi'$:
\begin{equation}\label{eq:psi2}
\psi^{1-\ep} \lesssim_\ep \psi'
\lesssim_\ep \psi^{1+\ep}
\qquad {\rm for\ every\ } \ep>0.
\end{equation}

\subsection{The reference measure $\gamma$}
To measure the growth of the entire function $F_\xi$, we use a slightly smoothed maximal term
of the power series $\sum_{n\ge 0} a(n) z^n$, letting
\[
\mu (R) =\max_{t\ge 0}\, \exp\Bigl[ t\log R - \int_0^t \varphi\Bigr].
\]
Then, assuming that $R\ge 1$, we get
\[
\log \mu (R) =\int_0^{\log R} \psi (s)\, {\rm d}s  = \int_1^R \frac{\nu(r)}{r}\, {\rm d}r\,,
\]
where the function $\nu =\psi \circ \log $ is the smoothed central index of the same power series.

We expect that for Wiener sequences $\xi$ without lacunas in their spectrum, the subharmonic functions $\log|F_\xi(z)|$ and $\log \mu (|z|)$ are sufficiently close to each other, and therefore, the counting
measure of zeroes of $F_\xi$,
\[
n_{F_\xi} = \, \sum_{\lambda\colon F_\xi(\lambda)=0}\, \delta_\lambda,
\]
is close to the radial measure
\[
\gamma = (2\pi)^{-1}\, \Delta \log\mu (|z|) =
\sigma(r) r^{-1} {\rm d}r \otimes \, {\rm d}\theta,
\qquad z=re(\theta)\,,
\]
where $\sigma = \psi' \circ \log r$, and $e(\theta)=e^{2\pi {\rm i}\theta}$.

Note that we could
equally use the functions $\mu_1(R) = \sum_{n\ge 0} a(n)R^n$ or
$\mu_2(R)=\bigl(\sum_{n\ge 0} a(n)^2 R^{2n} \bigr)^{1/2}$, and the corresponding Riesz measures
$\gamma_i = (2\pi)^{-1}\Delta \log\mu_i (|z|)$, $i=1, 2$. The difference between the measures
$\gamma$, $\gamma_1$, and $\gamma_2$ is inessential for our purposes.

\subsection{($\gamma, \rho$)-equidistribution}
To measure the proximity of $n_{F_\xi}$ and $\gamma$, we use a classical lattice-point-counting idea usually attributed to Gauss: for every compact set $K\subset\bC$,
\[
m(K_-) \le \# (\bZ^2 \cap K) \le m(K_+)\,,
\]
where $m$ is the planar Lebesgue measure, $K_+$ is the $\sqrt{2}$-Euclidean neighbourhood of $K$,
and $K_-$ is a subset of $K$, which consists of points that are $\sqrt{2}$-separated from the boundary of $K$. 
Thus, $|\# (Z^2 \cap K) - m(K)|$ is bounded by the Lebesgue measure of the $\sqrt{2}$-neighbourhood of the the boundary of $K$. 

Equivalently, the uniform transportation distance between the counting
measure $\sum_{a\in\bZ^2} \delta_a$ and the Lebesgue measure $m$ does not exceed $\sqrt{2}$.
In our setting, the zeroes of $F_\xi$ play the role of lattice points, the measure $\gamma$, defined 
above, replaces the Lebesgue measure $m$, and the Euclidean metric is replaced by a slowly varying metric, which locally looks Euclidean.

We introduce a distance $d_\rho$ on $\bC$, letting
\[
d_\rho (z_1, z_2) = \inf_\ell\, \int_\ell\,  \frac{|{\rm d}w|}{\rho(w)}\,,
\qquad z_1, z_2\in \bC,
\]
where the infimum is taken over all $C^1$-curves $\ell$ connecting the points $z_1$ and $z_2$.
Here and elsewhere, $\rho$ is a positive $C^1$-smooth radial function on $\bC$
such that $\rho'(R) = o(1)$ for $R\to\infty$. We will call the function
$\rho$ a {\em radial gauge}. For any set $X\subset\bC$, we denote by
\[
X_{+\tau} = \{z\in\bC\colon d_\rho(z, X)< \tau\}
\]
the $\tau$-neighbourhood of $X$.

\begin{definition}
The counting measure $n_{F_\xi}$ is said to be
$(\gamma, \rho)$-equi\-di\-stributed if
there exist positive constants $C$ and $\tau$ such that, for any compact set $K\subset\bC$,
\[
| n_{F_\xi} (K) - \gamma (K)| \le C\gamma((\partial K)_{+\tau})\,.
\]
\end{definition}

Note that $(\gamma, \rho)$-equidistribution is an asymptotic characteristic of zeroes of 
$F_\xi$, which is not affected by multiplying the gauge $\rho$ by a positive constant.

It is worth mentioning that, under our assumptions, $(\gamma,\rho)$-equidistribu\-tion of
zeroes of $F_\xi$
is equivalent to the seemingly stronger finiteness of the uniform transportation
distance, $\operatorname{Tra}_{d_\rho}(n_{F_\xi}, \gamma) < \infty$. We will
not be using the transportation distance $\operatorname{Tra}_{d_\rho}$ in the bulk of
this work, so we recall its definition and prove the aforementioned equivalence in Appendix~A.

\subsubsection{\ldots\, and how to use it}

First, we consider the disks $R\bar\bD = \{|z|\le R\}$. Recalling that $\mu (R)$ vanishes on $[0, 1]$,
we get
\[
\gamma(R\bar\bD) = \int_1^R \frac{\sigma (r)}r\, {\rm d}r = \nu (R)\,.
\]
Furthermore, it is easy to see that if $R$ is large enough,
$ (R\bT)_{+\tau} \subset \bigl\{R-2\tau \rho (R) \le |z| \le R + 2\tau \rho (R) \bigr\} $.
Hence,
\begin{align*}
\gamma( (R\bT)_{+\tau} ) &\le \int_{R-2\tau \rho(R)}^{R+2\tau \rho(R)}
\frac{\sigma (r)}r\, {\rm d} r \\
&\simeq \sigma (R)\rho (R)/ R\,,
\end{align*}
whenever $\sigma$ stays nearly constant on the intervals
$[R-O(\rho (R)), R+O(\rho (R))]$, which will be always the case in this work.
Under our regularity assumptions, the functions $\nu=\psi\circ\log$ and $\sigma=\psi'\circ\log$
have almost the same rate of growth (see~\eqref{eq:psi2}), while for the most part, in this work,
(except of the case when $\xi$ is the Thue--Morse sequence) we may take $\rho\simeq R\sigma (R)^{-\kappa}$ with $0<\kappa<\tfrac12$.
Hence, in these cases, we get
\[
| n_{F_\xi}(R\bar\bD)-\nu (R)| = o(\nu (R)^{1-\kappa'}), \quad R\to\infty\,,
\]
for any $\kappa'<\kappa$. In the Thue--Morse case, our estimate becomes worse but still meaningful:
\[
| n_{F_\xi}(R\bar\bD)-\nu (R)| = o\bigl(\nu (R)e^{-c\sqrt{\log\nu (R)}}\,\bigr), \quad R\to\infty\,,
\]
with some $c>0$.

Now we turn to estimates on local scales and consider the disks
$\bar D(z; r) = \{w\colon |w-z|\le r\}$.
If $r$ is comparable to $\rho (|z|)$, then
$(\gamma, \rho)$-equidistribution yields the upper bound
\begin{align*}
n_{F_\xi}(\bar D(z; r)) &\lesssim \gamma (\bar D(z; r)) \\
&\simeq \sigma (|z|)\Bigl( \frac{\rho(|z|)}{|z|} \Bigl)^2\,.
\end{align*}
When $r$ becomes much larger than $\rho (|z|)$ but in such a way that
$\sigma$ remains nearly constant on $\bar D(z; r)$, the main term is
\[
\gamma(\bar D(z; r)) \simeq \frac{\sigma (|z|)}{|z|^2}\cdot r^2\,,
\]
while the error term is bounded by
\[
\gamma ((\partial D(z; r))_{+\tau}) \lesssim \frac{\sigma (|z|)}{|z|^2}\cdot \rho (|z|)r\,,
\]
which is much smaller than the main term. Hence, in such disks, the number of zeroes
$n_{F_\xi}(\bar D(z; r))$ approximately equals $\gamma (\bar D(z; r))$,
with an explicit control of the $o(1)$ term, namely 
$$
|n_{F_\xi} (D(z, r)) - \gamma (D(z, r)) | \lesssim \gamma (D(z, r)) \cdot \frac{\rho(|z|)}{r}.
$$

\subsection{Main results}
We consider several instances of sequences $\xi$ having different origins.
For each of these sequences, we prove that the zeroes of the function $F_\xi$ are
$(\gamma, \rho)$-equidistributed for some explicitly computed radial gauge $\rho$.

Recall that the measure $\gamma$ has the density $(2\pi)^{-1}\sigma(R)/R^2$, where $\sigma = \psi' \circ \log$,
with respect to the area measure on $\bC$. Hence, the gauge $\rho$ has to be at least
$\gtrsim R(\sigma (R))^{-1/2}$. Our technique needs slightly more:
$\rho (R) \gtrsim R (\log\sigma (R)/\sigma (R))^{1/2}$,
and in two instances (the IID sequences $\xi$, and $\xi (n) = e(\alpha n^2)$ with irrational $\alpha$ badly approximated by rationals), we achieve that scale. In most other instances, our method stops
at the scale $\rho (R) = R (\sigma (R))^{-c}$ with some $0<c<\tfrac12$, which we will estimate.
It is worth mentioning that we lack examples which would clarify whether these estimates
reflect the correct order of discrepancy of zeroes of $F_\xi$.

To better digest the precision of our results, it
is instructive to consider the special
case $a(n) = (n!)^{-1/2}$, when the smoothed central index of the Taylor series
is $\nu (R) = \psi(\log R) = R^2 + O(R)$, and
$\sigma (R) = \psi' (\log R) = 2R^2 +O(1)$, and, up to an inessential correction,
$ \gamma = \displaystyle \frac1{\pi} \times {\rm area\ measure}$.

In order to simplify the statements, in this section we give our results in a somewhat weaker
form than what will be proven afterwards. For this reason, we enumerate them in accordance with 
their numeration in the bulk of the paper, adding the letter ``$\sf a$". For instance,
Theorem~$\ref{thm:IID}\sf a$ is a simplified version of the more general result given in Theorem~\ref{thm:IID}.

In this section, we always tacitly assume that the sequence of smooth coefficients $a(n)$ is defined by~\eqref{eq:a_n}, with the function $\varphi$ satisfying the regularity condition~\eqref{eq:phi}.

\subsubsection{Non-degenerate IID sequence}
\phantom{A}\mbox{}

\medskip\noindent{\bf Theorem~$\bf \ref{thm:IID}\sf a$.}
{\em
Suppose that $\xi$ is a non-degenerate IID sequence satisfying\footnote{
Note that, generally speaking, IID sequences are not Wiener sequences,
unless $\xi(0)$ has a finite second moment.
}
\[
\bE\bigl[ |\xi(0)|^\ep \bigr]< \infty, \quad {with\ some\ } \ep>0.
\]
Then, almost surely, the zeroes of $F_\xi$ are $(\gamma, \rho)$-equidistributed, 
provided that
\begin{equation}\label{eq:IID}
\rho(R) \simeq R\,\sqrt{\frac{\log\sigma(R)}{\sigma(R)}}\,.
\end{equation}
}\medskip

In the case $a(n) = (n!)^{-1/2}$ the radial gauge $\rho$
boils down to $\rho(R)\simeq\sqrt{\log R}$. Interestingly, this estimate cannot be essentially improved.
Theorem~\ref{thm:IID-example} shows that if $\xi$ is
a sequence of complex Gaussian independent random variables and if $\delta$ is sufficiently small, then,
a.s., the function $F_\xi$ does not vanish on infinitely many disks of the form
$\{ |z - j^2|\le \delta (\log j)^{1/4}\}$, $j\ge 2$.

\subsubsection{The sequence $\xi (n) = e(\alpha n^2)$ with irrational $\alpha$}
This is a Wiener sequence whose spectral measure is the Lebesgue measure.
For $a(n)=1/n!$ the angular equidistribution of the zeroes of
the corresponding entire function~\eqref{eq-specialF_xi} was proven by
Eremenko and Ostrovskii in~\cite{EO}.
A more general result, pertaining to the sequences $\xi (n) = e(Q(n))$, where $Q(x)=\sum_{j=2}^d
q_j x^j$ is a polynomial of degree $d\ge 2$ with real coefficients at least one of which
is irrational, was proven in~\cite[Theorem~1]{BNS}.
Here, we will consider the case of a Diophantine irrational $\alpha$, and
our estimates will depend on the Diophantine properties of $\alpha$.

Let $\| t \|$ denote the distance from $t$ to the closest integer.

\medskip\noindent{\bf Theorem~$\bf \ref{thm:diophnatine}\sf a$.}
{\em
Let $\xi (n) = e(\alpha n^2)$.

\smallskip\noindent{\rm (i)} Suppose that for any positive integer $q\ge 2$,
\[
\| q\alpha \| \ge \frac{c(\alpha)}{q(\log q)^a}\,, \quad with\ some\ a\ge 0.
\]
Then the zeroes of the entire function $F_\xi$ are $(\gamma, \rho)$-equidistributed 
with the radial gauge $\rho=R\sigma^{-1/2}(\log\sigma)^{(a+1)/2}$.

\smallskip\noindent{\rm (ii)} Suppose that for any positive integer $q\ge 2$,
\[
\| q\alpha \| \ge \frac{c(\alpha)}{q^{1+b}}\,, \quad with\ some\ b>0\,.
\]
Then, for every $b'>b$, the zeroes of $F_\xi$ are $(\gamma, \rho)$-equidistributed 
with the radial gauge $\rho=R\sigma^{-1/(2+b')}$.
}\medskip

Note that Nassif~\cite{Nassif} studied the case $\alpha=\sqrt{2}$
and, using a technique developed by Hardy and Littlewood in~\cite{HL},
got a rather precise information about the zero distribution of $F_\xi$
(later, Littlewood~\cite{Littlewood, Littlewood2} returned to this study for other
smooth sequences $a$). Tims~\cite{Tims}
showed that Nassif results yield that the zeroes of
the function $F_\xi$ form a ``slowly varying lattice''.

\subsubsection{Random multiplicative and completely multiplicative functions}
Random multiplicative functions we deal with here are sequences $\xi$  defined
by
\[
\xi(n) =
\begin{cases}
\prod_{j=1}^k X_{p_j}, &{\rm if\ } n=\prod_{j=1}^k p_j, \\
0 &{\rm if\ } n {\rm\  is\ not\ square-free},
\end{cases}
\]
in the multiplicative case, and
\[
\xi(n) = \prod_{j=1}^k (X_{p_j})^{m_j}, \qquad {\rm if\ } n = \prod_{j=1}^k p_j^{m_j},
\]
in the completely multiplicative case. Here $X_p$ are symmetric and unimodular IIDs parameterized by the primes. A.s., the sequence $\xi$ a Wiener sequence whose spectral measure is the Lebesgue measure.

\medskip\noindent{\bf Theorem~$\bf\ref{thm:random-multiplicative}\sf a$.}
{\em Let $\xi$ be a random multiplicative function. Then, the zeroes of $F_\xi$ are 
$(\gamma, \rho)$-equidistributed with the radial gauge $\rho=R\sigma^{-c}$, for any $c<1/6$.
}\medskip

The proof of this theorem uses a randomized version of the binary Chowla conjecture (Lemma~\ref{lem10q}), which might be of independent interest.

\subsubsection{The Golay--Rudin--Shapiro sequence}
Here we consider the sequence $\xi\colon \bZ_+\to \{\pm 1\}$, which at each $n\in \bZ_+$
equals the parity of the number of (possibly overlapping) pairs of consecutive
ones in the binary expansion of $n$, i.e., if $n=\sum_{\ell\ge 0} r_\ell (n) 2^\ell$
is a binary expansion of $n$, then $\xi (n) = (-1)^{\tau (n)}$, where
$\tau (n) = \sum_{\ell\ge 0} r_l(n)r_{\ell+1}(n)$. Equivalently, this sequence
can be defined by $\xi(0)=1$, $\xi(2n)=\xi(n)$, $\xi(2n+1)=(-1)^n \xi(n)$, $n\ge 0$.
This is a Wiener sequence, and,
as in the previous cases, its spectral measure is the Lebesgue measure.

\medskip\noindent{\bf Theorem~$\bf\ref{thm:GRS-sequences}\sf a$.}
{\em Let $\xi$ be the Golay--Rudin--Shapiro sequence.
Then, the zeroes of $F_\xi$ are $(\gamma, \rho)$-equidistributed 
with the radial gauge $\rho=R\sigma^{-c}$ for any $c<1/3$.
}\medskip

\subsubsection{The indicator-function of the square-free integers}
Next we consider the sequence
$\xi =\mu^2$, where $\mu$ is the M\"obius function,
\[
\mu (n) =
\begin{cases}
(-1)^k &{\rm if\ } n=p_1\,\ldots\, p_k, \\
0 &{\rm if\ } n \textrm{\ is not square-free}.
\end{cases}
\]
It follows from classical elementary number-theoretic estimates, due to Mir\-sky~\cite{Mirsky},
that $\mu^2$ is a Wiener sequence whose spectral measure is discrete and has a dense support\footnote{
The spectral properties of the sequence $\mu^2$ have been studied in details by Cellarosi and Sinai in~\cite{CS}.}.

\medskip\noindent{\bf Theorem~$\bf\ref{thm:sq-free}\sf a$.}
{\em Let $\xi$ be the indicator function of the square-free integers.
Then, the zeroes of $F_\xi$ are $(\gamma, \rho)$-equidistributed 
with the radial gauge $\rho=R\sigma^{-c}$ for any $c<2/21$.
}
\medskip

The proof uses Mirsky's estimates.
Possibly, the bound for the exponent $c$ can be improved using more advanced tools.

\subsubsection{The Thue--Morse sequence}
The Thue--Morse sequence $\xi\colon \bZ_+\to \{\pm 1\}$
is defined by $\xi(0)=1$, $\xi(2n)=\xi(n)$, $\xi(2n+1) = - \xi(n)$, $n\ge 0$.
Equivalently, $\xi (n) = (-1)^{\omega (n)}$, where $\omega (n)$ is the number of ones in the binary expansion of $n$. It is well-known since the works by Mahler~\cite{Mahler} and
Kakutani~\cite{Kakutani}, that this is a Wiener sequence with a singular continuous spectral
measure having no gaps in its support.

\medskip\noindent{\bf Theorem~$\bf \ref{thm:TM}\sf a$.}
{\em
Let $\xi$ be the Thue--Morse sequence.
Let the function $\varphi$, which defines the smooth coefficients $a$,
in addition to the regularity assumption~\eqref{eq:phi}, satisfy the
estimate $\varphi'(t) \lesssim (\log t)^{-C}$ with some $C>1$.
Then, the zeroes of $F_\xi$ are $(\gamma, \rho)$-equidistributed 
with the radial gauge $\rho=R \exp(-c\sqrt{\log\sigma}\,)$, provided that the constant $c$ is sufficiently small.
}
\medskip

This corresponds to $\rho(R) = R \exp(-c\sqrt{\log R}\,)$ in the case $a(n) = (n!)^{-1/2}$.

\subsection*{Notation}
Throughout the paper, we will be using the following notation.
\begin{itemize}
\item $e(t)=e^{2\pi {\rm i}t}$.
\item $F_\xi(z) = \displaystyle \sum_{n\ge 0} \xi(n) a(n) z^n$ the entire function.
\item $a(n) = \displaystyle \exp\Bigl[\, -\int_0^n \varphi \,\Bigr]$ ``smooth coefficients''.
\item $\psi = \varphi^{-1}$ the inverse of $\varphi$.
\item $\mu = \displaystyle \max_{t\ge 0}\, \exp\Bigl[\,
t\log R - \int_0^t \varphi \,\Bigr]$ the smoothed maximal term.
\item $\nu = \psi\circ\log$ the smoothed central index.
\item $\sigma = \psi'\circ\log$.
\item $\gamma=(2\pi)^{-1}\Delta\log\mu(|z|) =
\sigma (r)r^{-1}\, {\rm d}r \otimes {\rm d}\theta$, $z=r e(\theta)$, the reference measure.
\item $\rho$ the radial gauge on $\bC$.
\item $d_\rho (z_1, z_2) = \displaystyle \inf_\ell\, \int_\ell \frac{|{\rm d}w|}{\rho(w)}$,
the infimum is taken over all $C^1$ curves $\ell$ connecting $z_1$ with $z_2$, the distance in $\bC$.
\item $U_{+\tau} = \{z\in\bC\colon d_\rho(z, U)<\tau\}$ the $\tau$-neighbourhood of the set
$U\subset \bC$.
\item $W_R(\theta) = \displaystyle
\sum_{|n-\nu|\le N} \xi (n)e(n\theta) e^{-(n-\nu)^2/(2\sigma)}$,
$N=A\sqrt{\sigma\log\sigma}$ with sufficiently large positive $A$, the Weyl-type exponential sum.
\item $A\lesssim B$ means that $A\le CB$ with a positive constant $C$,
$A \gtrsim B$ means $A\ge cB$ with a constant $c>0$, and $A\simeq B$ means that $A\lesssim B$ and
$A\gtrsim B$ simultaneously.
\item The sign $\ll$ means ``sufficiently smaller than'' and $\gg$ means ``sufficiently large than''. For instance, the assumption `` given $A$ and $B$ such that $A\ll B$'' means that there exists
    $c\in (0, 1)$ such that the corresponding conclusion holds for every positive $A$ and $B$ satisfying $A\le cB$.
\end{itemize}

\section{The reader's guide}
 
In Section~\ref{sect:2} we will develop a subharmonic technique, which will help us prove
$(\gamma, \rho)$-equidistribution of the counting measure $n_{F_\xi}$.
A familiar heuristic suggests that, since the subharmonic function $V(z)=\log\mu(|z|)$ nearly majorizes $\log|F_\xi(z)|$,
in order to check that their Riesz
measures are $(\gamma, \rho)$-equidistributed it suffices to verify the opposite inequality on
a sufficiently dense set of points in $\bC$.

Given a radial gauge $\rho$, set $D_w = \{z\colon |z-w|<\rho(w)\}$.

\begin{definition}\label{def:rho-dense}
We call a set $W\subset \bC$ $\rho$-dense
if the collection of disks $\{D_w\colon w\in W\}$ covers every point in $\bC$ outside
a bounded set.
\end{definition}

Proposition~\ref{Lemma-subharm} will yield  that {\em if
\[
\log|F_\xi| \le V + O(1 + \log_+ V)
\]
everywhere in $\bC$, while the opposite inequality
\begin{equation}\label{eq:LB}
\log|F_\xi| \ge V - O(1 + \log_+ V)
\end{equation}
holds on a $\rho$-dense subset of $\bC$, then the measure $n_{F_\xi}$ is
$(\gamma, \rho)$-equidistributed}.

\medskip
Thus, our task boils down to proving the lower bound~\eqref{eq:LB} on a sufficiently dense
set of points in $\bC$. The denser this set is, the smaller we may take $\rho$, that is,
on a smaller scale we will get the equidistribution of zeroes.

As a first application, in Section~\ref{sect:5},
we consider a sequence $\xi$
of non-degenerate IID random variables
having a finite moment of some positive order. Using an idea from Nguyen and Vu~\cite{NV}, we apply
Hal\'asz's anti-concentration estimate~\cite{Halasz} to the exponential sum
\[
S = \, \sum_{|n-\nu|\le \sqrt\sigma}\, \xi(n) a(n) R^n e(n\theta)
\]
on short intervals of $\theta$. Using the independence of the random variables
$S$ and $F_\xi - S$, after some computation, we get an almost sure lower bound~\eqref{eq:LB}
on a $\rho$-dense subset of $\bC$ with $V(z)=\log\mu(|z|)$ and with
the function $\rho$ as in~\eqref{eq:IID}.

\medskip
In Section~\ref{sect:4}, using our smoothness assumptions\footnote{
Some smoothness of the coefficients $(a(n))$, likely, is
indispensable for our method. On the other hand, it could be that certain versions of
equidistribution of zeroes of $F_\xi$ persist for {\em any} sequence $a$ satisfying
$ |\log a (n) | = o(n^2)$ as $n\to\infty$ (the latter condition is needed in order to exclude
entire functions with a very slow growth, which, for instance, may have all zeroes real independently
of the choice of the real-valued sequence $\xi$). At least, known results for IID sequences $\xi$
(Offord~\cite{Off-Lac1} \cite{Off-Rand2} and Nazarov--Nishry--Sodin~\cite{NNS})
do not rule this out.
}
on the coefficients $(a(n))$,
we replace the power series $F_\xi(z)$ by an exponential sum
concentrated around the central term, that is, around $n=\nu(|z|)$. This exponential sum
has ``an effective size'' slightly larger than $\sqrt{\sigma(|z|)}$. We introduce a
Weyl-type sum
\[
W_R(\theta) \stackrel{\rm def}= \sum_{|n-\nu|\le N}
\xi(n) e(n\theta) E_R(n)\,, \qquad N\gg \sqrt{\sigma\log\sigma},
\]
where
$ E_R(n) = \exp\bigl[ -\tfrac12\, (n-\nu)^2/\sigma\, \bigr] $
is a Gaussian cut-off function of effective width $\sqrt{\sigma}$ concentrated around the smoothed central index $\nu$. In
Proposition~\ref{LemmaL-W-V},
using a Laplace-type estimate, we prove the lower bound
\[
|F_\xi(Re(\theta))| \ge \mu (R)\, \bigl[ W_R(\theta) - O_\ep(\sigma(R)^\ep) \bigr]\,,
\]
valid for any $\ep>0$. Thus, we need to show that on a sufficiently dense set of points
$Re(\theta)$, we have
\begin{equation}\label{eq:LB-prelim}
|W_R(\theta)| \gtrsim \sigma(R)^c
\end{equation}
with some $c>0$. This is where the spectral properties of the Wiener sequence
$\xi$ enter.

\medskip
To get some intuition for the next step, we replace the smooth cut-off $E_R$ by the sharp one.
We are thus led to the sum
\[
\widetilde{W}_R(\theta) = \sum_{|n-\nu|\le\sqrt{\sigma}}\, \xi(n) e(n\theta)\,,
\]
which we will try estimate pointwise from below. We have
\begin{equation}\label{eq:open-brackers}
|\widetilde{W}_R(\theta)|^2 \approx
\sum_{|h|\le 2\sqrt{\sigma}} e(-h\theta)\, \sum_{|n-\nu|\le \sqrt{\sigma}}\, \xi(n)\bar\xi(n+h)\,.
\end{equation}
Since $\xi$ is a Wiener sequence, we expect that the inner sum is $\approx 2\sqrt{\sigma}\, \widehat{\chi}_\xi (h)$, where $\chi_\xi$ is the spectral measure of the sequence $\xi$, and
$\widehat{\chi}_\xi (h)$ is its $h$-th Fourier coefficient.
This raises some hope for the estimate
\[
|\widetilde{W}_R(\theta)|^2 \approx 2\sqrt{\sigma}\, \sum_{h\in\bZ} \widehat{\chi}_\xi (h) e(-h\theta)\,.
\]
If the spectral measure $\chi_\xi$ has a nice positive density $\chi_\xi'$, then the series
on the RHS converges to $\chi_\xi'(-\theta)>0$, which would yield estimate~\eqref{eq:LB-prelim}
with $c=\tfrac14$.

\medskip
To make this heuristic rigorous, we fix a $C^\infty$-smooth non-negative even function $g$
with support on $[-\tfrac12, \tfrac12]$, such that $\displaystyle \int_\bR g = 1$,
and consider the average
\[
X = X(R, \vartheta)
= \int_{R}^{R+\beta R}\, \int_{-1/2}^{1/2}
|W_s(\theta)|^2 g(\beta^{-1}(\vartheta-\theta))\, {\rm d}\nu (s) {\rm d}\theta\,,
\]
where $\beta = \beta(R) = \rho(R)/R$.

In the case of the Lebesgue spectral measure, we
start with a version of~\eqref{eq:open-brackers}, integrate with respect to $\theta$,
and show that the non-diagonal terms are negligible,
while the diagonal terms give us estimate~\eqref{eq:LB-prelim}. This will be done in Proposition~\ref{Lemma5}.
In the rest of Section~\ref{sect:7} we demonstrate
how to apply these estimates to the Wiener sequences $\xi (n) = e(\alpha n^2)$ with Diophantine
$\alpha$, random multiplicative sequences, and the Golay--Rudin--Shapiro sequence.

\medskip 
In Section~\ref{sect:8} we turn to Wiener sequences with arbitrary spectral measures $\chi_\xi$
having no gaps in their supports. In Proposition~\ref{Lemma5a} we furnish
a lower bound for $X$, which can be viewed
as a quantitative version of~\cite[Lemma~5]{BBS}. Then we will illustrate our method with two examples
of Wiener sequences with singular spectral measures, having no gaps. We consider the indicator function
of the square-free integers,  that is, $\xi=\mu^2$, where $\mu$ is the M\"obius function, and
the Thue--Morse sequence.

\section{A subharmonic lemma}\label{sect:2}
Let $\rho\colon \bC \to (0, \infty)$ be a radial gauge, that is,
a positive radial $C^1$-smooth function such that 
$\rho'(r)\to 0$ as $r\to\infty$.
We set $D_w = \{z\colon |z-w|<\rho(w)\}$
and $tD_w = \{z\colon |z-w|<t\rho(w)\}$.

The following lemma is the main result of this section.

\begin{proposition}\label{Lemma-subharm}
Let $V$ be a radial $C^2$-smooth subharmonic function with the Riesz measure
$\gamma = \Gamma\, {\rm d}m$, and let $\rho$ be a radial gauge satisfying
\begin{equation}\label{eq-1}
\Gamma (r) \rho^2(r) \to \infty\,,
\end{equation}
and
\begin{equation}\label{eq-2}
\sup\Bigl\{ \Bigl| \frac{\Gamma(r')}{\Gamma(r)} - 1\Bigr|\colon
|r'-r|\le\rho(r) \Bigr\} \to 0\,,
\end{equation}
as $r\to\infty$.
Let $V_1$ be a subharmonic function in $\bC$ satisfying
\begin{equation}\label{eq-3}
V_1 \le V + O(1+\Gamma\rho^2) \quad {\rm everywhere\ in\ } \bC\,.
\end{equation}
Let $\gamma = \tfrac1{2\pi}\, \Delta V$ and $\gamma_1 = \tfrac1{2\pi}\, \Delta V_1$ be the Riesz measures of the functions $V$ and $V_1$.
Suppose that there exists a $\rho$-dense set $W\subset\bC$ such that
\begin{equation}\label{eq-4}
V_1 \ge V - O(1+\Gamma\rho^2) \quad {\rm everywhere\ on\ } W\,.
\end{equation}
Then, for every compact set $K\subset \bC$,
\begin{equation}\label{eq:K}
\bigl| \gamma(K) - \gamma_1(K) \bigr| \lesssim \gamma\bigl( (\partial K)_{+C} \bigr).
\end{equation}
\end{proposition}

\begin{dugma}\label{dugma}
{\rm $V(z)=|z|^2$, $\Gamma = 2\pi^{-1}$. In this case, we can take any radial gauge
$\rho$ satisfying $\rho(r)\to\infty$ as $r\to\infty$.
More generally, if $V(z)=|z|^{\lambda}$, $\lambda>0$, then
$\Gamma (z) = (2\pi)^{-1}\lambda^2\, |z|^{\lambda-2}$, and the radial gauge $\rho$
should satisfy the condition $r^{\lambda/2 - 1}\rho(r)\to\infty$, as $r\to\infty$}.
\end{dugma}

We prove Proposition~\ref{Lemma-subharm} in several steps. We start with a simple lemma, which
shows that the metric $\rho(\zeta)^{-1}|{\rm d}\zeta|$ and
the distance $d_\rho$ locally behave like the Euclidean metric.
In the next step we show that
\[
\int_{2D_w} |V-V_1|\, {\rm d}m \lesssim \gamma(2D_w)m(2D_w), \qquad w\in W,
\]
provided that $|w|$ is sufficiently large.
From there, we deduce estimate~\eqref{eq:K}.

\subsection{Local estimates}

Given $\ep>0$, we choose $r(\ep)$ large enough that, for $r\ge r(\ep)$, we have
\begin{equation}\label{eq:10}
|\rho'(r)|\le \ep
\end{equation}
and
\begin{equation}\label{eq:20}
\rho (r)\le \ep r\,.
\end{equation}

\begin{lemma}\label{LemmaA}
Let $|w|\ge 2r(\ep)$, and let $z\in C\bar D_w$ with $C\ep \le \tfrac12$. Then,
\begin{equation}\label{eq:30}
|z|\ge r(\ep),
\end{equation}
and
\begin{equation}\label{eq:40}
(1-C\ep)\rho(w) \le \rho(z) \le (1+C\ep)\rho(w).
\end{equation}
Moreover, for $z, z'\in C\bar D_w$ with $C\ep \le \tfrac16$, we have
\begin{equation}\label{eq:50}
\frac1{1+3C\ep}\, \frac{|z-z'|}{\rho(w)} \le d_\rho(z, z') \le \frac1{1-C\ep}
\frac{|z-z'|}{\rho(w)}\,.
\end{equation}
\end{lemma}

\begin{proof}
Bound~\eqref{eq:30} follows from~\eqref{eq:20} combined with the triangle inequality. Bounds~\eqref{eq:40} are straightforward consequences of~\eqref{eq:10} and~\eqref{eq:30}.
To get the upper bound in~\eqref{eq:50}, we note that
\[
d_\rho(z, z') \le \int_{[z, z']} \frac{|{\rm d}\zeta|}{\rho(\zeta)} \le
\frac1{1-C\ep}\, \frac{|z-z'|}{\rho(w)}\,.
\]
To get the lower bound in~\eqref{eq:50}, we observe that if a curve $\ell$ joins $z$ and $z'$,
and exits the disk $3CD_w$, then it traverses the annulus
$3C\bar D_w\setminus CD_w$ at least twice, and therefore, the corresponding integral 
is at least
\[
\int_{\ell\cap(3C\bar D_w\setminus CD_w)} \frac{|{\rm d}\zeta|}{\rho(\zeta)}
\stackrel{\eqref{eq:40}}\ge \frac{{\rm Length}(\ell \cap(3C\bar D_w\setminus CD_w))}{(1+3C\ep)\rho(w)}
\ge \frac{4C}{1+3C\ep}\,,
\]
while
\[
\int_{[z, z']} \frac{|{\rm d}\zeta|}{\rho(\zeta)} \stackrel{\eqref{eq:40}}\le
\frac1{1-C\ep}\, \frac{|z-z'|}{\rho(w)}
\le \frac{2C}{1-C\ep} \stackrel{C\ep\le\frac16}\le \frac{4C}{1+3C\ep}
\le  \int_{\ell} \frac{|{\rm d}\zeta|}{\rho(\zeta)}\,.
\]
Thus, estimating from below the distance $d_\rho(z, z')$ we can assume that the curve $\ell$ does not
exit the closed disk $3\bar D_w$, in which case,
\[
\int_\ell  \frac{|{\rm d}\zeta|}{\rho(\zeta)} \stackrel{\eqref{eq:40}}\ge \frac1{1+3C\ep}\,
\frac{{\rm Length}(\ell)}{\rho(w)} \ge \frac1{1+3C\ep}\,
\frac{|z-z'|}{\rho(w)}\,,
\]
proving the lower bound in~\eqref{eq:50}. 
\end{proof}

\subsection{Rarefying the set $W$}\label{subsect:rarefying}

Let $W$ be a $\rho$-dense set, and let
$\widetilde W$ be a maximal subset of $W$ such that the closed disks
$\tfrac12\bar D_w$, $w\in\widetilde W$, are pairwise disjoint.

\begin{lemma}\label{lemmaB}
Let $0<\ep<\tfrac1{20}$, and let $\min\{|w|\colon w\in W\}\ge 2r(\ep)$.
Then, the set $\widetilde W$ is $2.5\rho$-dense,
while the disks $\{10\bar D_w\colon w\in\widetilde W\}$
have a bounded multiplicity of covering.
\end{lemma}

\begin{proof}
To prove the first statement, we show that
$ \bigcup_{W} D_w \subset\bigcup_{\widetilde W}2.5D_{\widetilde w} $.
Suppose that $z\in D_w$ with $w\in W$.
Then, by maximality of $\widetilde W$,
there exists $\widetilde w\in \widetilde W$ such that
$ \tfrac12 \bar D_w \cap \tfrac12 \bar D_{\widetilde w} \ne \emptyset$, whence
\begin{align*}
|z-\widetilde w| &\le |z-w|+|\widetilde w  - w|
\le \rho (w) + \tfrac12 (\rho(w) + \rho(\widetilde w)) \\
&\stackrel{\eqref{eq:40}}\le \tfrac12\, \bigl[ 3(1+\ep/2)(1-\ep/2)^{-1} + 1 \bigr] \rho(\widetilde w)
< 2.5\rho(\widetilde w)\,.
\end{align*}
That is, $D_w\subset 2.5D_{\widetilde w}$, which yields the
first part of the lemma.

To prove the second part, we assume that $z\in 10\bar D_{w_j}$ with disjoint
$w_j\in\widetilde W$, $1\le j \le N$. Then, by~\eqref{eq:40},
$\rho(w_j) \le (1-10\ep)^{-1}\rho(z) < 2\rho(z)$, whence, for any $\zeta\in\tfrac12\bar D_{w_j}$,
we have
\[
|\zeta - z| \le |\zeta-w_j| + |w_j-z| \le 10.5\, \rho(w_j)<21\rho(z).
\]
That is, all the disks
$\tfrac12 \bar D_{w_j}$ are contained in $21D_z$. Since the disks $\tfrac12 \bar D_{w_j}$
are disjoint, comparing the areas, we get
\[
\frac14\, \sum_{j=1}^N \rho(w_j)^2 \le 21^2 \rho(z)^2\,.
\]
Recalling that, for each $j$, $\rho(w_j) \stackrel{\eqref{eq:40}}\ge (1+10\ep)^{-1}\rho(z) > \tfrac23 \rho(z)$,
we see that $N\le 3^2 \cdot 21^2$, completing the proof. 
\end{proof}

\subsection{$L^1$-bound for $V-V_1$}

To simplify our writing, we replace the radial gauge $\rho$ by $2.5\rho$ and will
use notation $W$ for $\widetilde W$. That is, from now on, we assume that the disks
$\{D_w\colon w\in W\}$ cover the complex plane $\bC$, save for a bounded set,
and that four times larger disks $\{4\bar D_w\colon w\in W\}$ have a bounded multiplicity of
covering.

\begin{lemma}{\label{LemmaC}}
Under the assumptions \eqref{eq-1}--\eqref{eq-4} of Proposition~\ref{Lemma-subharm}, we have
\[
\int_{2D_w} |V-V_1|\, {\rm d}m \lesssim \gamma(2D_w) m(2D_w)\,,
\]
provided that $w\in W$, and $|w|$ is sufficiently large.
\end{lemma}

\begin{proof}
Fix $w\in W$, and let
$ V_1-V = P - G_{\gamma_1} + G_\gamma $
be the Poisson--Jensen representation of the function $V_1-V$ in the disk $3D_w$, 
see, for instance,~\cite[Theorem~3.14]{HK}).
Here $P$ is the Poisson integral of $V_1-V$ in $3D_w$, and $G_\gamma$, $G_{\gamma_1}$
are Green's potentials in $3D_w$ of the Riesz measures $\gamma$ and $\gamma_1$.
We estimate each of these terms separately.

We have
\begin{align*}
0 \le G_\gamma (z)
&= \int_{3D_w} \log\Bigl|
\frac{(3\rho(w))^2 - (z-w)(\bar\zeta-\bar w)}{3\rho(w)(z-\zeta)} \Bigr|\, {\rm d}\gamma(\zeta) \\
&\le \int_{3D_w} \log
\frac{6\rho(w)}{|z-\zeta|}\, {\rm d}\gamma(\zeta) \\
&< \int_{3D_w} \Bigl( \log_+ \frac{\rho(w)}{|z-\zeta|} + 2 \Bigr) \Gamma(\zeta)\, {\rm d}m(\zeta) \\
&\simeq
\Gamma(w) \Bigl( \int_0^{\rho(w)} \log\frac{\rho(w)}t \cdot t\,{\rm d}t + m(3D_w) \Bigr) \\
&\simeq \Gamma(w) \rho^2(w) \\
&\simeq \gamma(2D_w),\qquad z\in 3D_w.
\end{align*}
Furthermore, everywhere on $\partial(3D_w)$ we have
\[
P \lesssim 1 + \Gamma\rho^2 \simeq \Gamma (w)\rho^2(w) \simeq \gamma (2D_w)\,.
\]
Thus, by the maximum principle,
$ P \lesssim \gamma (2D_w) $ everywhere on $3\bar D_w$. 
Besides,
\begin{align*}
-P(w) &= (V-V_1)(w) + G_\gamma (w) - G_{\gamma_1}(w) \\
&\lesssim 1 + \Gamma (w)\rho^2(w) + \gamma (2D_w) \\
&\simeq \gamma(2D_w)\,.
\end{align*}
Then, by Harnack's inequality (applied to the positive harmonic function $-P + C\gamma (2D_w)$ in $3D_w$), we have $|P|\lesssim \gamma (2D_w)$ everywhere on $2\bar D_w$.

At last,
\begin{align*}
\int_{2D_w} G_{\gamma_1}\, {\rm d}m &\le \int_{2D_w} \Bigl(
\int_{3D_w} \bigl( \log_+ \frac{\rho(w)}{|z-\zeta|} + 2 \bigr)\, {\rm d}\gamma_1 (\zeta)
\Bigr)\, {\rm d}m(z) \\
&= \int_{3D_w} \Bigl(
\int_{2D_w} \bigl( \log_+ \frac{\rho(w)}{|z-\zeta|} + 2 \bigr)\, {\rm d} m(z) \Bigr)\,
{\rm d} \gamma_1 (\zeta) \\
&\lesssim m(2D_w) \gamma_1(3\bar D_w)\,,
\end{align*}
and it remains to bound $\gamma_1(3\bar D_w)$, which can be readily done using Jensen's formula:
\begin{align*}
\gamma_1(3\bar D_w) &\lesssim \int_0^{4} \frac{\gamma_1(t\bar D_w)}{t}\, {\rm d}t \\
&= \int_{-\pi}^\pi V_1(w+4\rho(w)e^{{\rm i}\theta})\, \frac{{\rm d}\theta}{2\pi} - V_1(w) \\
&\le \int_{-\pi}^\pi V(w+4\rho(w)e^{{\rm i}\theta})\, \frac{{\rm d}\theta}{2\pi} - V(w) +
O\bigl( \Gamma(w)\rho^2(w) \bigr) \\
&= \int_0^{4} \frac{\gamma(tD_w)}{t}\, {\rm d}t + O\bigl( \Gamma(w)\rho^2(w) \bigr) \\
&\simeq \Gamma(w)\rho^2(w) \simeq \gamma (2D_w)\,,
\end{align*}
completing the proof of Lemma~\ref{LemmaC}. 
\end{proof}

\subsection{Bounding the difference $\gamma-\gamma_1$}

Now, we estimate the difference $\gamma-\gamma_1 = \tfrac1{2\pi}\Delta (V-V_1)$
of the Riesz measures.

\begin{lemma}\label{LemmaD}
Let $W\subset \bC$ be a set such that
the disks $\{D_w\colon w\in W\}$ cover the complex plane $\bC$, save for
a bounded set, and suppose that the disks $\{4\bar D_w\colon w\in W\}$
have bounded multiplicity of covering.
Let $V$ and $V_1$ be subharmonic functions in $\bC$ with Riesz measures
$\gamma$ and $\gamma_1$, such that
\[
\int_{2D_w} |V-V_1|\, {\rm d}m \lesssim \gamma(2D_w) m(2D_w), \qquad w\in W,
\]
provided that $|w|$ is sufficiently large.
Then, for every compact set $K\subset\bC$, estimate~\eqref{eq:K} holds.
\end{lemma}

\begin{proof}
We will prove estimate~\eqref{eq:K} assuming that $K\subset \{|z|\ge r_0\}$
with sufficiently large $r_0$. Clearly, this yields the general case.

We will be using a smooth partition of unity associated with
the set $W$. For every $w\in W$, we choose a $C^2$-function $\varphi_w\colon \bC\to [0, 1]$ so that
$\varphi_w\bigl|_{D_w} = 1$, $\operatorname{supp}(\varphi_w) \subset 2\bar D_w$, $\|\Delta \phi_w \|_\infty\lesssim \rho(w)^{-2}$, and set
$ \varphi = \sum_{w\in W} \varphi_w \simeq 1 $,
$\psi_w = \varphi_w/\varphi$. Then,
$|\psi_w| \simeq 1  $ on $\bar D_w$, $\operatorname{supp}(\psi_w) \subset 2D_w$,
$ \| \Delta \psi_w \|_\infty \lesssim  \rho(w)^{-2}$, and, outside a bounded set,
$ \sum_{w\in W} \psi_w = 1 $.

Given a closed set $X\subset \{|z|\ge r_0\}$ with sufficiently large $r_0$, we let
\[
\Psi_X = \sum_{w\in W\colon D_w\subset X_{+4}} \psi_w\,.
\]
Then $\Psi_X\colon\bC\to [0, 1]$ is a $C^2$-smooth function with the following properties:
\begin{equation}\label{eq:Psi5}
\Psi_X \bigl|_{X} = 1
\end{equation}
and
\begin{equation}\label{eq:Psi6}
\operatorname{supp}(\Psi_X) \subset X_{+8}\,.
\end{equation}
To verify~\eqref{eq:Psi5}, we take any $z\in X$ and note that if $z\in\operatorname{supp}(\psi_w)$,
$w\in W$, then $z\in 2D_w$, and therefore, $d_\rho(z, w) \le \tfrac52$, whence, $D_w\subset X_{+4}$.
Thus,
\[
\Psi_X(z) = \sum_{w\in W\colon \psi_w(z)>0} \psi_w(z) = \sum_{w\in W} \psi_w(z) = 1\,.
\]
To check~\eqref{eq:Psi6}, we note that if $w\in W$, $D_w\subset X_{+4}$, and $z\in \operatorname{supp}(\psi_w)$,
then $d_\rho(z, w)\le 3$, and therefore, $z\in X_{+8}$.

Now, we proceed with the proof of~\eqref{eq:K}.
We let \[K_{-C} = \{z\in\bC\colon d_\rho(z, \bC\setminus K)\ge C\}.\]
Then, $(\partial K)_{+C} = K_{+C}\setminus K_{-C}$.

First, we show that
\begin{equation}\label{eq:gamma-u-b}
(\gamma - \gamma_1)(K) \lesssim \gamma(K\setminus K_{-14})\,.
\end{equation}
We assume that $K_{-8}\ne\emptyset$ (otherwise, there is nothing to prove) and let
$\Psi=\Psi_{K_{-8}}$.  Then, $\Psi = 1 $ on $K_{-8}$,
$\operatorname{supp}(\Psi) \subset (K_{-8})_{+8} \subset K$, and
\[
(\gamma-\gamma_1)(K) = \int_K \Psi\, {\rm d}(\gamma-\gamma_1) + \int_K (1-\Psi)\, {\rm d}(\gamma-\gamma_1).
\]
Estimate of the second integral on the RHS is straightforward:
\[
\int_K (1-\Psi)\, {\rm d}(\gamma-\gamma_1) = \int_{K\setminus K_{-8}} (1-\Psi)\,
{\rm d}(\gamma-\gamma_1) \le \gamma (K\setminus K_{-8}).
\]
Next, by Green's identity, we get
\begin{align*}
\int_K \Psi\, {\rm d}(\gamma-\gamma_1) &= \frac1{2\pi}\, \int_K \Delta\Psi \cdot (V-V_1)\, {\rm d}m \\
&= \frac1{2\pi}\, \int_{K\setminus K_{-8}} \Delta\Psi \cdot (V-V_1)\, {\rm d}m \\
&\le \frac1{2\pi}\,  \int_{K\setminus K_{-8}} |\Delta\Psi| \cdot |V-V_1|\, {\rm d}m \\
&\le \sum_{w\in W\cap(K_{-4}\setminus K_{-11})} \frac1{2\pi}\, \int_{\bC} |\Delta\psi_w|\cdot |V-V_1|\,
{\rm d}m \\
&\lesssim \sum_{w\in W\cap(K_{-4}\setminus K_{-11})} \rho(w)^{-2}\, \int_{2D_w} |V-V_1|\, {\rm d}m\,.
\end{align*}
By Lemma~\ref{LemmaC}, the RHS is
\[
\lesssim \sum_{w\in W\cap(K_{-4}\setminus K_{-11})} \gamma(2D_w) \lesssim
\gamma(K_{-1}\setminus K_{-14}),
\]
proving~\eqref{eq:gamma-u-b}.

Now, we verify the opposite bound
\begin{equation}\label{eq:gamma-l-b}
(\gamma_1-\gamma)(K) \lesssim \gamma(K_{+8}\setminus K_{-8}).
\end{equation}
Set $\Psi = \Psi_K$. Since $\Psi\big|_K=1$ and $\operatorname{supp}(\Psi)\subset K_{+8}$, we have
\[
(\gamma_1-\gamma)(K) = \int_{\bC} \Psi\, {\rm d}(\gamma_1-\gamma)
+ \int_{K_{+8}\setminus K} \Psi\, {\rm d}(\gamma-\gamma_1).
\]
The second integral on the RHS is
\[
\le \int_{K_{+8}\setminus K} \Psi\, {\rm d}\gamma \le \gamma(K_{+8}\setminus K).
\]
At last, arguing as above, we see that
\begin{align*}
\int_{\bC} \Psi\, {\rm d}(\gamma_1-\gamma)
&= \frac1{2\pi}\, \int_{\bC} \Delta\Psi \cdot (V_1-V)\, {\rm d} m \\
&= \frac1{2\pi}\, \int_{K_{+8}\setminus K}\Delta\Psi \cdot (V_1-V)\, {\rm d} m \\
&\le \frac1{2\pi}\, \int_{K_{+8}\setminus K}|\Delta\Psi| \cdot |V_1-V|\, {\rm d} m \\
&\le \sum_{w\in W\cap (K_{+4}\setminus K_{-4})}
\frac1{2\pi}\, \int_{2D_w} |\Delta\psi_w| \cdot |V_1-V|\, {\rm d} m \\
&\lesssim \sum_{w\in W\cap (K_{+4}\setminus K_{-4})}
\rho(w)^{-2}\, \int_{2D_w} |V_1-V|\, {\rm d} m \\
&\lesssim \sum_{w\in W\cap (K_{+4}\setminus K_{-4})} \gamma(2D_w) \\
&\lesssim \gamma(K_{+8}\setminus K_{-8}),
\end{align*}
proving~\eqref{eq:gamma-l-b}. Clearly, estimates~\eqref{eq:gamma-u-b} and~\eqref{eq:gamma-l-b}
together yield estimate~\eqref{eq:K}, completing the proof of Lemma~\ref{LemmaD}, and hence,
of Proposition~\ref{Lemma-subharm} as well. 
\end{proof}

\section{From the power series $F_\xi (z)$ to Weyl-type sums $W_R(\theta)$}
\label{sect:4}

In view of Proposition~\ref{Lemma-subharm}, we are after a lower bound for $|F|$
on a sufficiently dense set of points in $\bC$.
The main result of this section, Proposition~\ref{LemmaL-W-V},  reduces this question to the problem of
obtaining lower bounds for certain Weyl-type exponential sum.

\subsection{Regularity of $\varphi$ and $\psi$}

\begin{definition}[$\Delta$-regularity]
Let $\Delta\colon (0, \infty)\to (0, \infty)$ be a non-decreasing function satisfying
$\Delta (2s) \simeq \Delta (s)$ and $ \Delta(s) = o( \sqrt{s}\,(\log s)^{-3/2})$ as $s\to\infty$.
The function $\varphi$ will be called $\Delta$-regular,
if it is a non-negative, increasing, concave $C^2$-function on $[0, \infty)$, satisfying
$$
\varphi (0) = 0,
\quad \lim_{t\to\infty} \varphi (t) = \infty, \quad \lim_{t\to\infty} \varphi' (t)= 0\,,
$$
with
\begin{equation}\label{eq:phi-very-new}
|\varphi''| \le (\varphi')^2\, \Delta(1/\varphi')\,.
\end{equation}
\end{definition}

By $\psi=\varphi^{-1}$ we denote the inverse function to $\varphi$.
This is a convex function, which, together with its derivative, grows to $+\infty$.

We start with some simple estimates for $\Delta$-regular functions $\varphi$ and their inverses
$\psi$, which will be used throughout this work.

\begin{lemma}\label{Lemma-AuxEstimates-new}
Suppose that the function $\varphi$ is $\Delta$-regular.
Then,
\begin{itemize}
\item[{\rm (a)}] $\psi''\le \psi'\Delta(\psi')$;
\item[{\rm (b)}] $\varphi'(t)=(1+o(1))\varphi'(\tau)$, for
$\tau\to\infty$ and
$|t-\tau|\varphi'(\tau)\Delta(1/\varphi'(\tau))=o(1)$;
\item[{\rm (c)}] $\psi'(u) = (1+o(1))\psi'(v)$, for
$v\to\infty$ and $|u-v|\Delta(\psi'(v))=o(1)$;
\item[{\rm (d)}] $\varphi'(t)\Delta(1/\varphi'(t)) \gtrsim 1/t$;
\item[{\rm (e)}] $\psi'/\Delta(\psi') \lesssim \psi$.
\end{itemize}
\end{lemma}

\begin{proof}
Estimate (a) follows from~\eqref{eq:phi-very-new} combined with the formulas
\[
\psi' = \frac1{\varphi'(\psi)}\,,
\qquad \psi''=\frac{|\varphi''(\psi)|}{(\varphi'(\psi))^3}\,.
\]

Suppose that estimate (b) does not hold, i.e., that there exists a function $\ep(\tau)\to 0$
as $\tau\to\infty$, such that
\[
\limsup_{\tau\to\infty}\, \max\{|\varphi'(t)/\varphi'(\tau)-1|\colon |t-\tau|\varphi'(\tau)\Delta(1/\varphi'(\tau)) \le \ep(\tau) \} > 0.
\]
Then, for some $\delta\in (0, \tfrac12)$, there exist an arbitrarily large $\tau$ and a
$\vartheta$ such that $|\varphi' (\vartheta)/\varphi'(\tau)-1|=\delta$, while
$|\varphi' (t)/\varphi'(\tau)-1|<\delta$ everywhere on the open interval $I$ with endpoints $\tau$ and $\vartheta$ and with $ |I|\varphi'(\tau)\Delta(1/\varphi'(\tau)) \le \ep(\tau) $.
Then,
\[
\delta \lesssim \Bigl| \log\,\frac{\varphi'(\vartheta)}{\varphi'(\tau)}\, \Bigr|
\le \int_I\, \Bigl|\, \frac{\varphi''}{\varphi'}\, \Bigr| \le |I|\cdot \max_I \Bigl|\, \frac{\varphi''}{\varphi'}\, \Bigr| \le |I|\cdot \max_I \varphi' \Delta(1/\varphi').
\]
Since $\delta$ was chosen less than $\tfrac12$,
everywhere on $I$ we have $ \varphi' \le 2\varphi'(\tau) $. Hence,
$ \max_I \varphi' \Delta(1/\varphi') \lesssim \varphi'(\tau) \Delta(1/\varphi'(\tau)) $, and therefore,
\[
\delta \lesssim |I| \cdot \varphi'(\tau) \Delta(1/\varphi'(\tau)) \lesssim \ep(\tau)\,,
\]
arriving at a contradiction as $\tau\to\infty$, which proves (b).
The proof of estimate (c) follows the same pattern, so we skip it.

Estimate (d) easily follows from (b). Indeed, assume that (d) does not hold.
That is, there is a sequence $\tau_j\uparrow\infty$ such that
$\tau_j \varphi'(\tau_j)\Delta(1/\varphi'(\tau_j))\to 0$. Consider the intervals
$[0, \tau_j]$. By (b), $|\varphi'(+0)/\varphi'(\tau_j)-1|\to 0$. Since $\varphi'(\tau_j)\to 0$,
we get $\varphi'(+0) = 0$, which is impossible since $\varphi'(\tau)$ monotonically decreases to $0$
as $\tau$ grows. At last, estimate (e) is just a restatement of (d).
\end{proof}

\subsection{The central group of terms of the power series $F_\xi$}\label{subsec:Laplace}

Let
\[
F_\xi (z) = \sum_{n\ge 0} \xi (n) a(n) z^n,
\]
with
$a(n) = \displaystyle \exp\Bigl[ - \int_0^n \varphi \Bigr]$, and let
\[
\mu = \mu (R) = \max_{t\ge 0}\, \exp \Bigl[ t\log R - \int_0^t \varphi \Bigr].
\]
Then, for $R\ge 1$,
\[
\log \mu (R) = \int_0^{\log R} \psi = \int_1^R \frac{\nu (r)}{r}\, {\rm d}r\,,
\]
where $\nu = \nu (R) = \psi (\log R)$.
We also let $\sigma = \sigma (R) = \psi'(\log R)$. 
Set
\[
\omega (t) = \omega_R (t) = t\log R - \int_0^t \varphi.
\]
Then $a_k R^k = e^{\omega (k)}$, $\omega (\nu) = \log \mu$,
$\omega'(\nu)=0$, and $\omega''(\nu) = - \varphi'(\nu) = - 1/\sigma$.
We will need simple estimates for the function $\omega$.

\begin{lemma}\label{Lemma-omega}
Suppose that the function $\varphi$ is $\Delta$-regular.
Then, we have
\begin{itemize}
\item[{\rm (i)}]
$\bigl| \omega (\nu +t) - \omega (\nu) - \tfrac12 \omega''(\nu)t^2 \bigr|
\lesssim \sigma^{-2}\Delta (\sigma)\,|t|^3$, \quad
$|t|\lesssim \sqrt{\sigma\log\sigma}$;
\item[{\rm (ii)}]
$ \bigl| \omega (\nu +t) - \omega (\nu) \bigr| = (1+o(1)) \sigma^{-1}\, t^2$,
\quad
$|t|\simeq \sqrt{\sigma \log\sigma}$;
\item[{\rm (iii)}]
$\bigl| \omega'(\nu+t) \bigr| = (1+o(1)) \sigma^{-1}\, |t|$,
\quad $|t|\simeq \sqrt{\sigma\log\sigma}$.
\end{itemize}
\end{lemma}

\begin{proof}
To prove~(i), we note that, by Taylor's formula,
\[
\omega (\nu +t) - \omega (\nu) - \tfrac12 \omega''(\nu)t^2 =
\tfrac16\, \omega'''(\nu+x) t^3
= \tfrac16\, |\varphi''(\nu+x)|\, t^3,
\]
with some $x$,  $|x|\lesssim  \sqrt{\sigma\log\sigma}$. Therefore, applying first
estimate~\eqref{eq:phi-very-new} and then
Lemma~\ref{Lemma-AuxEstimates-new}(b), we see that the LHS of~(i) is
\[
\le  (\varphi' (\nu+x))^2 \Delta(1/\varphi'(\nu+x))\, |t|^3\,
\simeq \sigma^{-2}\Delta(\sigma)\, |t|^3\, ,
\]
which gives us~(i).

Estimate~(ii) follows from~(i): for $|t|\simeq \sqrt{\sigma \log\sigma}$, we have
$ |\omega''(\nu)| t^2 \simeq \log\sigma $, while
$ \sigma^{-2}\Delta(\sigma)\, |t|^3 = o(1) $.

To prove~(iii), we again apply Taylor's formula:
$ \omega'(\nu+t) = \omega''(\nu+x)\, t = |\varphi'(\nu+x)| t $
with some $x$, $|x|\lesssim\sqrt{\sigma\log\sigma}$. By Lemma~\ref{Lemma-AuxEstimates-new}(b),
we get~(iii).
\end{proof}

Next, we observe that, by $\Delta$-regularity of $\varphi$,
for large $R$, we have $\nu \gg \sqrt{\sigma\log \sigma}$
(indeed, by Lemma~\ref{Lemma-AuxEstimates-new}(e),
$ \sqrt{\sigma\log \sigma} \le \nu \cdot \Delta(\sigma) \sqrt{\sigma^{-1}\log\sigma} = o(\nu) $).
We choose
$N = A\, \sqrt{\sigma\log\sigma}$ with $A\gg 1$, and define the Weyl-type sum
\[
W_R (\theta) = \sum_{|k-\nu|\le N} \xi (k) e(k\theta) e^{-(k-\nu)^2/(2\sigma)}\,.
\]

\begin{proposition}\label{LemmaL-W-V}
Let $F_\xi(z)=\sum_{n\ge 0} \xi(n) a(n) z^n$
be an entire function with smooth coefficients $a$ and with
bounded coefficients $\xi$. Suppose that the function $\varphi$ is $\Delta$-regular.
Then, for $R\gg 1$ and $A\gg 1$,
\[
\max_{\theta\in[0,1]}\bigl| F_\xi (Re(\theta)) - \mu W_R(\theta) \bigr| \lesssim
\mu\, \Delta(\sigma)(\log\sigma)^{3/2}.
\]
\end{proposition}

This proof of this Proposition is a simple application of the classical Laplace method.

\begin{proof}
Recalling that $\mu=e^{\omega(\nu)}$ and that $\omega''(\nu)=-1/\sigma$, we
have
\begin{multline*}
\bigl| F_\xi (Re(\theta)) - \mu W_R(\theta) \bigr|
\lesssim \mu \Bigl( \sum_{0\le k < \nu-N} + \sum_{k>\nu + N} \Bigr)
\exp\bigl[\omega (k) - \omega(\nu) \bigr] \\
 + \mu\, \sum_{|k-\nu|\le N}
 \Bigl| \exp\bigl[ \omega(k)-\omega(\nu) \bigr] - \exp\bigl[\,
 \frac12\, \omega''(\nu) (k-\nu)^2 \bigr] \Bigr|\,.
\end{multline*}
We claim that the first two sums on the RHS tend to zero as $R\to\infty$.

Indeed, for $0\le k < \nu-N$, using concavity of $\omega$ and Lemma~\ref{Lemma-omega}(ii) and~(iii),
we get
\begin{align*}
\omega (k) - \omega (\nu) &=
\omega (k) - \omega (\nu-N) + \omega (\nu-N) - \omega (\nu) \\
&\le - (\nu - N - k)|\omega'(\nu-N)| - c_1 \sigma^{-1} N^2 \\
&\le -c_2 (\nu-N-k) \sigma^{-1} N - c_1 \sigma^{-1} N^2 \\
&=  - c_2 (\nu-N-k) A\, \sqrt{\sigma^{-1}\log\sigma} -  c_1 A^2\log\sigma\,.
\end{align*}
Then,
\begin{align*}
\sum_{0\le k < \nu-N} e^{\omega(k)-\omega(\nu)} &<
e^{-c_1 A^2\log\sigma}\,
\sum_{\ell\ge 0} e^{-c_2\ell A\,\sqrt{\sigma^{-1}\log\sigma}} \\
&\lesssim \sigma^{-c_1A^2} \cdot A^{-1}\, \sqrt{\sigma/\log\sigma} \\
&= o(1)\,,
\end{align*}
provided that $A$ is large enough that $c_1 A^2 > \tfrac12$.
The case $k>\nu+N$ follows by almost the same argument and we skip it.

Finally, by Lemma~\ref{Lemma-omega}(i), for $|k-\nu|\le N$, we have
\begin{align*}
\Bigl| \exp\bigl[ \omega(k) &- \omega(\nu) \bigr] - \exp\bigl[ \frac12\, \omega''(\nu) (k-\nu)^2 \bigr]\Bigr| \\
&= \exp\bigl[ -(k-\nu)^2/(2\sigma) \bigr]\,
\Bigl| \exp\bigl[ \omega(k)-\omega(\nu) - \frac12\, \omega''(\nu) (k-\nu)^2 \bigr] - 1 \Bigr| \\
&\lesssim \sigma^{-2}\Delta(\sigma) N^3\,
\exp\bigl[-(k-\nu)^2/(2\sigma)\bigr] \\
&= A^3 \sigma^{-1/2}\Delta(\sigma) (\log\sigma)^{3/2} \,
\exp\bigl[-(k-\nu)^2/(2\sigma)\bigr]\,.
\end{align*}
Hence,
\begin{multline*}
\sum_{|k-\nu|\le N} \Bigl| \exp\bigl[ \omega(k)-\omega(\nu) \bigr]
- \exp\bigl[ \frac12\, \omega''(\nu) (k-\nu)^2 \bigr]
\Bigr| \\
\lesssim
\sigma^{-1/2}\Delta(\sigma) (\log\sigma)^{3/2}\,
\sum_{\ell\ge 0} \exp\bigl[ -\ell^2/(2\sigma) \bigr] \lesssim
\Delta(\sigma) (\log\sigma)^{3/2}\,,
\end{multline*}
proving Proposition~\ref{LemmaL-W-V}.
\end{proof}

The following modification of Proposition~\ref{LemmaL-W-V} will be used in the next section,
when we will deal with random independent coefficients $\xi$.

\begin{proposition}\label{Lemma-L-W-V-upper}
Let $F_\xi$ be an entire function with smooth coefficients $a$.
Suppose that the function $\varphi$ is $\Delta$-regular and
that the sequence
$\xi$ has at most power growth: $|\xi (k)| \lesssim (k+1)^B$.
Then, for  $R\gg 1$,
\[
\bigl| F_\xi (R\,e(\theta)) \bigr| \lesssim \nu^{B}\sigma^{1/2}\, \mu\,.
\]
\end{proposition}

\begin{proof}
As in the proof of the previous lemma, we write
\[
|F_\xi (Re(\theta))|  \le \mu\, \Bigl(
\sum_{0\le k < v-N} + \sum_{|k-\nu|\le N} + \sum_{k>\nu+N}
\Bigr) |\xi (k)| \exp\bigl[ \omega(k)-\omega(\nu) \bigr]\,.
\]
Arguing as in the proof of that lemma, we see that the first sum on the RHS is bounded by $o(1)\nu^B$, while the second sum is
$\lesssim \nu^B(\Delta(\sigma)(\log\sigma)^{3/2} + \sigma^{1/2}) \lesssim \nu^B \sigma^{1/2}$.
At last, the third sum is
\begin{multline*}
\lesssim \nu^B\, \sum_{k>\nu+N} \Bigl( \frac{k}\nu \Bigr)^B\,
\exp\bigl[ -  c_1 A^2\log\sigma - c_2 (k-\nu - N) A\, \sqrt{\sigma^{-1}\log\sigma}\,  \bigr] \\
\lesssim \nu^B\, \sigma^{-c_1 A} \sum_{\ell\ge 0} \max\{1, (\ell/\nu)^B \}\,
e^{-c_2 \ell A\, \sqrt{\sigma^{-1}\log\sigma}}
= o(\nu^B),
\end{multline*}
proving the lemma.
\end{proof}

\section{IID sequences $\xi$}\label{sect:5}
Our first station is the case of IID sequences $\xi$. In this well-studied setting
the result on equidistribution of zeroes on local scales appears to be new.

\subsection{Equidistribution of zeroes on local scales}

The main result of this section is the following

\begin{theorem}\label{thm:IID}
Let $\xi$ be a sequence of independent identically distributed complex-valued random
variables with a non-degenerate distribution satisfying the moment condition
\begin{equation}\label{eq:IID1}
\exists\delta>0, \quad \bE\bigl[ |\xi(0)|^\delta \bigr] < \infty,
\end{equation}
and let $F_\xi(z)=\displaystyle \sum_{n\ge 0} \xi(n) a(n) z^n$
be a random entire function with smooth coefficients
$a(n) = \displaystyle \exp\Bigl[ -\int_0^n \varphi \Bigr]$.
Suppose that the function $\varphi$ is $\Delta$-regular
and that
\begin{equation}\label{eq:21}
\varphi'(t) \lesssim t^{-c}
\end{equation}
with some $c>0$. Let $\sigma = \psi'\circ\log$, where $\psi=\varphi^{-1}$
is the inverse function to $\varphi$.

Then, almost surely, the zeroes of $F_\xi$ are
$(\gamma, \rho)$-equidistributed with the gauge
\begin{equation}\label{eq:IID2}
\rho(R) = R\, \sqrt{\sigma^{-1}\log\sigma\,}.
\end{equation}
\end{theorem}

Note that condition~\eqref{eq:21} is equivalent to the bound $\psi' \gtrsim \psi^c$ (i.e.,
to $\sigma \gtrsim \nu^c$). By convexity of $\psi$, this yields $\psi'(s) \gtrsim s^c$, that is,
$\sigma (R) \gtrsim (\log R)^c$. 

\begin{proof}
We will show that, a.s., the subharmonic functions $V(z)=\log\mu (|z|)$ and
$V_1(z) = \log |F_\xi(z)|$ satisfy the assumptions of Proposition~\ref{Lemma-subharm}, provided that
the radial gauge $\rho$ is chosen according to~\eqref{eq:IID2} (note that
the function $\rho$ defined by~\eqref{eq:IID2} satisfies $\rho'(R) = o(1)$ as $R\to\infty$, that is, $\rho$ is indeed a radial gauge).

Recall that the Riesz measure $\gamma = (2\pi)^{-1}\, \Delta V$ equals
${\rm d}\gamma = \sigma (R)R^{-1}{\rm d}R \otimes {\rm d}\theta$, $z=Re(\theta)$, so its
density $\Gamma$ equals $(2\pi)^{-1}\sigma (R) R^{-2}$, and therefore,
$\Gamma (R) \rho^2(R) = (2\pi)^{-1}\log\sigma (R)$ satisfies condition~\eqref{eq-1}.

The verification of condition~\eqref{eq-2} is also straightforward. We have
\[
\frac{\Gamma (R)}{\Gamma (R')} = \Bigl( \frac{R'}R \Bigr)^2 \cdot
\frac{\psi'(\log R)}{\psi' (\log R')}\,.
\]
Both factors on the RHS tend to $1$ uniformly in $|R'-R|\le \rho(R)$.
For the first factor this holds since $\rho (R) = o(R)$, while for the second factor
this follows from Lemma~\ref{Lemma-AuxEstimates-new}(c).

The upper bound~\eqref{eq-3} follows from~\eqref{eq:IID1}. Indeed,
the moment condition yields that
\[
\bP\bigl[ |\xi (k)| \ge (k+1)^{2/\delta} \bigr]
\le (k+1)^{-2}\, \bE\bigl[ |\xi (k)|^{\delta} \bigr] \le C(k+1)^{-2},
\]
whence, by the Borel--Cantelli lemma, a.s., we have $|\xi (k)| \lesssim (k+1)^{2/\delta}$,
$k\ge 0$. Then, by Proposition~\ref{Lemma-L-W-V-upper},
\begin{align*}
\log |F_\xi (z)| &\le \log\mu (|z|) + O(\log\sigma (|z|)) \\
&= V(z) + O((\Gamma \rho^2)(|z|))\,,
\end{align*}
which gives us~\eqref{eq-3}.

The proof of the lower bound~\eqref{eq-4} relies on a version of the Nguyen-Vu
anti-concentration estimate for trigonometric sums.

\begin{proposition}[Nguyen--Vu]\label{lemma-Nguyen-Vu}
Let
\[
S(\theta) = \sum_{\lambda\in\Lambda} \xi_\lambda c_\lambda e(\lambda\theta),
\qquad \Lambda\subset\bZ, \ |\Lambda| = n,
\]
be a random trigonometric sum with complex-valued IID coefficients $\xi_\lambda$
having a non-degenerate
distribution, and with non-random coefficients $(c_\lambda)\subset\bC$ such that
$|c_\lambda|\ge \kappa$. Then, for any $\alpha \ge 1$, there exists $\beta<\infty$ such that,
for any interval $I\subset\bR$ of length $|I| \gtrsim 1/n$, we have
\[
\inf_{\theta\in I}\, \sup_{Z\in\bC}\,
\bP\bigl[ |S(\theta)-Z| < \kappa n^{-\beta} \bigr] \lesssim n^{-\alpha}\,.
\]
\end{proposition}

The proof of this lemma is an almost verbatim repetition of the original one~\cite[Lemma~9.2]{NV},
so we relegate it to Appendix~B, proceeding with the proof of Theorem~\ref{thm:IID}.

\medskip
Fix $R\ge 1$. Note that, by Lemma~\ref{Lemma-omega}(i),
for $|k-\nu|\le \sqrt{\sigma}$, we have
$|\omega (k)-\omega(\nu)| \lesssim 1$, whence
$a_kR^k/\mu (R) = \exp\bigl[ \omega(k) - \omega (\nu) \bigr] \simeq 1$,
so we can apply Proposition~\ref{lemma-Nguyen-Vu} with $c_k = a_kR^k/\mu (R)$.
Fix $\theta\in\bR$.
Since the random variables
$F_\xi (Re(\theta))- \mu S(\theta)$ and $\mu S(\theta)$ are independent, we conclude
that there exists $\vartheta$ so that
$|\vartheta-\theta|\lesssim 1/\sqrt{\sigma}$, and
\[
\bP \bigl[
|F_\xi(Re(\vartheta))| < \mu \sigma^{-\beta/2}
\bigr] \lesssim \sigma^{-\alpha/2}
\]
with $\alpha$ to be chosen later, and $\beta=\beta(\alpha)$.
We take the sequence $(R_j)$ such that $R_{j+1}=R_j + \tfrac12\, \rho(R_j)$,
split the circle $R_j\bT$ into $\simeq \sqrt{\sigma(R_j)}$ arcs $R_j I_{j, k}$ of the angular
size $|I_{j, k}|\simeq 1/\sqrt{\sigma(R_j)}$, and denote by $\zeta_{j, k}$ the centers of
these arcs. Then the union of the disks $D_{j, k} = \{|z-\zeta_{j, k}|\le\rho(R_j)\}$ covers
the whole complex plane, except for a bounded set.
Rarefying the set $\{\zeta_{j, k}\}$ (cf. Section~\ref{subsect:rarefying}), we assume that the
disks $D_{j, k}$ have bounded multiplicity of covering.

Consider the events
\[
X_{j, k} = \bigl\{\max_{D_{j, k}} \log |F_\xi| \le \log\mu(R_j) - \tfrac12\, \beta
\log\sigma (R_j)  \bigr\}.
\]
Our next step is to show the convergence of the series
\begin{equation}\label{eq-IID5}
\sum_{j,\, k}\, \bP\bigl[ X_{j, k} \bigr] < \infty\,.
\end{equation}
By construction, $\bP\bigl[ X_{j, k} \bigr] \lesssim \sigma (|\zeta_{j, k}|)^{-\alpha/2}$,
that is,
\[
\sum_{j,\, k}\, \bP\bigl[ X_{j, k} \bigr]
\lesssim \sum_{j,\, k} \sigma(\zeta_{j, k})^{-\alpha/2}\,.
\]
To conclude that the series on the RHS converges, we observe that
\begin{multline*}
\iint_{D_{j, k}} \frac{{\rm d}m(z)}{|z|^2\log^2|z|}
\gtrsim \frac{\rho^2(\zeta_{j, k})}{|\zeta_{j, k}|^2 \log^2|\zeta_{j, k}|} \\
= \frac{\log \sigma (|\zeta_{j, k}|)}{\sigma (|\zeta_{j, k}|)} \cdot \frac1{\log^2 |\zeta_{j, k}|}
\, \stackrel{\eqref{eq:21}}\gtrsim\, \sigma (|\zeta_{j, k}|)^{-\alpha/2}\,,
\end{multline*}
provided that $\alpha$ was chosen sufficiently large. That is, condition~\eqref{eq-IID5} holds.
Then, by the Borel--Cantelli lemma, a.s., only finitely many events $X_{j, k}$ may occur, that is,
a.s., all but finitely many disks $D_{j, k}$ contain a point $w_{j, k}$ such that
\begin{align*}
\log |F_\xi (w_{j,k})| &\ge \log\mu(|w_{j, k}|) - \tfrac12\, \beta \log\sigma (|w_{j, k}|) \\
&= V(w_{j, k}) - O((\Gamma\rho^2)(w_{j, k})).
\end{align*}
It remains to take the set $W=(w_{j, k})$, to observe that the union of the
disks $\{|z-w_{j, k}|\le 6\rho(w_{j, k})\}$ covers all but finitely many of the disks
$D_{j, k}$, so applying Proposition~\ref{Lemma-subharm} with the radial gauge $6\rho$,
we complete the proof of Theorem~\ref{thm:IID}.
\end{proof}

\subsection{A Gaussian example}
In this section we will provide an example, which shows that the result of Theorem~\ref{thm:IID}
cannot be essentially improved with respect to the size of the local scale. We consider the case
when $a(n)=1/\sqrt{n!}$ and $\xi$ is a sequence of independent standard complex-valued Gaussian random
variables, and let
\[
F(z) = \sum_{n\ge 0} \xi(n)\, \frac{z^n}{\sqrt{n!}}\,.
\]
This function is called the Gaussian Entire Function, GEF, for short.
It is distinguished from other Gaussian entire functions by the remarkable
distribution invariance of its zero set with respect to isometries of the plane,
see~\cite[Ch.~2]{HKPV} or~\cite{NS-WhatIs}.

In this case, a straightforward computation shows that the function $\varphi$ satisfies
$ \varphi (t) = \tfrac12\, \log t + O(t^{-1})$, and $\varphi'(t) = (2t)^{-1} + O(t^{-2})$.
Then,
$ \psi (s) = e^{2s} + O(1)$, and $\psi'(s) = 2e^{2s} + O(1)$,
whence, for large $|z|$,
\[ V(z) = \log\mu (|z|) =  \displaystyle \int_0^{\log |z|}\psi = \tfrac12\, |z|^2 + O(\log |z|),\]
and $\sigma (R) = \psi'(\log R) = 2R^2 + O(1)$.
Furthermore, the density of the Riesz measure of the function $V$ equals
$(2\pi)^{-1}\sigma (R)/R^2 = \pi^{-1} + O(R^{-2})$.
Hence, by Theorem~\ref{thm:IID}, a.s.,
the number of zeroes of the GEF $F$ in any disk of radius
$r$ and center $z$, with sufficiently large $|z|$ and with $r\gtrsim \sqrt{\log |z|}$,
is close to $r^2$.

The following theorem shows that this local equidistribution breaks down at the scale $\log^{1/4} t$.

\begin{theorem}\label{thm:IID-example}
Let $F$ be a GEF, and let $D_j=D(j^2, \kappa \log^{1/4}j)$, $j\ge 2$,
with a sufficiently small parameter $\kappa>0$.
Then a.s. $F$ does not vanish on infinitely many disks from the collection $(D_j)$.
\end{theorem}

\begin{proof}
We will be using the asymptotic independence property of the zero set of the GEF $F$.
For $w\in\bC$, set
\[
T_w F (z) =  F(z+w) e^{-z\bar w - \frac12 |w|^2}\,.
\]
Let $r_j=\kappa \log^{1/4}j$. Then, by Lemma~5 in~\cite{NSV}, there exist independent GEFs
$F_j$, $j\ge 2$, such that
$ T_{j^2} F =F_j+H_j $ and
$$
\mathbb P\Bigl[ \sup_{D(0,r_j)} |H_j(z)|e^{-|z|^2/2}\ge e^{-j^2}\Bigr] \le 2e^{-j^2/2}.
$$
By the Borel--Cantelli lemma, a.s., there exists $j_0$ so that, for $j\ge j_0$,
$\displaystyle \sup_{D(0,r_j)} |H_j| \le \tfrac12 $.

Now we argue as in \cite[Section~1]{ST3}. Let
$$
F_j(z)=\sum_{n\ge 0}\xi_j(n) \frac{z^n}{\sqrt{n!}}\,.
$$
For every $j$, we consider the event
$Y_j$ that
$|\xi_j(0)|\ge 2$, $|\xi_j(n)|\le e^{-r_j^2}$  for $1\le n\le 16r_j^2$,
and $|\xi_j(n)|\le 2^n$ for $n> 16r_j^2$. Then
$ \mathbb P[Y_j]\ge Ce^{-cr_j^4}\gtrsim j^{-1} $,
provided that $\kappa$ is such that $c\kappa^4\le 1$. Denote by $J$ the set of $j$s
such that the events $Y_j$ occur.
Since $\sum_j \bP[Y_j] = +\infty$, and the events $Y_j$ are independent,
by the Borel--Cantelli lemma, a.s., $\card J=\infty$.

Finally, if $j\in J$ and $j\ge j_1$, then, by a straightforward computation~\cite[Section~1]{ST3},
$ \displaystyle \inf_{D(0, r_j)} |F_j|\ge 1$,
and therefore, for $ j\ge \max(j_0,j_1) $, we have
$ \displaystyle \inf_{D(0, r_j)} |T_{j^2}F(z)|\ge \tfrac12 $.
Thus, $F$ does not vanish on $D_j$ for $j\in J$, $j\ge \max(j_0,j_1)$,
and we are done.
\end{proof}

It might be interesting to construct similar examples for other distributions
of the IID sequence $\xi$ (for instance, Rademacher or Steinhaus ones),
at least, for the same choice $a(n)=1/\sqrt{n!}$.

\section{Wiener sequences $\xi$ with the Lebesgue spectral measure}
\label{sect:7}

In Proposition~\ref{Lemma5} we give a lower bound for the Weyl-type sum $W_R(\theta)$ on
a sufficiently dense set of points $(R, \theta)$. The crucial role in this bound will be
played by a quantitative smallness of autocorrelations
$$
\frac1{B-A}\sum_{A\le k < B} \xi(k)\bar\xi(k+h), \qquad h\ne 0. 
$$
Combined
with Proposition~\ref{LemmaL-W-V} and Proposition~\ref{Lemma-subharm}, 
it will guarantee $(\gamma, \rho)$-equidistribution
of zeroes of $F_\xi$ with an appropriate radial gauge $\rho$ (Theorem~\ref{thm:main1}).
We will demonstrate how this approach works for three different instances of Wiener sequences
$\xi$ with Lebesgue spectral measure: (i) $\xi (n) = e(\alpha n^2)$ with irrational non-Liouville $\alpha$, (ii) random multiplicative sequences, and (iii) the Golay--Rudin--Shapiro sequence.

\subsection{Auxiliary estimates}
Here, we collect estimates of the function \newline\noindent $V_R(k; h)$, which will be defined
momentarily. These estimates will be used in the proofs of Proposition~\ref{Lemma5}
and Proposition~\ref{Lemma5a}.

Throughout this section we assume that the function $\varphi$ is $\Delta$-regular.
As above, $\nu=\nu(R)$ and $\sigma=\sigma(R)$. Let $\beta=\beta(R)$ be a small parameter
satisfying
\[
\sigma^{-1/2} = o(\beta), \quad
\beta\Delta(\sigma) = o(1)\,,
\]
as $R\to\infty$. Set
\begin{align*}
V_R(k; h) &= \int_R^{R(1+\beta)}
\exp\Bigl[ -\frac{(k-\nu(s))^2 + (k+h-\nu(s))^2}{2\sigma(s)}\, \Bigr]\, {\rm d}\nu(s) \\
&=
\int_{\nu(R)}^{\nu(R(1+\beta))}
\exp\Bigl[ -\frac{(k-t)^2 + (k+h-t)^2}{2\sigma(\nu^{-1}(t))}\, \Bigr]\, {\rm d}t\,.
\end{align*}
A simple and useful observation is that the function $\sigma = \psi'\circ\log $ stays
approximately constant on the integration interval $[R, R(1+\beta)]$:

\begin{lemma}\label{Lemma5A}
We have
\begin{itemize}
\item[{\rm (i)}]
$\sigma (s) = (1+o(1))\sigma $ everywhere on $[R, R(1+\beta)]$;
\item[{\rm (ii)}]
$\nu(R(1+\beta))-\nu (R) = (1+o(1))\beta\sigma$.
\end{itemize}
\end{lemma}

\begin{proof}
Since $|\log(R(1+\beta))-\log R \,| \le \beta = o(1/\Delta(\sigma))$,
relation (i) readily follows from Lemma~\ref{Lemma-AuxEstimates-new}(c).
To prove (ii), we recall that $\nu'(r) = \sigma (r)/r$. Hence,
(ii) follows from (i).
\end{proof}

The next lemma provides us with crude upper and lower bounds on $V_R$.

\begin{lemma}\label{Lemma5C}
We have
\begin{itemize}
\item[{\rm (i)}] $V_R(k; h)\lesssim \sqrt{\sigma}$;
\item[{\rm (ii)}] $V_R(k; 0) \gtrsim \sqrt{\sigma}$,
provided that  $\nu(R)\le k \le \nu(R(1+\beta))$.
\end{itemize}
\end{lemma}

\begin{proof}
By Lemma~\ref{Lemma5A}(i),
\begin{align*}
V_R(k; h) &\le \int_{\nu(R)}^{\nu(R(1+\beta))}
\exp\bigl[ - (1+o(1))\,(k-t)^2/(2\sigma)\bigr]\, {\rm d}t \\
&<\int_\bR \exp\bigl[ -(1+o(1))u^2/{2\sigma} \bigr]\, {\rm d}u \\
&\lesssim\sqrt{\sigma}\,,
\end{align*}
proving (i).

The proof of (ii) is also straightforward. By Lemma~\ref{Lemma5A}(i),
\begin{align*}
V_R(k; 0) &=
\int_{\nu(R)}^{\nu(R(1+\beta))} \exp\bigl[ - (1+o(1))(k-t)^2/(2\sigma) \bigr]\, {\rm d}t \\
&=
\int_{\nu(R)-k}^{\nu(R(1+\beta))-k} \exp\bigl[ - (1+o(1))u^2/(2\sigma) \bigr]\, {\rm d}u\,.
\end{align*}
Observe that either $\nu (R) - k \le -\tfrac12\,L$, or $\nu (R(1+\beta))-k \ge \tfrac12\, L$,
where $L=\nu (R(1+\beta))-\nu (R)$. By Lemma~\ref{Lemma5A}(ii), $L=(1+o(1))\beta\sigma$, which is
$\gg \sqrt{\sigma}$. Thus,
\[
\int_0^{\frac12\, L} \exp\bigl[ - (1+o(1))u^2/(2\sigma) \bigr]\, {\rm d}u \gtrsim \sqrt{\sigma}\,,
\]
proving the lemma.
\end{proof}

The next lemma gives us more accurate upper bounds.

\begin{lemma}\label{Lemma5B}\mbox{}
\begin{itemize}
\item[{\rm (i)}] $V_R(k; h) \lesssim e^{-ch^2/\sigma}\sqrt{\sigma}$;
\item[{\rm (ii)}]
\[
V_R(k; h) \lesssim
\begin{cases}
e^{-c(\nu(R)-k)^2/\sigma}\sqrt{\sigma}, \ &k<\nu(R), \\
e^{-c(k-\nu(R(1+\beta)))^2/\sigma}\sqrt{\sigma}, \ &k>\nu(R(1+\beta))\,.
\end{cases}
\]
\end{itemize}
\end{lemma}

\begin{proof}
We have
\begin{align*}
V_R(k; h) &=
\int_{\nu (R)-k-\frac{h}2}^{\nu(R(1+\beta))-k-\frac{h}2}
\exp\Bigl[ -\frac{1+o(1)}{2\sigma}\,
\bigl( (u-\tfrac12 h)^2 + (u+\tfrac12 h)^2 \bigr)
\Bigr]\, {\rm d}u \\
&=
\int_{\nu (R)-k-\frac{h}2}^{\nu(R(1+\beta))-k-\frac{h}2}
\exp\bigl[ -(1+o(1))(u^2 + \tfrac14 h^2)/\sigma \bigr]\, {\rm d}u \\
&<
e^{-(1+o(1))h^2/(4\sigma)}\,
\int_{\bR} e^{-(1+o(1))u^2/\sigma}\, {\rm d}u \\
&\lesssim e^{-ch^2/\sigma}\, \sqrt{\sigma}\,.
\end{align*}
To prove the second estimate, we note that, for $k<\nu(R)$, we have
\begin{align*}
V_R(k; h) &<
\int_{\nu(R)}^{\nu(R(1+\beta))}
\exp\Bigl[ -\frac{1+o(1)}{2\sigma}\, (t-k)^2 \Bigr]\, {\rm d}t \\
&<
\int_{\nu(R)-k}^\infty e^{-(1+o(1))u^2/(2\sigma)}\, {\rm d}u \\
&< e^{-(1+o(1))(\nu(R)-k)^2/(2\sigma)}\sqrt{\sigma}\,.
\end{align*}
The case $k>\nu(R(1+\beta))$ is very similar and we skip it.
\end{proof}

The next lemma estimates the oscillation of the function $k\mapsto V_R(k; h)$ along $\mathbb Z$.

\begin{lemma}\label{Lemma5D}
$ \displaystyle \sum_{k\in\bZ}
|V_R(k; h) - V_R(k-1; h)| \lesssim \sqrt{\sigma} $.
\end{lemma}

\begin{proof}
Set $m_1 = \nu(R) - A\sqrt{\sigma\log\sigma}$, $m_2 = \nu(R(1+\beta)) +A\sqrt{\sigma\log\sigma}$
with sufficiently large $A$. By Lemma~\ref{Lemma5B}(ii), the sums over $k<m_1$ and
$k>m_2$
are $o(1)$ if the constant $A$ was chosen big enough. Hence, we need to show that
\[ \sum_{m_1 \le k \le m_2} |V_R(k; h) - V_R(k-1; h)| \lesssim \sqrt{\sigma}\,. \]

We will represent the function $V_R(k; h)$ as a difference of two increasing
functions in $k$. Recalling that $\sigma (\nu^{-1}(t)) = 1/\varphi'(t))$,
we get
\begin{align*}
V_R(k; h) &= \int_{\nu(R)}^{\nu(R(1+\beta))}
\exp\Bigl[\,
- \frac{(k-t)^2+(k+h-t)^2}{2\sigma(\nu^{-1}(t))}
\,\Bigr]\, {\rm d}t \\
&= \int_{\nu(R)}^{\nu(R(1+\beta))}
\exp\bigl[\,
- \tfrac12\, \varphi'(t)((k-t)^2+(k+h-t)^2)
\,\bigr]\, {\rm d}t \\
&= \int_{\nu(R)-k-\frac12 h}^{\nu(R(1+\beta))-k-\frac12 h}
\exp\bigl[\,
- \tfrac12\, \varphi'(x+k+\tfrac12\, h)(x^2 + \tfrac14\, h^2)
\,\bigr]\, {\rm d}x \\
&= \Bigl(\, \int_{\nu(R)-k-\frac12 h}^{\nu(R(1+\beta))-m_1 -\frac{h}2}
- \int_{\nu(R(1+\beta))-k-\frac12 h}^{\nu(R(1+\beta))-m_1 -\frac{h}2} \,\Bigr)\\
&\qquad\qquad\qquad\qquad\qquad\exp\bigl[
- \tfrac12\, \varphi'(x+k+\tfrac12\, h)(x^2 + \tfrac14\, h^2)
\,\bigr]\, {\rm d}x \\
&= V_{1, R}(k; h) - V_{2, R}(k; h).
\end{align*}

First, observe that both functions $V_{1, R}$ and $V_{2, R}$ have uniform upper bounds
\begin{equation}\label{eq:UB-V_1-V_2}
V_{1, R}, V_{2, R} \lesssim \sqrt{\sigma}\,.
\end{equation}
Indeed, in the integration range, $\nu(R) \le x+k+\tfrac12\, h \le \nu(R(1+\beta)) + (m_2-m_1)
= \nu + O(\beta\sigma)$. Since $\beta\sigma\Delta(\sigma) \varphi'(\nu) =
\beta \Delta(\sigma) = o(1)$, by Lemma~\ref{Lemma-AuxEstimates-new}(b), in this range,
$\varphi'(x+k+\tfrac12\, h) = (1+o(1))\varphi'(\nu) = (1+o(1))\sigma^{-1}$,
which yields~\eqref{eq:UB-V_1-V_2}.

Next, we observe that since the function $t\mapsto \varphi'(t)$ decreases,
the functions $V_{1, R}$ and $V_{2, R}$ are increasing functions in $k$.

The rest is straightforward:
\begin{align*}
 \sum_{m_1 \le k \le m_2}
|V_R(k; h) &- V_R(k-1; h)| \\
& 
\le
\sum_{1\le j\le 2}\,\sum_{m_1 \le k \le m_2}
|V_{j, R}(k; h) - V_{j, R}(k-1; h)|
\\
& 
= \sum_{1\le j\le 2}\,\sum_{m_1 \le k \le m_2}
\bigl[ V_{j, R}(k-1; h) - V_{j, R}(k; h) \bigr]
\\
&   < \sum_{1\le j\le 2}V_{j, R}(m_1; h) 
\\
&  \lesssim \sqrt{\sigma}\,,
\end{align*}
proving the lemma.
\end{proof}

Set
\[
V_R(h) = \sum_{k\in\bZ} V_R(k; h)\,.
\]

\begin{lemma}\label{Lemma10a} We have
\[
V_R(0) \simeq \beta \sigma^{3/2}\,.
\]
\end{lemma}

\begin{proof}
To prove the lower bound, we write
\[
V_R(0) \ge \sum_{\nu(R)\le k \le \nu(R(1+\beta))}  V_R(k; 0)
\]
and note that by Lemma~\ref{Lemma5A}(ii) and Lemma~\ref{Lemma5C},
the RHS is
\[
\simeq \sqrt{\sigma} \bigl[ \nu(R(1+\beta)) - \nu (R) \bigr] \simeq \beta \sigma^{3/2}.
\]
To prove the upper bound, we split the sum into three parts:
\[
V_R(0) = \Bigl( \sum_{k<\nu (R)}\ + \sum_{\nu(R)\le k \le \nu(R(1+\beta))}
+ \sum_{k>\nu(R(1+\beta))}\, \Bigr) V_R(k; 0)\,.
\]
By Lemma~\ref{Lemma5B}(ii), the first and the third sums are $O(\sigma)=o(\beta\sigma^{3/2})$,
while, by the above, the middle sum is $\simeq \beta\sigma^{3/2}$.
\end{proof}

\begin{lemma}\label{Lemma10b} For $ h^2 \lesssim \sigma $, we have
\[
|V_R(h) - V_R(0)| \lesssim (1+ h^2\beta)\sqrt{\sigma}\,.
\]
\end{lemma}

\begin{proof}
First, we observe that
\[
\Bigl| V_R(h) - \int_{\bR} V_R(\kappa; h)\, {\rm d}\kappa \Bigr|
\lesssim \sqrt{\sigma}\,.
\]
Indeed, as in the proof of Lemma~\ref{Lemma5D}, we set
$m_1 = \nu(R) - A\sqrt{\sigma\log\sigma}$, $m_2 = \nu(R(1+\beta)) +A\sqrt{\sigma\log\sigma}$
with sufficiently large $A$. Then, by Lemma~\ref{Lemma5B}(ii),
\[
V_R(h) - \int_{\bR} V_R(\kappa; h)\, {\rm d}\kappa
= \sum_{m_1 \le k \le m_2} \Bigl[ V_R(k; h) -
\int_k^{k+1} V_R(\kappa; h)\, {\rm d}\kappa \Bigr] + o(1)\,,
\]
provided that the constant $A$ was chosen sufficiently big. It remains to recall that,
as we have shown in the proof of Lemma~\ref{Lemma5D}, the total variation of the function
$k\mapsto V_R(k; h)$ on $[m_1, m_2]$ is $\lesssim\sqrt{\sigma}$.

Thus, we need to bound the difference of the integrals
\[
\Bigl| \int_{\bR} V_R(\kappa; h)\, {\rm d}\kappa -
\int_{\bR} V_R(\kappa; 0)\, {\rm d}\kappa\Bigr|
\le \int_{m_1}^{m_2} \bigl| V_R(\kappa; h) - V_R(\kappa-\tfrac12\, h; 0) \bigr|\, {\rm d}\kappa
+ o(1)\,.
\]
Estimating the integrand on the RHS, we get
\begin{align*}
\bigl|&V_R(\kappa; h) - V_R(\kappa-\tfrac12\, h; 0) \bigr| \\
&=
\Bigl|
\int_{\nu(R)-\kappa-\frac12\, h}^{\nu (R(1+\beta))-\kappa-\frac12\, h}
\bigl[ \exp\bigl( -\varphi'(u+\kappa+\tfrac12\, h)(u^2 + \tfrac14\, h^2)\,\bigr)\\
&\qquad\qquad\qquad\qquad\qquad\qquad\qquad\qquad -\exp\bigl( -\varphi'(u+\kappa+\tfrac12\, h)u^2 \bigr]\, {\rm d}u
\Bigr| \\
&\le \int_{\nu(R)-\kappa-\frac12\, h}^{\nu (R(1+\beta))-\kappa-\frac12\, h}
\exp\bigl( - (1+o(1))u^2/\sigma \bigr)\\
&\qquad\qquad\qquad\qquad\qquad\qquad\qquad\qquad\times\bigl| \exp\bigl( -(1+o(1))h^2/(4\sigma) \bigr) - 1\bigr|\, {\rm d}u
\\
&\lesssim (h^2/\sigma) \cdot \sqrt{\sigma} = h^2/\sqrt{\sigma}\,,
\end{align*}
Thus,
\begin{multline*}
\Bigl| \int_{\bR} V_R(\kappa; h)\, {\rm d}\kappa -
\int_{\bR} V_R(\kappa; 0)\, {\rm d}\kappa\Bigr|
\lesssim (m_2-m_1)h^2/\sqrt{\sigma} + o(1) \\
\lesssim (\beta\sigma + \sqrt{\sigma\log\sigma})h^2/\sqrt{\sigma}
\lesssim h^2\beta\sqrt{\sigma}\,,
\end{multline*}
completing the proof.
\end{proof}

\subsection{Lower bound for Weyl-type sums $W_R$}\label{subsec:LB-Lebesgue}

Throughout this section we assume that the function $\varphi$ is $\Delta$-regular.
As above, $\nu=\nu(R)$ and $\sigma=\sigma(R)$. Let $\beta=\beta(R)$ be a small parameter
satisfying
\begin{equation}\label{eq:beta}
\sqrt{\frac{\log\sigma}\sigma} \ll \beta, \quad \beta\Delta(\sigma) = o(1)
\end{equation}
(later, in applications, we set $\rho(R)=R\beta(R)$).
We aim to estimate from below the exponential sum
\[
W_R(\theta) = \sum_{|k-\nu|\le N} \xi (k) e(k\theta)e^{(k-\nu)^2/(2\sigma)}\,,
\qquad N \gg \sqrt{\sigma\log\sigma}\,,
\]
on a sufficiently dense set of points $(R, \theta)$.

For $M_1<\nu(R)$ and $M_2>\nu(R(1+\beta))$, we set
\[
S^*(M_1, M_2; h) = \max_{M_1\le k \le M_2}\,
\Bigl| \sum_{k\le s \le M_2} \xi(s)\bar\xi(s+h)\Bigr|\,.
\]

\begin{proposition}\label{Lemma5}
Suppose that
\begin{equation}
\label{eq:condition1}
\sum_{\nu \le k \le \nu+\frac12 \beta\sigma}\, |\xi(k)|^2 \gtrsim \beta\sigma\,,
\end{equation}
and that, for some $p>1$,
\begin{equation}
\label{eq:condition2}
\sum_{h \ge 1} (1+\beta h)^{-p} S^*(M_1, M_2; h) \ll \beta\sigma,
\end{equation}
with
\[
M_1 = \nu (R) - \beta\sigma, \quad M_2 =  \nu(R(1+\beta)) + \beta\sigma\,.
\]
Then, for every $\vartheta\in [-\frac12, \frac12]$, there exist
\[
R'\in [R, R(1+\beta)], \quad \theta'\in (\vartheta-\beta, \vartheta+\beta),
\]
such that
\[
|W_{R'}(\theta')| \gtrsim \sigma^{1/4}\,.
\]
\end{proposition}

Note that, by Lemma~\ref{Lemma5A}(ii),
$M_2 - M_1 \lesssim \beta\sigma$, and therefore, we have 
$S^*(M_1, M_2; h) \lesssim \beta\sigma$. Hence,
for $1<q<p/(p-1)$, the terms in the sum on the LHS of~\eqref{eq:condition2} with $h\ge \beta^{-q}$ can be discarded.  
Lemma~\ref{Lemma5B}(i) implies that 
the terms with $h\gg \sqrt{\sigma\log\sigma}$ can be discarded as well.

Furthermore, the $\Delta$-regularity of $\varphi$ and our conditions \eqref{eq:beta} on $\beta$ yield that $M_1,M_2\simeq \nu(R)$.

\begin{proof}
To simplify our notation, we extend the sequence $\xi$, letting $\xi$ equal $0$
on negative integers, and set
\[
\widetilde{W}_R(\theta)
= \sum_{k\in\bZ} \xi(k)e(k\theta)\exp\Bigl[ - \frac{(k-\nu)^2}{2\sigma} \Bigr].
\]
It's easy to see that, for $N = A\,\sqrt{\sigma\log\sigma}$ with sufficiently large positive
$A$ (used in the definition of the sum $W_R$), we have
\[
\sum_{|k-\nu|\ge N} \exp\Bigl[ - \frac{(k-\nu)^2}{2\sigma} \Bigr] \lesssim
\int_N^\infty e^{-x^2/(2\sigma)}\, {\rm d}x = o(1),
\]
so that, in order to prove Proposition~\ref{Lemma5}, it will be enough to estimate the sum
$\widetilde{W}_R$, rather than $W_R$, from below.

We fix a non-negative function $g\in C_0^\infty(\bR)$, such that $\operatorname{supp}(g)\subset (-\tfrac12, \tfrac12)$, and $\displaystyle \int_\bR g = 1$, fix $\vartheta\in [-\tfrac12, \tfrac12]$,
and set
\[
\widetilde{X} =
\int_R^{R(1+\beta)}\, \int_{-1/2}^{1/2} |\widetilde{W}_s(\theta)|^2\,
g(\beta^{-1}(\vartheta-\theta))\, {\rm d}\theta\, {\rm d}\nu (s)\,.
\]
By Lemma~\ref{Lemma5A}(ii), $\nu(R(1+\beta))-\nu(R) = (1+o(1))\beta\sigma$, so, to prove Proposition~\ref{Lemma5}, 
we need to show that $\widetilde{X} \gtrsim \beta^2\sigma^{3/2}$.

First, we rewrite $\widetilde X$ as a Fourier series
\begin{align*}
\widetilde{X} &=
\int_R^{R(1+\beta)}\, \underset{k_1, k_2\in\bZ}{\sum\, \sum}\, \xi(k_1)\bar\xi(k_2)\,
\Bigl[
\int_{-1/2}^{1/2} e((k_1-k_2)\theta) g(\beta^{-1}(\vartheta-\theta))\, {\rm d}\theta
\Bigr] \\
& \qquad \qquad \qquad \qquad \qquad
\cdot \exp\Bigl[ -\frac{(k_1-\nu(s))^2 + (k_2-\nu(s))^2}{2\sigma(s)}\, \Bigr]\, {\rm d}\nu(s) \\
&= \beta\, \int_R^{R(1+\beta)}\, \underset{k_1, k_2\in\bZ}{\sum\, \sum}\,
e((k_1-k_2)\vartheta) \widehat{g}(\beta(k_2-k_1))\, \xi(k_1)\bar \xi(k_2) \\
& \qquad \qquad \qquad \qquad \qquad
\cdot \exp\Bigl[ -\frac{(k_1-\nu(s))^2 + (k_2-\nu(s))^2}{2\sigma(s)}\, \Bigr]\, {\rm d}\nu(s) \\
&= \beta\,
\sum_{h\in\bZ} \widehat{g}(\beta h) e(-h\vartheta)\,
\sum_{k\in\bZ} \xi(k)\bar\xi(k+h) V_R(k; h)\,,
\end{align*}
where
\begin{align*}
V_R(k; h) &= \int_R^{R(1+\beta)} \exp\Bigl[ -\frac{(k-\nu(s))^2 + (k+h-\nu(s))^2}{2\sigma(s)}\, \Bigr]\, {\rm d}\nu(s) \\
&=
\int_{\nu(R)}^{\nu(R(1+\beta))} \exp\Bigl[ -\frac{(k-t)^2 + (k+h-t)^2}{2\sigma(\nu^{-1}(t))}\, \Bigr]\, {\rm d}t \,.
\end{align*}
Then, we apply a usual strategy: in order to estimate the sum
\[
\widetilde{X} =
\beta\,
\sum_{h\in\bZ} \widehat{g}(\beta h) e(-h\vartheta)\,
\sum_{k\in\bZ} \xi(k)\bar\xi(k+h) V_R(k; h)\,,
\]
from below, we split it into two parts, estimate the diagonal terms ($h=0$)  from below,
and the non-diagonal terms ($|h|\ge 1$) from above:
\begin{align*}
|\widetilde{X}| &\ge \beta \sum_{k\in\bZ} |\xi(k)|^2 V_R(k; 0)
- \beta \sum_{|h|\ge 1} |\widehat{g}(\beta h)|\,
\Bigl|\,
\sum_{k\in\bZ} \xi(k)\bar\xi(k+h) V_R(k; h)
\,\Bigr| \\
&= {\rm DT} - {\rm NDT}\,.
\end{align*}
By Lemma~\ref{Lemma5C}(ii),
\[
{\rm DT} \gtrsim \beta\sqrt{\sigma}\, \sum_{\nu \le k \le  \nu + \frac12\, \beta\sigma} |\xi(k)|^2
\stackrel{\eqref{eq:condition1}}\gtrsim \beta^2\sigma^{3/2}\,,
\]
so it remains to show that the non-diagonal terms are $\ll \beta^2\sigma^{3/2}$.

Next, we cut the non-diagonal sums. Recalling that the Fourier transform $\widehat{g}(\lambda)$
decays faster than any negative power of $\lambda$, and using Lemma~\ref{Lemma5B}, given $p>1$,
we get
\[
{\rm NDT} \lesssim
\beta \sum_{1\le h \le A\sqrt{\sigma\log\sigma}} \frac1{(1+\beta h)^p}\,
\Bigl|\,
\sum_{M_1\le k \le M_2} \xi(k)\bar\xi(k+h) V_R(k; h)
\,\Bigr| + O(1)\,.
\]

Applying first summation by parts and then Lemma~\ref{Lemma5C}(i) and Lemma~\ref{Lemma5D},
we estimate the inner sum by
\begin{align*}
\Bigl|& \sum_{M_1\le k \le M_2} \xi(k)\bar\xi(k+h) V_R(k; h) \,\Bigr| \\
& \qquad \le \Bigl( V_R(M_1; h) + \sum_{M_1<k\le M_2} |V_R(k; h)-V_R(k-1; h)|\Bigr)
\cdot S^*(M_1, M_2; h) \\
& \qquad \lesssim \sqrt{\sigma} S^*(M_1, M_2; h)\,,
\end{align*}
where, as before,
\[
S^*(M_1, M_2; h) = \max_{M_1\le k \le M_2}
\Bigl|\,
\sum_{k\le s \le M_2} \xi(s)\bar\xi(s+h)
\,\Bigr|\,.
\]
Hence,
\[
{\rm NDT} \lesssim
\beta\sqrt{\sigma}\,
\sum_{1 \le h \le A\sqrt{\sigma\log\sigma}} (1+\beta h)^{-p}\,
S^*(M_1, M_2; h) + O(1) \stackrel{\eqref{eq:condition2}}\ll \beta^2\sigma^{3/2}\,,
\]
completing the proof of Proposition~\ref{Lemma5}.
\end{proof}

\subsection{Tying the loose ends together}

Combining Proposition~\ref{Lemma5} with Pro\-position~\ref{LemmaL-W-V} and Proposition~\ref{Lemma-subharm},
we arrive at the following result.

\begin{theorem}\label{thm:main1}
Let $F_\xi (z) = \sum_{n\ge 0} \xi(n) a(n) z^n $ be an entire function with smooth
coefficients $a(n) = \exp\Bigl[ -\displaystyle \int_0^n \varphi \Bigr]$ with a $\Delta$-regular
function $\varphi$. Let $\sigma = \psi'\circ\log$,
where $\psi$ is the function inverse to $\varphi$. Let $\beta$
be equal to $A\sigma^{-1/2}\log\sigma$ with $A\gg 1$, or to $\sigma^{-1/2}\log^a\sigma$
with $a>\tfrac12$,
or to $\sigma^{-c}$ with $0<c<\tfrac12$, and let $\beta \Delta(\sigma) = o(1)$ as $R\to\infty$.
Suppose that $\xi$ is a bounded sequence satisfying
conditions~\eqref{eq:condition1} and~\eqref{eq:condition2} in Proposition~\ref{Lemma5}.
Then, the zero set of $F_\xi$ is $(\gamma, \rho)$-equidistributed with the radial gauge $\rho=R\beta$.
\end{theorem}

\begin{proof}
First, we observe that the radial function $\rho=R\beta$ is a gauge.
Indeed, $\beta' = o(R^{-1})$ because of the bound $\sigma'\lesssim \sigma\Delta(\sigma)/R$, which follows from
Lemma~\ref{Lemma-AuxEstimates-new}(a) and of the condition $\beta \Delta(\sigma) = o(1)$.
Hence,  $\rho'(R)=o(1)$.

Since $\max_\theta |W_R(\theta)|
\lesssim \sum_{k\in \bZ} e^{(k-\nu)^2/(2\sigma)} \lesssim \sqrt{\sigma} $,
Proposition~\ref{LemmaL-W-V} yields the upper bound
$\log |F_\xi (z)| \le \log \mu (|z|) + O(\log \sigma (|z|))$.

Next, we note that
Proposition~\ref{Lemma5} combined with Proposition~\ref{LemmaL-W-V} produce a set
$W$ such that, at each $w\in W$, we have the matching lower bound
$ \log |F_\xi (w)| \ge \log \mu (|w|) - O(\log\sigma (|w|)) $,
and that, for some positive constant $C$,
the union of the disks $\bigcup_{w\in W} \{|z-w|\le C\rho(w)\}$
covers the complex plane, maybe, except of a bounded set.

We apply Proposition~\ref{Lemma-subharm} to the subharmonic functions
$V(z)=\log\mu (|z|)$ and $V_1(z)=\log|F_\xi(z)|$ with the radial gauge $C\rho$.
Recall that the density $\Gamma$ of the Riesz measure of $V$ equals $(2\pi)^{-1}R^{-2}\sigma$,
so that, $\Gamma \rho^2 = (2\pi)^{-1}\beta^2\sigma\gtrsim \log^2\sigma$. Since
the equidistributions with radial gauges $\rho$ and $C\rho$ are equivalent,
we are done. \end{proof}

We proceed with application of Theorem~\ref{thm:main1}. In each of the
instances we'll need to check conditions~\eqref{eq:condition1} and~\eqref{eq:condition2}.

\subsection{The sequence $\xi(n)=e(\alpha n^2)$ with Diophantine $\alpha$}

Given $\alpha\in\bR / \bZ$, set $\xi(n) = e(\alpha n^2)$. In this case, our result
depends on the diophantine properties of $\alpha$. We let $\| t \|$ be the distance from
$t$ to the closest integer, and assume that, for some non-decreasing
function $f\colon [1, \infty)\to [1, \infty)$ and for any positive integer $q$, we have
\begin{equation}\label{eq:diophantine}
\| q\alpha \| \ge \frac{c(\alpha)}{qf(q)}\,.
\end{equation}

\begin{theorem}\label{thm:diophnatine}
Let $F_\xi (z) = \sum_{n\ge 0} e(\alpha n^2) a(n) z^n $ be an entire function with smooth
coefficients $a(n) = \exp\Bigl[ -\displaystyle \int_0^n \varphi \Bigr]$ with $\Delta$-regular
function $\varphi$, and let $\sigma=\psi'\circ\log$,
where $\psi$ is the function inverse to $\varphi$.

\smallskip\noindent{\rm (i)}
Suppose that $\alpha$ satisfies the diophantine condition~\eqref{eq:diophantine} with
$f(q)=1+\log^a q$, $a\ge 0$. Then
the zero set of $F_\xi$ is $(\gamma, \rho)$-equidistributed with the radial gauge $\rho=R\sigma^{-1/2}(\log\sigma)^{(a+1)/2}$, provided that
$\Delta(s) = o( \sqrt{s}\,(\log s)^{-(a+1)/2})$, as $s\to\infty$.

\smallskip\noindent{\rm (ii)}
Suppose that $\alpha$ satisfies the diophantine condition~\eqref{eq:diophantine} with
$f(q)=q^b$, $b>0$. Then, for any $b'>b$,
the zero set of $F_\xi$ is $(\gamma, \rho)$-equidistributed with the radial gauge $\rho=R\sigma^{-1/(2+b')}$, provided that
$\Delta (s) = o(s^{1/(2+b')})$, as $s\to\infty$.
\end{theorem}

It is worth mentioning here, that the case $a=0$ (i.e., $f(q)=1$)
corresponds to $\alpha$s whose continuous fraction expansion has bounded partial quotients
(for example, quadratic irrationalities belong to this class), and that, by Khinchin's classical theorem~\cite[\S14]{Khinchin}, given $a>1$, almost every $\alpha$ satisfies the diophantine condition~\eqref{eq:diophantine} with $f(q)=1+\log^a q$.

\begin{proof}
Given $A\le B$, $h\in\bZ$, we have
\[
\Bigl|\,
\sum_{A\le s \le B} \xi (s) \bar\xi(s+h) \, \Bigr|
= \Bigl|\,
\sum_{A\le s \le B} e(-2\alpha h) \, \Bigr| \le \frac2{|1-e(-2\alpha h)|}\,.
\]
Hence, $S^*(M_1, M_2; h) \le 2|1-e(-2\alpha h)|^{-1}$, so, in order to satisfy
condition~\eqref{eq:condition2}, we will choose $\beta$ so that, for some positive $p$,
\begin{equation}\label{eq:sum}
\sum_{h\ge 1} (1+\beta h)^{-p}\, |1-e(-2\alpha h)|^{-1}
\ll \beta\sigma\,.
\end{equation}

We let $S_{k, H} = \{1\le h \le H\colon |1-e(-2\alpha h)| \le 2^{-k}\}$ and estimate the
cardinality of $S_{k, H}$ by showing that any two distinct integers in $S_{k, H}$ are
well-separated. If $h_1, h_2\in S_{k, H}$, $h_1\ne h_2$, then
$|1-e(-2\alpha (h_1-h_2))| \le |1-e(-2\alpha h_1)| +  |1-e(-2\alpha h_2)| \le 2^{1-k}$.
On the other hand, $|1-e(-2\alpha (h_1-h_2))|\ge 4\| 2\alpha (h_1-h_2)\|$.
Thus, $ \| 2\alpha (h_1-h_2)\| \le 2^{-k-1}$. Then,
\[
\frac1{|h_1-h_2|f(|h_1-h_2|)} \le C(\alpha) 2^{-k}\,,
\]
and therefore, $|h_1-h_2| \gtrsim 2^k/f(2H)$, whence, $|S_{k, H}|\lesssim Hf(2H)2^{-k}$.

Estimating the sum on the LHS of~\eqref{eq:sum}, we split it into the blocks of the length $\beta^{-1}2^\ell$, $\ell\ge 0$. Summing over the $\ell$th block, we take $H_\ell = \beta^{-1}2^\ell$.
We get
\begin{align*}
{\rm LHS\ of\ }\eqref{eq:sum} &\lesssim
\sum_{\ell\ge 0} 2^{-p\ell}
\sum_{0\le k \lesssim \log(\beta^{-1} 2^\ell)} 2^k \cdot
|S_{k, H_\ell}| \\
&\lesssim
\sum_{\ell\ge 0} 2^{-p\ell}
\sum_{0\le k \lesssim \log(\beta^{-1} 2^\ell)} 2^k \cdot H_\ell f(2H_\ell) 2^{-k} \\
&\lesssim
\frac1\beta\, \sum_{\ell\ge 0} 2^{(1-p)\ell} f\Bigl( \frac1\beta\, 2^{1+\ell} \Bigr)
\cdot \log\Bigl( \frac1{\beta} 2^\ell \Bigr)\,.
\end{align*}

First, we assume that $\alpha$ satisfies~\eqref{eq:diophantine} with  $f(q)=1+\log^a q$.
In this case,
\[
{\rm LHS\ of\ }\eqref{eq:sum} \lesssim
\frac1\beta\, \sum_{\ell\ge 0} 2^{(1-p)\ell} \cdot \log^{a+1}\Bigl( \frac1{\beta} 2^\ell \Bigr)
\lesssim \frac1\beta\, \log^{a+1}\Bigl( \frac1{\beta} \Bigr)\,,
\]
provided that we took $p>1$. To guarantee that $\beta^{-1} \log^{a+1}(\beta^{-1}) \ll \beta\sigma$,
we take $\beta = C\, (\sigma^{-1}\log^{a+1}\sigma)^{1/2}$
with sufficiently large $C$, proving the theorem in the case~(i).

Similarly, for $\alpha$ satisfying~\eqref{eq:diophantine} with  $f(q)=q^b$,
we have
\[
{\rm LHS\ of\ }\eqref{eq:sum} \lesssim
\frac1\beta\, \sum_{\ell\ge 0} 2^{(1-p)\ell} \cdot \Bigl( \frac1{\beta} 2^\ell \Bigr)^b
\log\Bigl( \frac1{\beta}\, 2^\ell \Bigr)
\lesssim \Bigl( \frac1{\beta} \Bigr)^{1+b''}\,,
\]
provided that we took $p>b+1$, $b<b''<b'$. This time, to guarantee that
$\beta^{-(1+b'')} \ll \beta\sigma$, we take $\beta  =\sigma^{-1/(2+b')}$,
proving the theorem in the case~(ii).
\end{proof}

\subsection{Random multiplicative and completely multiplicative sequen\-ces}

Denote by $\mathsf P$ the set of primes.
Let $(X_p)_{p\in\mathsf P}$ be a sequence of independent identically distributed
unimodular random variables. Suppose that they are symmetric (that is,
$X_p$ and $-X_p$ are equidistributed), for instance, the Rademacher or the Steinhaus random variables will do. Then $\bE\,[X_p^n\bar X_p^m]=0$ if $n-m$ is odd.

Consider two random multiplicative functions:
\begin{align*}
\xi_1(n)&=
\begin{cases}
\prod_{p|n}X_p, & n {\rm\ is\ square-free},\\
0, & {\rm otherwise},
\end{cases} \\
\xi_2(n)&=\prod_{p^a||n}X^a_p.
\end{align*}
The function $\xi_1$ is a random counterpart of the M\"obius function,
the function $\xi_2$ is completely multiplicative.

\begin{theorem}\label{thm:random-multiplicative}
Let $\xi$ be a random multiplicative sequence $\xi_1$ or $\xi_2$.
Let $F_\xi (z) = \sum_{n\ge 0} \xi(n) a(n) z^n $ be an entire function with smooth
coefficients $a(n) = \exp\Bigl[ -\displaystyle \int_0^n \varphi \Bigr]$, with
a $\Delta$-regular function $\varphi$, such that, for every $\ep>0$,
\begin{equation}\label{eq:phi'}
\varphi'(t) = o(t^{-1+\ep}), \qquad t\to\infty.
\end{equation}
Let $\sigma=\psi'\circ\log$,
where $\psi$ is the function inverse to $\varphi$.
Then, almost surely,
the zero set of $F_\xi$ is $(\gamma, \rho)$-equidistributed with the radial gauge $\rho=R\sigma^{-c}$
with any $0<c<1/6$, provided that $\Delta(s) = o(s^c)$ as $s\to\infty$.
\end{theorem}

The proof will use the following estimate for the binary correlations of $\xi$,
which improves Lemma~9 in~\cite{BBS} and, probably, is of independent
interest.

\begin{lemma} \label{lem10q}
Let $a\in(0,1)$, $b>0$. Then
$$
\mathbb E\,\Bigl[\,
\Bigl|\, \sum_{x\le k< (1+\eta)x}\xi(k)\bar\xi(k+h)\, \Bigr|^2\, \Bigr]
\lesssim \eta x^{1+b},
$$
provided that $0<h\lesssim \eta x^{1-a}$ and $x^{a-1}\le\eta\le 1$.
\end{lemma}

\begin{proof}Let $h > 0$.
We have
\begin{multline*}
Y \stackrel{\rm def}=
\mathbb E\,\Bigl[\, \Bigl|\sum_{x\le k< (1+\eta)x}\xi(k)\bar\xi(k+h)\Bigr|^2 \Bigr] \\
= \sumsum_{n_1,n_2\in[x,(1+\eta)x]}\,
\mathbb E\,\bigl[ \xi(n_1)\bar\xi(n_1+h)\bar\xi(n_2)\xi(n_2+h)\bigr].
\end{multline*}
Observe that if
\begin{equation}
\mathbb E\,\bigl[\xi(n_1)\bar\xi(n_1+h)\bar\xi(n_2)\xi(n_2+h)\bigr]\not=0,
\label{11q}
\end{equation}
then $n_1(n_1+h)n_2(n_2+h)$ is a square.

Denote $d_1=\gcd(n_1,n_1+h)$, $d_2=\gcd(n_2,n_2+h)$, $k_1=n_1/d_1$, $k_2=n_2/d_2$.
Since $d_1$ and $d_2$ divide $h$, the number of possible pairs
$(d_1,d_2)$ is bounded by $\tau^2(h)\lesssim_\varepsilon x^\varepsilon$, where $\tau$ is the
divisor function. Fix $d_1$ and $d_2$.

\medskip\noindent Case 1: $\xi=\xi_1$. Under condition \eqref{11q} we have
$$
k_1\Big(k_1+\frac{h}{d_1}\Big)=k_2\Big(k_2+\frac{h}{d_2}\Big).
$$
Therefore, for every $k_1$, there exists at most two possible values for $k_2$ and, hence,
$Y\lesssim_\varepsilon \eta x^{1+\varepsilon}$.

\medskip\noindent Case 2: $\xi=\xi_2$.
Let $e_1^2,f_1^2,e_2^2,f_2^2$ be the largest square divisors of, correspondingly,
$k_1$, $k_1+\frac{h}{d_1}$, $k_2$, $k_2+\frac{h}{d_2}$. Under condition \eqref{11q} we have
\begin{equation}
\frac{k_1(k_1+\frac{h}{d_1})}{e_1^2f_1^2}=\frac{k_2(k_2+\frac{h}{d_2})}{e_2^2f_2^2}.
\label{12q}
\end{equation}
First, the left hand side of \eqref{12q} is determined by $n_1$ and, hence, takes at most $\eta x$ possible square-free values $m$.
For every such value $m$ and for every triple $(k_2,e_2,f_2)$ satisfying \eqref{12q} there are $m'$ and $m''$ verifying the equations
$$
\begin{cases}
m=m'm'',\\
k_2=m'e_2^2,\\
k_2+\frac{h}{d_2}=m''f_2^2.
\end{cases}
$$
For fixed $m$, the number of such couples $(m',m'')$ is $\lesssim_\varepsilon x^\varepsilon$.
Given $m'$ and $m''$, it remains to solve the equation
\begin{equation}
m''f_2^2-m'e_2^2=\frac{h}{d_2}.
\label{14q}
\end{equation}
Now, \cite[Proposition 1]{CilGar} shows (the discriminant $4m'm''=4m$ is not a square) that the
number of solutions $(e_2,f_2)$ to \eqref{14q} is $\lesssim_\varepsilon x^\varepsilon$.
Finally, $Y\lesssim_\varepsilon \eta x^{1+\varepsilon}$, proving the lemma.
\end{proof}

The next lemma is a simple corollary to the previous one.

\begin{lemma}\label{lem10q'}
Let $0<a<1$, $A>1/(1-a)$, $b'>0$, and $H=H(m)\lesssim m^{A(1-a)-1}$.
Then, almost surely,
\[
\frac1H\, \sum_{1\le h  \le H}
\Bigl|\sum_{m^A\le k < (m+1)^A}\xi(k)\bar\xi(k+h)\Bigr| \le
m^{\frac12\, A(1+b')}\,,
\]
provided that $m$ is sufficiently large.
\end{lemma}

\begin{proof}
Applying the Cauchy--Schwarz inequality, we have
\begin{multline*}
\Bigl(\sum_{1\le h\le H}\, \Bigl|\sum_{m^A\le k < (m+1)^A}\xi(k)\bar\xi(k+h)\Bigr|\Bigr)^2 
\\ \le
H \sum_{ 1 \le h\le H}\, \Bigl|\sum_{m^A\le k < (m+1)^A}\xi(k)\bar\xi(k+h)\Bigr|^2.
\end{multline*}
Set $\lambda = H m^{\frac12\, A(1+b')}$.
Applying Lemma~\ref{lem10q} with $x=m^A$, $\eta = (m+1)^A/m^A -1 \simeq m^{-1}$, and with
$0<b<b'$, we obtain
\begin{align*}
\mathbb P\,\Bigl[
\sum_{1\le h \le H}\,
\Bigl| &\sum_{m^A\le k < (m+1)^A}\xi(k)\bar\xi(k+h)\Bigr| \ge \lambda \Bigr] \\
&\le
\mathbb P\,\Bigl[
\sum_{1\le h \le H}\,
\Bigl| \sum_{m^A\le k < (m+1)^A}\xi(k)\bar\xi(k+h)\Bigr|^2 \ge H^{-1} \lambda^2 \Bigr]
\\
&\le H \lambda^{-2} \,
\mathbb E\,\Bigl[\sum_{1 \le h\le H}\,
\Bigl|\sum_{m^A\le k\le (m+1)^A}\xi(k)\bar\xi(k+h)\Bigr|^2\Bigr] \\
&\lesssim H^2\lambda^{-2}m^{A(1+b)-1}\,.
\end{align*}
This application of Lemma~\ref{lem10q} is legal since
$H$ was chosen $\lesssim m^{A(1-a)-1} \simeq \eta x^{1-a}$.
The convergence of the series
\[
\sum_m H^2\lambda^{-2}m^{A(1+b)-1}  = \sum_m m^{ A(b-b')-1} < \infty\,
\]
allows us to apply the Borel--Cantelli lemma, which shows that, almost surely,
we have
\[
\sum_{1\le h\le H}\, \Bigl|\sum_{m^A\le k\le (m+1)^A}\xi(k)\bar\xi(k+h)\Bigr|
\le \lambda\,,
\]
provided that $m$ is sufficiently large. \end{proof}

\begin{proof}[Proof of Theorem~\ref{thm:random-multiplicative}]
First, we note that
\[
\sum_{\nu \le k \le \nu + \frac12 \beta\sigma} |\xi (k)|^2 \gtrsim \beta\sigma\,.
\]
This is obvious in the completely multiplicative case, when $\xi=\xi_2$. In the random M\"obius case,
$\xi=\xi_1$, this follows from the classical estimate~\cite[Theorem~333]{HW},
which states that the number of the square-free integers in $[1, x]$ equals $\kappa x +O(\sqrt{x})$ with $\kappa=6/\pi^2$ (recall that $\beta=\sigma^{-c}$ with $c<1/6$, so $\beta\sigma \gg \nu^{1/2}$).
Thus, we need to show that, for some $p$,
\[
\sum_{h\ge 1} (1+\beta h)^{-p}\, S^*(M_1, M_2; h) \ll \beta\sigma\,,
\]
with $[M_1, M_2] = [\nu-\beta\sigma, \nu + (2+o(1))\beta\sigma]$.

Next, observe that, for $q>1$ and $p>q/(q-1)$,
\[
\sum_{h\ge \beta^{-q}} (1+\beta h)^{-p}\, S^*(M_1, M_2; h) \ll
\beta\sigma\, \sum_{h\ge \beta^{-q}} (1+\beta h)^{-p} \ll \beta\sigma\,,
\]
so our task boils down to
\[
\sum_{1 \le h \le \beta^{-q}} \,
\max_{|k-\nu|\le 3\beta\sigma}\,
\Bigl| \sum_{k\le s \le \nu +3\beta\sigma} \xi(s)\bar\xi(s+h) \Bigr|
\ll \beta\sigma\,.
\]

Given $A$, $a$, $b'$ as in Lemma~\ref{lem10q'}, let $R$ be sufficiently large, and let $m\simeq \nu^{1/A}$. We
divide the interval $[\nu-3\beta\sigma, \nu+3\beta\sigma]$
into $L\simeq \beta\sigma/m^{A-1} \simeq\beta\sigma \nu^{-(A-1)/A}$
intervals $[(m+s)^A, (m+s+1)^A]$ of length $\simeq m^{A-1}$.
Assuming that
\begin{equation}\label{per2}
\beta^{-q} \lesssim m^{A(1-a)-1}
\end{equation}
and applying Lemma~\ref{lem10q'} with $H=\beta^{-q}$,
almost surely, we have
\[
\sum_{1 \le h \le \beta^{-q}} \,
\max_{|k-\nu|\le 3\beta\sigma}\,
\Bigl| \sum_{k\le s \le \nu +3\beta\sigma} \xi(s)\bar\xi(s+h) \Bigr|
\lesssim \beta^{-q}m^{\frac12A(1+b')}\cdot L + \beta^{-q} \cdot m^{A-1}\,.
\]
Plugging in $L\simeq \beta\sigma \nu^{1/A-1}$,
$m^{A-1}\simeq \nu^{1-1/A}$, and recalling that by assumption~\eqref{eq:phi'},
$\sigma \gg \nu^{1-\ep})$, we see that the RHS is
\begin{multline*}
\lesssim \beta\sigma\, \bigl(\beta^{-q}\nu^{(1+b')/2 +1/A-1}
+ \beta^{-q-1}\nu^{\ep - 1/A} \bigr) \\
\ll \beta\sigma\, \bigl( \nu^{q(c+\ep) + (1+b')/2 +1/A -1} + \nu^{(q+1)(c+\ep) +\ep -1/A} \bigr)\,,
\end{multline*}
provided that $\beta = \sigma^{-c} \gg \nu^{-c-\ep}$.
Since we can take $q$ sufficiently close to $1$, and $a$, $b'$ and $\ep$ sufficiently small,
we conclude that the parameters $A$ and $c$ need to satisfy two conditions
\[
\begin{cases}
1/A < 1/2 - c, \\
1/A > 2c
\end{cases}
\]
(condition~\eqref{per2} boils down to $c<1-a-1/A$ and, since $a$ can be taken arbitrarily small,
is weaker than the first one). It remains to choose $A=3$ together with any $c<1/6$,
completing the proof of Theorem~\ref{thm:random-multiplicative}. 
\end{proof}

\subsection{The Golay--Rudin--Shapiro sequence}

Let $\xi$ be the Golay--Rudin--Shapiro sequence, that is,
$\xi (0)=1$, $\xi(2n)=\xi(n)$, and $\xi(2n+1)=(-1)^n\xi (n)$.

\begin{theorem}\label{thm:GRS-sequences}
Let $F_\xi (z) = \sum_{n\ge 0} \xi(n) a(n) z^n $ be an entire function with smooth
coefficients $a(n) = \exp\Bigl[ -\displaystyle \int_0^n \varphi \Bigr]$, and
with the Golay--Rudin--Shapiro sequence $\xi$. Let $\sigma=\psi'\circ\log$,
where $\psi$ is the function inverse to $\varphi$.
Then, for any $0<c<1/3$, the zero set of $F_\xi$ is $(\gamma, \rho)$-equidistributed with the radial gauge $\rho=R\sigma^{-c}$, provided that the function $\varphi$ is $\Delta$-regular
with $\Delta (s) = o(s^c)$, as $s\to\infty$.
\end{theorem}

\begin{proof}
As in the two previous instances, we will apply Theorem~\ref{thm:main1}.
We use the estimate for the binary correlations of $\xi$ due to Mauduit and
S\'ark\"ozy~\cite[Theorem~4]{MaSa}:
\[
\Bigl| \sum_{1\le s \le M} \xi(s)\xi(s+h) \Bigr| \lesssim h(1+\log M)\,,
\qquad h\ge 1\,.
\]
This immediately yields $ S^*(M_1, M_2; h) \lesssim h\log\sigma $.
Splitting the sum below into the blocks of length $\beta^{-1}2^\ell$, $\ell\ge 0$,
and taking $p>2$, we get
\[
\sum_{h\ge 1} (1+\beta h)^{-p} S^*(M_1, M_2; h)
\lesssim \log\sigma\, \sum_{\ell \ge 0} 2^{-p\ell}\, \Bigl( \frac{2^\ell}\beta \Bigr)^2
\lesssim \frac{\log\sigma}{\beta^2}\,.
\]
To satisfy condition~\eqref{eq:condition2}, we take $\beta=\sigma^{-c}$
with any $c<1/3$. Then, obviously, $\beta^{-2}\log\sigma \ll \beta\sigma$, and we are done.
\end{proof}

\section{Wiener sequences $\xi$ whose spectral measures have no gaps}
\label{sect:8}

Throughout this section we assume that
\begin{equation}\label{eq:C1}
\Bigl|\, \frac1X\, \sum_{0\le s < X} \xi(s)\bar\xi(s+h) - \widehat{\chi}(h) \Bigr|
\lesssim \ep_1(X; h), \qquad 0\le h \le H=H(X),
\end{equation}
with $X\mapsto \ep_1(X; h)$ decreasing to $0$ and $X\mapsto H(X)$ increasing to $\infty$,
as $X\to\infty$, and that
\begin{equation}\label{eq:C2}
\inf\bigl\{\chi(J)\colon
J\subset [-\tfrac12, \tfrac12\,] {\rm\ an\ interval}, |J|=\tfrac12\, \eta\bigr\}
\gtrsim \ep_2(\eta),
\end{equation}
with a positive non-decreasing function $\ep_2$.
The first condition quantifies the fact that $\chi$ is a spectral measure
of the Wiener sequence $\xi$, while the second condition is a quantitative version
of the statement that $\chi$ has no gaps in its support.

In Proposition~\ref{Lemma5a} we provide a set of conditions which will yield a lower
bound on the Weyl-type sum $W_R(\theta)$ on a sufficiently dense set of points
$(R, \theta)$. Then, we combine Proposition~\ref{Lemma5a} with Proposition~\ref{Lemma-subharm} and Proposition~\ref{LemmaL-W-V} and show (in Theorem~\ref{thm:main2}) that these
conditions guarantee equidistribution of zeroes of $F_\xi$ on appropriate local scales.
This set of conditions looks somewhat cumbersome, but then, to demonstrate how neatly it works,
we consider two instances of Wiener sequences $\xi$  with singular spectral measures
having no gaps in their support: the indicator-function of the square-free integers
and the Thue--Morse sequence.

\subsection{Another lower bound for Weyl-type sums}
\label{subsec:LB-general}

Denote by $\mathcal G$ the class of non-negative test-functions $g\in C_0^\infty (\bR)$ such that
$\operatorname{supp}(g)\subset \bigl(-\tfrac12, \tfrac12\, \bigr)$, $\displaystyle \int_\bR g = 1$, and
$g=1$ on $\bigl[-\tfrac14, \tfrac14\, \bigr]$.
As before, we assume that $\varphi$ is a $\Delta$-regular function, and
that $\nu = \psi\circ\log$, $\sigma=\psi'\circ\log$, where $\psi=\varphi^{-1}$ is the inverse
function.

\begin{proposition}\label{Lemma5a}
Let $R\gg 1$. Suppose that there exist $q>1$, $\beta=\beta(R)\ll 1/\Delta(\sigma)$,
and $g\in \mathcal G$, satisfying the following set of conditions:

\smallskip\noindent{\rm (a)}
$\beta^{-q} \le \min\bigl\{\sqrt{\sigma}, H(\tfrac12\, \nu)\bigr\}$;

\medskip\noindent{\rm (b)}
$\beta^{-(1+2q)} \ll \sigma\ep_2(\beta)$;

\medskip\noindent{\rm (c)}
$\displaystyle \nu \sum_{0\le h \le \beta^{-q}} \ep_1(\tfrac12\, \nu; h) \ll \beta\sigma\ep_2(\beta) $;

\medskip\noindent{\rm (d)}
$\displaystyle \sum_{h>\beta^{-q}} |\widehat{g}(\beta h)| \ll \ep_2(\beta) $.

\smallskip
Then, for every $\vartheta\in \bigl[-\tfrac12,  \tfrac12\, \bigr]$, there exist
\[
R'\in [R, R(1+\beta)],  \quad \theta'\in (\vartheta-\beta, \vartheta+\beta),
\]
such that
\[
|W_{R'}(\theta')| \gtrsim \sigma^{1/4}\, \sqrt{\ep_2 (\beta)}\,.
\]
\end{proposition}

\begin{proof}
As in the proof of Proposition~\ref{Lemma5}, we estimate from below the average
\[
\widetilde{X} =
\int_R^{R(1+\beta)}\, \int_{-1/2}^{1/2} |\widetilde{W}_s(\theta)|^2\,
g(\beta^{-1}(\vartheta-\theta))\, {\rm d}\theta\, {\rm d}\nu (s)\,,
\]
where
\[
\widetilde{W}_R(\theta) = \sum_{k\in\bZ} \xi(k)e(k\theta)
\exp\bigl[\,\frac{(k-\nu)^2}{2\sigma} \,\bigr]\,.
\]
By Lemma~\ref{Lemma5A}(ii), $\nu(R(1+\beta))-\nu(R) = (1+o(1))\beta\sigma$, so, to prove Proposition~\ref{Lemma5a} we need to show that
$\widetilde{X} \gtrsim \beta^2\sigma^{3/2}\ep_2(\beta)$.
As before, we rewrite $\widetilde X$ as a Fourier series
\[
\widetilde{X}
= \beta\,
\sum_{h\in\bZ} \widehat{g}(\beta h) e(-h\vartheta)\,
\sum_{k\in\bZ} \xi(k)\bar\xi(k+h) V_R(k; h)\,.
\]
Recalling the notation
$\displaystyle V_R(h) = \sum_{k\in\bZ} V_R(k; h)$, we split the RHS into three parts:
\begin{align*}
\widetilde{X}  &= \beta V_R(0)\, \sum_{h\in\bZ} \widehat{g}(\beta h)\widehat{\chi}(h) e(-h\vartheta) \\
&\quad + \beta\, \sum_{h\in\bZ} \widehat{g}(\beta h) e(-h\vartheta)\,
\sum_{k\in\bZ} \bigl( \xi(k)\bar\xi(k+h) - \widehat{\chi}(h) \bigr) V_R(k; h) \\
&\quad + \beta\, \sum_{h\in\bZ} \widehat{g}(\beta h) \widehat{\chi}(h) e(-h\vartheta)\,
(V_R(h) - V_R(0)) \\
&= I +II + III\,.
\end{align*}
We will show that $I\gtrsim \beta^2\sigma^{3/2}\ep_2(\beta)$, while
the terms $|II|$ and $|III|$ are  $\ll \beta^2\sigma^{3/2}\ep_2(\beta)$.

\medskip\noindent\underline{Lower bound of $I$}:
We set $g_\beta(\theta) = \beta^{-1}g(\beta^{-1}\theta)$, and denote by
$(\chi * g_\beta)'$ the density of the convolution of $\chi$ with
$g_\beta$. Then, \[ I=\beta V_R(0) (\chi * g_\beta)'(-\vartheta).\]
By Lemma~\ref{Lemma10a}, $V_R(0)\gtrsim \beta \sigma^{3/2}$. Since $g_\beta = \beta^{-1}$ on
$ \bigl[-\tfrac14\, \beta, \tfrac14\, \beta \bigr]$, we have
\[
(\chi * g_\beta)'(-\vartheta) \gtrsim
\chi\bigl[-\vartheta-\tfrac14\, \beta, -\vartheta + \tfrac14\, \beta \bigr]
\gtrsim \ep_2(\beta).
\]
Thus, $I \gtrsim \beta^2 \sigma^{3/2} \ep_2(\beta)$.

\medskip\noindent\underline{Upper bound of $II$}:
Recalling that $\beta\sigma \stackrel{(a)}\gg \sqrt{\sigma\log\sigma}$
and using Lemma~\ref{Lemma5B}(ii), we cut the sum in $k$, getting
\begin{align*}
|II| &\lesssim \beta\, \sum_{h\in\bZ} |\widehat{g}(\beta h)| \cdot
\bigl| \sum_{|k-\nu|\le 2\beta\sigma}
\bigl( \xi(k)\bar\xi(k+h) - \widehat{\chi}(h) \bigr) V_R(k; h)\, \bigr| + o(1) \\
&= \beta\, \bigl(\, \sum_{|h|\le \beta^{-q}} + \sum_{|h|>\beta^{-q}} \,\bigr)\,
|\widehat{g}(\beta h)| \\ &\qquad\qquad\qquad\qquad \times
\bigl|\, \sum_{|k-\nu|\le 2\beta\sigma}
\bigl( \xi(k)\bar\xi(k+h) - \widehat{\chi}(h) \bigr) V_R(k; h) \,\bigr| + o(1)\,.
\end{align*}
For $|h|>\beta^{-q}$, using the crude estimate
\[
\bigl|\, \sum_{|k-\nu|\le 2\beta\sigma}
\bigl( \xi(k)\bar\xi(k+h) - \widehat{\chi}(h) \bigr) V_R(k; h) \,\bigr|
\lesssim \beta\sigma \cdot \sqrt{\sigma} = \beta\sigma^{3/2}\,,
\]
we get
\begin{multline*}
\beta\, \sum_{|h|> \beta^{-q}}
|\widehat{g}(\beta h)| \cdot
\bigl|\, \sum_{|k-\nu|\le 2\beta\sigma}
\bigl( \xi(k)\bar\xi(k+h) - \widehat{\chi}(h) \bigr) V_R(k; h) \,\bigr|
\\
\lesssim \beta^2\sigma^{3/2} \sum_{|h|> \beta^{-q}}
|\widehat{g}(\beta h)|\,,
\end{multline*}
which is $\ll \beta^2\sigma^{3/2}\ep_2(\beta)$ by assumption (d).

Now, we consider the sum over $|h|\le\beta^{-q}$.
Applying summation by parts and using Lemma~\ref{Lemma5D}, we have
\begin{align*}
\Bigl| &\sum_{|k-\nu|\le 2\beta\sigma}
\bigl( \xi(k)\bar\xi(k+h) - \widehat{\chi}(h) \bigr) V_R(k; h) \Bigr|
\\
&\lesssim
\Bigl(\max_{|k-\nu|\le 2\beta\sigma} V_R(k; h)+ \sum_{|k-\nu|\le 2\beta\sigma} |V_R(k; h)-V_R(k-1; h)|\,\Bigr) 
\\&\qquad\qquad\qquad\qquad\qquad\times
\max_{|k-\nu|\le 2\beta}\, \bigl|\,
\sum_{k\le s \le \nu +2\beta\sigma}
\bigl(  \xi(s)\bar\xi(s+h) - \widehat{\chi}(h) \bigr)\,
\bigr|
\\
&\lesssim
\sqrt{\sigma}\, \max_{|k-\nu|\le 2\beta}\, \bigl|\,
\sum_{k\le s \le \nu +2\beta\sigma}
\bigl(  \xi(s)\bar\xi(s+h) - \widehat{\chi}(h) \bigr)
\,\bigr|\,.
\end{align*}
First, we assume that $0\le h \le \beta^{-q}$. Then, by estimate~\eqref{eq:C1},
the maximum on the RHS is $ \lesssim (\nu + 2\beta\sigma)\ep_1(\nu-2\beta\sigma; h)$.
Since $\beta\sigma \ll \sigma/\Delta(\sigma) \lesssim \nu$, the latter expression is
$ \lesssim \, \nu\ep_1(\tfrac12\, \nu; h) $. The application of estimate~\eqref{eq:C1} was legal since,
$  \beta^{-q} \stackrel{(a)}\le H(\tfrac12\, \nu) \,
\stackrel{\beta\sigma\ll\nu}\le\, H(\nu-2\beta\sigma) $.
Thus,
\begin{multline*}
\beta\, \sum_{0\le h \le \beta^{-q}}\,
|\widehat{g}(\beta h)| \cdot
\bigl|\, \sum_{|k-\nu|\le 2\beta\sigma}
\bigl( \xi(k)\bar\xi(k+h) - \widehat{\chi}(h) \bigr) V_R(k; h) \,\bigr|
\\
\lesssim
\beta\sqrt{\sigma}\cdot \nu\, \sum_{0\le h \le \beta^{-q}} \ep_1(\tfrac12\, \nu, h)\,,
\end{multline*}
which is $\ll \beta^2\sigma^{3/2} \ep_2(\beta)$ by assumption (c).

The sum over $-\beta^{-q} \le h \le -1$ needs only a minor modification.
In this case we have
\begin{align*}
\bigl|\, \sum_{k\le s \le \nu +2\beta\sigma}
\bigl(  \xi(s)\bar\xi(s+h) &- \widehat{\chi}(h) \bigr) \,\bigr|\\
&= \bigl|\, \sum_{k-|h|\le s \le \nu +2\beta\sigma-|h|}\,
\bigl(  \xi(s)\bar\xi(s+|h|) - \widehat{\chi}(|h|) \bigr) \,\bigr| \\
&= \bigl|\, \sum_{k\le s \le \nu +2\beta\sigma}\,
\bigl(  \xi(s)\bar\xi(s+|h|) - \widehat{\chi}(|h|) \bigr) \,\bigr| + O(|h|)\,.
\end{align*}
Therefore,
\begin{multline*}
\beta\, \sum_{-\beta^{-q}\le h \le -1}\,
|\widehat{g}(\beta h)| \cdot
\bigl|\, \sum_{|k-\nu|\le 2\beta\sigma}
\bigl( \xi(k)\bar\xi(k+h) - \widehat{\chi}(h) \bigr) V_R(k; h) \,\bigr|
\\
\lesssim
\beta\sqrt{\sigma}\bigl(\, \nu\, \sum_{0\le h' \le \beta^{-q}} \ep_1(\tfrac12\, \nu, h')
+ \beta^{-2q}\,\bigr)\,,
\end{multline*}
and, by assumptions (c) and (b), both terms on the RHS are $\ll \beta^2\sigma^{3/2} \ep_2(\beta)$.

\medskip\noindent\underline{Upper bound of $III$}:
We have
\[
| III | \lesssim \beta\, \bigl(\, \sum_{|h|\le\sqrt\sigma} + \sum_{|h|>\sqrt{\sigma}} \,\bigr)
|\widehat{g}(\beta h)|\cdot |(V_R(h) - V_R(0))|\,.
\]
To estimate the first sum, we apply Lemma~\ref{Lemma10b} and use that the Fourier transform of $g$ decays faster than any negative power. We get
\begin{multline*}
\beta\, \sum_{|h|\le\sqrt\sigma}
|\widehat{g}(\beta h)|\cdot |(V_R(h) - V_R(0))| \lesssim \beta\sqrt{\sigma}\,
\sum_{|h|\le\sqrt\sigma}\,  |\widehat{g}(\beta h)| (1+ h^2\beta) \\
\lesssim \beta\sqrt{\sigma} \cdot \beta^{-2} = \frac{\sqrt{\sigma}}\beta\,.
\end{multline*}
To estimate the second sum, we use the crude bound $V_R(h) \lesssim \beta\sigma^{3/2}$,
which follows from Lemma~\ref{Lemma5C}(i) and Lemma~\ref{Lemma5B}(ii).
Using again that the Fourier transform of $g$ decays faster than any negative power, we get
\begin{multline*}
\beta\, \sum_{|h|>\sqrt\sigma}
|\widehat{g}(\beta h)|\cdot |(V_R(h) - V_R(0))| \lesssim \beta^2\sigma^{3/2}\,
\sum_{|h|>\sqrt\sigma} |\widehat{g}(\beta h)| \\
\lesssim \beta^2 \sigma^{3/2} \cdot \frac1{\beta}\,
\sum_{\ell\ge\beta\sqrt\sigma} \frac1{\ell^3} \lesssim \frac{\sqrt{\sigma}}\beta\,.
\end{multline*}
It remains to recall that
condition (b) guarantees that 
$$
\beta^{-1}\sqrt{\sigma} \ll \beta^2 \sigma^{3/2}\, \ep_2(\beta).
$$
This completes the proof of Proposition~\ref{Lemma5a}.
\end{proof}

\subsection{Making the ends meet}

Now, combining Proposition~\ref{Lemma5a} with Proposition~\ref{Lemma-subharm} and Proposition~\ref{LemmaL-W-V},
we obtain

\begin{theorem}\label{thm:main2}
Let $F_\xi (z) = \displaystyle \sum_{n\ge 0} \xi(n)a(n) z^n$ be an entire function with smooth
coefficients $a(n) = \displaystyle \exp\Bigl[ -\int_0^n \varphi \Bigr]$ with a $\Delta$-regular
function $\varphi$. Let $\sigma = \psi'\circ\log$, where $\psi=\varphi^{-1}$ is the inverse to
$\varphi$. Let $\beta=\beta(R) \ll 1/\Delta(\sigma)$ be a small parameter satisfying
\[
\beta' (R) = o(1/R), \quad R\to\infty\,.
\]
Suppose that $\xi$ is a bounded sequence, for which
conditions~\eqref{eq:C1} and~\eqref{eq:C2} hold with functions $\ep_1$ and $\ep_2$
satisfying assumptions (a), (b), (c), and (d) in Proposition~\ref{Lemma5a}. Suppose, in addition,
that
\begin{equation}\label{eq:additional}
\sigma^{-1/4}\Delta(\sigma)(\log\sigma)^{3/2} \ll \sqrt{\ep_2(\beta)}\,.
\end{equation}
Then the zero set of $F_\xi$ is $(\gamma, \rho)$-equidistributed with the radial gauge $\rho = R\beta$.
\end{theorem}

We skip the proof this theorem, which is rather straightforward and close to the proof
of Theorem~\ref{thm:main1}. The only difference is that now, instead of Proposition~\ref{Lemma5},
we will use Proposition~\ref{Lemma5a}. We mention that the purpose of the additional
restriction~\eqref{eq:additional} is to guarantee that the lower bound on the Weyl-type sum
$W_R$, provided by Proposition~\ref{Lemma5a}, can be combined with the approximation error in Proposition~\ref{LemmaL-W-V}.

\subsection{The indicator-function of square-free integers}

Here, we consider $\xi (n) = \mu^2(n)$, where $\mu$ is the M\"obius function.
The main result of this section is Theorem~\ref{thm:sq-free} below.
The key ingredient in its proof is Mirsky's classical estimate for binary correlations.
Set
\[
D = \prod_p \Bigl( 1 - \frac2{p^2} \Bigr)\,.
\]

\begin{lemma}[Mirsky~\cite{Mirsky}]\label{Lemma:sq-free}
For $\varepsilon>0$, we have
$$
\Bigl|\sum_{0\le k \le x} \mu^2(k) \mu^2(k+h) - D(h) x\Bigr| \le C_\varepsilon x^{2/3+\varepsilon},\qquad 0\le h\le x,
$$
where $D(0)=6\pi^{-2}$, and for $h\ne 0$,
$$
D(h) = D\, \prod_{p^2 | h} \Bigl(1+\frac1{p^2-2} \Bigr)\,.
$$
\end{lemma}

The only difference with Mirsky's result is he did not specify the dependence of
the constant on the shift $h$. For the reader's convenience, we give the
proof in Appendix~C. We will follow Mirsky's work very closely.

Having Lemma~\ref{Lemma:sq-free} at hand, it is not difficult to compute the spectral measure $\chi$
of the sequence $\mu^2$.

\begin{lemma}\label{lemma:sq-free-spectral1}
The spectral measure of the sequence $\mu^2$ equals
$$
\chi=D\,\sum_{\mu^2(d)=1}\Bigl(\frac1{d^2}\prod_{p | d} \frac1{p^2-2} \Bigr)\sum_{j=0}^{d^2-1}\delta_{e(j/d^2)}.
$$
\end{lemma}

\noindent Since the proof is only a few lines, we give it here:

\begin{proof}
We need to check that $\widehat{\chi}(h)=D(h)$. For $h=0$ this is obvious. For $h\ge 1$, we have
$$
\widehat{\chi}(h)=D\, \sum_{\mu^2(d)=1}\Bigl(\frac1{d^2}\prod_{p | d} \frac1{p^2-2} \Bigr)\sum_{j=0}^{d^2-1}e(jh/d^2).
$$
Since
$$
\sum_{j=0}^{m-1}e(jk/m)=\begin{cases}
m,\qquad m\,|\,k,\\
0, \qquad m{\not|}\,k,
\end{cases}
$$
we obtain that
$$
\widehat{\chi}(h)=D\,
\sum_{\mu^2(d)=1,\,d^2|h}\ \prod_{p | d} \frac1{p^2-2} =
D\, \prod_{p^2|h} \Bigl(1+\frac1{p^2-2} \Bigr) = D(h),
$$
completing the proof. 
\end{proof}

The next lemma tells us how thin the measure $\chi$ can be, i.e, how
estimate~\eqref{eq:C2} looks in this case:

\begin{lemma}\label{lemma:sq-free-spectral2}
For any interval $I\subset \bigl[ -\tfrac12, \tfrac12]$,
we have $\chi (I)\gtrsim |I|^{3/2}$.
\end{lemma}

\begin{proof}
We have
\begin{multline*}
\chi(I)\gtrsim \sum_{\substack{\mu^2(d)=1,\\ d^2|I|>1}}\, \Bigl(\frac1{d^2}\prod_{p | d} \frac1{p^2-2}
\Bigr)d^2|I| \\
=|I|\sum_{\substack{\mu^2(d)=1,\\ d^2|I|>1}}\, \prod_{p | d} \frac1{p^2-2}
\gtrsim |I|\, \sum_{\substack{\mu^2(d)=1,\\ d^2|I|>1}}\, \frac1{d^2}
\gtrsim |I|\, \sum_{\ell>|I|^{-1}}\, \frac1{\ell^2} \gtrsim |I|^{3/2},
\end{multline*}
where in the inequality before last, we used that the square-free numbers have positive density.
\end{proof}

Combining Theorem~\ref{thm:main2} with Lemma~\ref{Lemma:sq-free} and Lemma~\ref{lemma:sq-free-spectral2}, we arrive at

\begin{theorem}\label{thm:sq-free}
Let
\[
F(z) = \sum_{\mu^2(n)=1} a(n) z^n
\]
be an entire function with smooth coefficients
$ a(n) = \displaystyle \exp\Bigl[ - \int_0^n \varphi \Bigr] $ with a $\Delta$-regular function
$\varphi$, such that, for every $\ep>0$, $\varphi'(t) = o(t^{-1+\ep})$ as $t\to\infty$.
Let $\sigma = \psi' \circ \log$, where $\psi=\varphi^{-1}$ is the inverse function to $\varphi$.
Then the zero set of $F$ is $(\gamma, \rho)$-equidistributed with $\rho = R\sigma^{-c}$, provided that $0<c<\tfrac2{21}$ and $\Delta (s) = o(s^c)$ as $s\to\infty$.
\end{theorem}

\begin{proof}
By Lemma~\ref{Lemma:sq-free}, condition~\eqref{eq:C1} holds with
$\ep_1 (X) = X^{-\frac13+\ep}$ and $H(X)=X$.
By Lemma~\ref{lemma:sq-free-spectral2},
condition~\eqref{eq:C2} holds with $\ep_2 (\beta) = \beta^{\frac32}$.
We take $\beta=\sigma^{-c}$ and verify that, for $c<\tfrac2{21}$,
the assumptions of Theorem~\ref{thm:main2} hold, provided that $q>1$ is
chosen sufficiently close to $1$.

The verification is quite straightforward.
Since $H(\tfrac12\, \nu) = \tfrac12\, \nu \gtrsim \sigma/\Delta(\sigma) \gg \sigma^{1-c}$,
assumption (a) boils down to $cq < \min\bigl( \tfrac12, 1-c\bigr)$, that is, to $c<\tfrac12$.
Assumption (b) holds for $c(1+2q) < 1 - \tfrac32\, c$, that is, for $c<\tfrac29$.

Assumption (c) is true when $\nu^{\frac23+\ep} \ll \beta^{1+\frac32+q}\sigma =
\sigma^{1-(\frac52+q)c}$. Since we are assuming that $\nu = o(\sigma^{1+\ep})$,
this boils down to $\tfrac72\, c < \tfrac13$, that is, to $c<\tfrac2{21}$.

Since the Fourier transform $\widehat{g}$ decays faster than
any negative power, assumption (d) holds for any choice of $c>0$.
At last, to satisfy condition~\eqref{eq:additional}, we need $-\tfrac14 + c < - \tfrac34\, c$,
i.e., $c<\tfrac1{7}$.
\end{proof}

Likely, using more advanced analytic number theory techniques, one can improve
the exponent $\tfrac2{21}$.

\subsection{The Thue--Morse sequences}
The Thue--Morse sequence is defined in an inductive way by the relations $\xi(0)=1$, $\xi(2n)=\xi(n)$, $\xi(2n+1)=-\xi(n)$, $n\ge 0$. The Thue--Morse sequence is a Wiener sequence with
purely singular continuous spectral measure. This fact goes back to Mahler~\cite{Mahler}. In that work Mahler proved that 
the Thue--Morse sequence is a Wiener sequence, computed its spectral measure, and
proved that it has no discrete component and has a non-trivial singular continuous component. The fact that the spectral measure is purely absolutely continuous was proven later by Kakutani~\cite{Kakutani}.
Curiously, Mahler published his result in 1927, as a follow-up to Winer's celebrated work~\cite{Wiener}, in which Wiener introduced the class of sequences, which today bears his name.

\begin{lemma}[Mahler~\cite{Mahler}]\label{Lemma:TM1}
Let $\xi$ be the Thue--Morse sequence. Then,
$$
\Bigl|\sum_{0\le k < x} \xi(k) \xi(k+h) - \sigma(h) x\Bigr| \le Ch \log(x+1),\qquad 0\le h< x,
$$
where the even sequence $\sigma\colon \bZ \to [-1, 1]$ is defined by the recurrence relations
$\sigma(0)=1$, and $\sigma(2h)=\xi(h)$,
$\sigma(2h+1)=-\frac12(\sigma(h)+\sigma(h+1))$.
\end{lemma}

Our formulation is slightly different from the original one, since Mahler did not specify
the rate of convergence of binary correlations. We give the proof, which follows
Mahler's one mutatis mutandis, in Appendix~D.

Let $\chi$ be the spectral measure of the Thue--Morse sequence, that is, $\widehat{\chi}(h)=\sigma(h)$, $h\in\bZ$. The following lemma is probably well-known to experts.
Its proof exploits the identity for the generating function of $\xi$:
\[
\sum_{n\ge 0} \xi(n) z^n  = \prod_{\ell\ge 0} \Bigl( 1 - z^{2^\ell} \Bigr)\,.
\]

\begin{lemma}\label{lemma:TM2}
For any interval $I\subset [-\tfrac12, \tfrac12\,]$,
\begin{equation}
\chi (I)\gtrsim \exp(-c(\log|I|^2)).
\label{ytr1}
\end{equation}
\end{lemma}

\begin{proof}
Given $N\ge 1$, we define
\begin{align*}
P_N(t)&=\sum_{0\le k<N}\xi(k)e(kt),\\
{\rm d}\chi_N(t)&=\frac1N|P_N(t)|^2\,{\rm d}t.
\end{align*}
Then $\chi_N([0,1])=1$ and for $h\in\mathbb Z$ we have
$$
\widehat{\chi_N}(h)=\frac1N\, \sum_{0\le k<N}\xi(k)\xi(k+h)\to\sigma(h),\qquad N\to\infty.
$$
Hence, the measures $\chi_N$ tend to $\chi$ weakly, and to verify \eqref{ytr1}, it suffices to check that for any interval
$I=[p2^{-m},(p+1)2^{-m}]$ with integer $p$ we have
$$
\chi_{2^n}(I)\ge \exp(-Cm^2),\qquad n>m.
$$

Since $P_1(t)=1$ and
\begin{multline*}
P_{2^{n+1}}(t)=\sum_{0\le k<2^{n+1}}\xi(k)e(kt)\\=\sum_{0\le k<2^n}\xi(2k)e(2kt)+\sum_{0\le k<2^n}\xi(2k+1)e((2k+1)t)\\=
\sum_{0\le k<2^n}\xi(k)e(2kt)-\sum_{0\le k<2^n}\xi(k)e((2k+1)t)=(1-e(t))P_{2^n}(2t),
\end{multline*}
we obtain that
$$
P_{2^n}(t)=\prod_{0\le j<n}\bigl(1-e(2^jt)\bigr).
$$
Therefore,
$$
{\rm d}\chi_{2^n}(t)=\Bigl( \prod_{0\le j<n}\bigl(2\sin^2(2^j\pi t)\bigr)\Bigr) \,{\rm d}t.
$$
Set $I'=[(p+\frac14)2^{-m},(p+\frac34)2^{-m}]$. We have
$$
|\sin(2^j\pi t)|\ge c2^{j-m},\qquad 0\le j\le m,\,t\in I'.
$$
Hence,
$$
\prod_{0\le j\le m}\bigl(2\sin^2(2^j\pi t)\bigr)\ge 2^{-m^2-Cm},\qquad t\in I'.
$$
Furthermore, for $n>m$, the measure
$$
d\gamma_{m,n}(t)=\Bigl( \prod_{m<j<n}\bigl(2\sin^2(2^j\pi t)\bigr)\Bigr)\,{\rm d}t
$$
is $2^{-(m+1)}$-periodic, and $\gamma_{m,n}([0,1])=1$. Hence, $\gamma_{m,n}(J)=2^{-(m+1)}$ for any interval $J$ of length $2^{-(m+1)}$.
In particular, $\gamma_{m,n}(I')=2^{-(m+1)}$, and, finally,
$$
\chi_{2^n}(I')\ge 2^{-m^2-Cm},\qquad n>m.
$$
This gives us \eqref{ytr1}.
\end{proof}

\begin{theorem}\label{thm:TM}
Let $\xi$ be the Thue--Morse sequence, and let 
$$ 
F(z) = \displaystyle \sum_{n\ge 0} \xi(n) a(n) z^n 
$$
be an entire function with smooth coefficients
$ a(n) = \displaystyle \exp\Bigl[ - \int_0^n \varphi \Bigr] $ with a $\Delta$-regular function
$\varphi$ such that, for some $C>1$,
$ \varphi'(t) \lesssim (\log t)^{-C} $.
Let $\sigma = \psi' \circ \log$, where $\psi=\varphi^{-1}$ is the inverse function to $\varphi$.
Then the zero set of $F$ is $(\gamma, \rho)$-equidistributed with $\rho = R e^{-c\sqrt{\log\sigma}}$,
provided that the constant $c$ is sufficiently small and that $\Delta (s) =
e^{c_1\sqrt{\log\sigma}}$ with $c_1<c$.
\end{theorem}

\begin{proof}
The proof will be a straightforward inspection of assumptions of Theorem~\ref{thm:main2}.
By Lemma~\ref{Lemma:TM1}, $\ep_1(X; h) = (h\log X)/X$, $H(X)=X$, and, by Lemma~\ref{lemma:TM2},
$\ep_2(\beta) = e^{-c_0|\log\beta|^2}$. We take $q=2$, and choose $\beta$ so that
$\ep_2(\beta) = \sigma^{-\kappa}$ with some $\kappa\in (0, 1)$,
i.e., $|\log\beta|\simeq \sqrt{\log\sigma}$. The $\Delta$-regularity condition on $\varphi$ yields that $\beta'=o(1/R)$. 

Since $\beta$ in any negative power grows much slower than $\sigma$ in any positive power,
condition (a) implies no restriction, while condition (b) requires that $\kappa<1$.
Condition (c) is met provided that $\log\nu \ll \beta^5 \sigma^{1-\kappa}$.
Since $\varphi'(t)\lesssim (\log t)^{-C}$, we have $\log\nu \lesssim \sigma^{1/C}$, that is,
condition (c) boils down to $\kappa < 1-\tfrac1C$.

To satisfy condition (d), we choose the function $g\in\mathcal G$ so that
$\widehat{g}(\lambda) \lesssim e^{-\sqrt{|\lambda|}}$.
Then,
\[
\sum_{h>\beta^{-2}} |\widehat{g}(\beta h)|
\lesssim \frac1{\beta}\, \int_{1/\beta}^\infty e^{-\sqrt{|\lambda|}}\, {\rm d}\lambda
\lesssim \bigl( \frac1\beta \bigr)^{3/2} e^{-\beta^{-1/2}}\,,
\]
which is much smaller than $\ep_2(\beta)=e^{-c|\log\beta|^2}$.

Finally, it is easy to see that condition~\eqref{eq:additional} holds for any $\kappa<\tfrac12$.
Hence, choosing $\kappa<\min(\tfrac12, 1-\tfrac1C)$, we complete the proof.
\end{proof}

\section*{Appendix A: $(\gamma, \rho)$-equidistribution and the uniform transportation}
\renewcommand{\theequation}{A.\arabic{equation}}
\setcounter{equation}{0}
\renewcommand{\thelemma}{A.\arabic{lemma}}
\setcounter{lemma}{0}

The idea to measure the proximity between $n_{F_\xi}$ and $\gamma$ by
the uniform transportation distance was used in Sodin--Tsirelson~\cite{ST2} in the case
$a(n) = (n!)^{-1/2}$ when $\xi$ is a complex Gaussian IID sequence
(see also~\cite{ST-Transp}). In a somewhat different set-up, a
similar idea was used by Sj\"ostrand, see~\cite[Ch~12]{Sjostrand} and references therein.

\subsection*{A.1 The uniform transportation distance and its dual version}

Let $\gamma$ and $\gamma_1$ be locally finite Borel measures on $\bC$. We call a non-negative locally finite measure $\mathfrak n$ on $\bC\times\bC$ {\em a transportation} from $\gamma_1$ to $\gamma$, if $\mathfrak n$ has marginals $\gamma_1$ and $\gamma$, that is,
\[
\iint_{\bC \times\bC} h(x)\, {\rm d}\mathfrak n (x, y)
= \int_{\bC} h(x)\, {\rm d}\gamma_1 (x)\,,
\]
and
\[
\iint_{\bC \times\bC} h(y)\, {\rm d}\mathfrak n (x, y)
= \int_{\bC} h(y)\, {\rm d}\gamma (y)\,,
\]
for all compactly supported continuous functions $h\colon\bC\to \bC$.
Note that if there exists a map $T\colon \bC\to\bC$ that pushes forward the
measure $\gamma_1$ to $\gamma$, then the corresponding transportation $\mathfrak n$ is defined
by
\[
\iint_{\bC\times\bC} H(x, y)\, {\rm d}\mathfrak n(x, y)
= \int_{\bC} H(x, Tx)\, {\rm d}\gamma_1(x)
\]
for an arbitrary compactly supported continuous function
$H\colon \bC\times\bC\to \bC$.

The better $\mathfrak n$ is concentrated near the diagonal of $\bC \times\bC $, the closer
the measures $\gamma_1$ and $\gamma$ are to each other. We shall measure such a concentration in the $L^\infty$-sense, and set
\[
\mathsf{Tra}_{\,d}(\gamma_1, \gamma) =
\inf\, \sup\{{d}(x, y)\colon x, y \in \operatorname{supp}(\mathfrak n)\}\,,
\]
where ${d}\colon \bC\times\bC\to \bR_+$ is a distance function on $\bC$, and the infimum
is taken over all transportations $\mathfrak n$ from $\gamma_1$ to $\gamma$.
Note that the distance $\mathsf{Tra}_{\,d}$ might be infinite.

There exists a dual version of the transportation distance,
which measures the discrepancy between the measures $\gamma_1$ and $\gamma$.
The distance $\mathsf{Di}_{\, d}(\gamma_1, \gamma)$ is defined to be the infimum of $\tau>0$ such that, for each bounded Borel set $U\subset\bC$,
\begin{equation}\label{eq:Di}
\gamma_1(U)\le \gamma(U_{+\tau})  \quad {\rm and} \quad \gamma(U)\le \gamma_1(U_{+\tau}),
\end{equation}
where $U_{+\tau} = \{z\in\bC\colon d(z, U)< \tau\}$ is a $\tau$-neighbourhood of $U$.
The equality $ \mathsf{Tra} = \mathsf{Di} $ is classical (Strassen, Sudakov, Laczkovich);
its proof can be found, for instance, in~\cite[Appendix~A-1]{ST-Transp}.

\subsection*{A.2 Equivalence of two notions}
Clearly, for any Borel set $U\subset\bC$, we have
$|\gamma_1 (U) - \gamma_2(U)| \le 2 \gamma_1 ((\partial U)_{+2\tau}) $
with $\tau = 2 \mathsf{Di}_{\, d}(\gamma_1, \gamma)$. Thus, the measures
$n_{F_\xi}$ and $\gamma$ are $(\gamma, \rho)$-equidistributed provided that
$\mathsf{Di}_{\, d_\rho}(n_{F_\xi}, \gamma)<\infty$. The converse is less obvious:

\begin{lemma}\label{LemmaF}
Suppose that we are given the measure $\gamma = \Gamma m$, where $m$ is the area measure, and the radial gauge $\rho$ satisfies estimates~\eqref{eq-1}
and~\eqref{eq-2} in Proposition~\ref{Lemma-subharm}, and let $\gamma_1$ be a locally finite Borel measure on $\bC$.
Suppose that there exist positive constants $\tau$ and $C$ such that, for any compact
set $K\subset\bC$, 
\begin{equation}
|\gamma_1 (K)-\gamma(K)|\le C \gamma((\partial K)_{+\tau}). 
\label{dop5}
\end{equation}
Then
$\mathsf{Di}_{\, d_\rho}(\gamma_1, \gamma)<\infty$.
\end{lemma}

\subsubsection*{A.2.1 Proof of estimate~\eqref{eq:Di} for compact sets which are far from the origin}

\begin{lemma}\label{LemmaE} Under conditions of Lemma~\ref{LemmaF},
there exist positive values $r_1$ and $\tau_1$ such that, for any compact set
$K\subset\{|z|\ge r_1\}$,
\[
\gamma(K) \le \gamma_1(K_{+\tau_1}) \quad {\rm and} \quad
\gamma_1(K) \le \gamma(K_{+\tau_1})
\]
\end{lemma}

The idea of the proof of this lemma is borrowed from Laczkovich~\cite{Laczk} (see also~\cite[Lemma~2.1]{ST-Transp}).

\begin{proof}[Proof of Lemma~\ref{LemmaE}]
Set $R_1=1$, $R_{j+1} = R_j + M\rho(R_{j})$, where $M$ is a large parameter to be fixed. Clearly,
$R_j \uparrow \infty$ (otherwise, $\lim_j \rho(R_j) = \rho(\lim_j R_j) >0$,
which would lead us to a contradiction). Set $A_j = \{z\colon R_j\le |z|\le R_{j+1}\}$. We partition the annulus $A_j$ into 
equal closed sectors of size $\simeq M\rho(R_j)$, and denote these sectors
by $Q=Q_{jk}=\{z\colon R_j \le |z| \le R_{j+1}, \theta_k \le \arg(z) \le \theta_{k+1}\}$.
We denote the ``center'' of the sector $Q$ by $w(Q)$,
$w(Q) = \sqrt{R_jR_{j+1}}\, e^{{\rm i} (\theta_k + \theta_{k+1})/2}$.
We say that the sectors $Q$ and $Q'$ are {\em neighbours} if $Q\cap Q'\ne \emptyset$,
we denote this relation by $Q\sim Q'$, and set $\widetilde{Q} =  \bigcup_{Q'\sim Q} Q'$.

Next, we note that, since $\rho'(r)\to 0$ for $r\to\infty$, we have
\begin{equation}\label{eq:E1}
\rho(R_{j+1})/\rho(R_j) \to 1, \quad j\to\infty\,,
\end{equation}
and that, by our assumption~\eqref{eq-2}, we have
\begin{equation}\label{eq:E2}
\Gamma (R)/\Gamma (R_j) \to 1, \quad R_{j-1}\le R \le R_{j+1}, \ j\to\infty\,.
\end{equation}
Furthermore, by~\eqref{eq:E1} and \eqref{eq:E2}, for $j\ge j_0$ and for any sector $Q\subset A_j$,
the hollowing holds:
\begin{itemize}
\item there are at most $9$ neighbouring sectors $Q'\sim Q$;
\item ${\rm diam}_\rho (\widetilde Q) \le 10 M$ (as usual, ${\rm diam}_\rho(X) =
\sup\{d_\rho(z_1, z_2)\colon z_1, z_2\in X\}$);
\item $\tfrac12 \le \Gamma(z_1)/\Gamma (z_2) \le 2$, $z_1, z_2\in \widetilde Q_{+\tau}$, where $\tau$ is the constant from \eqref{dop5}.
\end{itemize}

Let $r_1=R_{j_0+1}$,  
and let $K\subset \{|z|\ge r_1\}$ be a compact set. Let
\[
A=\bigcup_{Q\cap K \ne\emptyset} Q, \quad B=\bigcup_{Q\cap K \ne\emptyset} \widetilde Q,
\]
and let ${\mathsf b} A = \{Q\subset A\colon \exists Q'\sim Q, Q'\subset B\setminus A\}$ be a collection of ``boundary sectors'' in $A$. Clearly, $K\subset A \subset B \subset K_{+10M}$. 
By \eqref{dop5} we have
\[
\gamma (K) \le \gamma (A) \le \gamma_1(A) + C \gamma \bigl( (\partial A)_{+\tau} \bigr),
\]
and
\[
\gamma_1 (K) \le \gamma_1 (A) \le \gamma(A) + C \gamma \bigl( (\partial A)_{+\tau} \bigr).
\]
Furthermore, since $A$ is a union of squares $Q$, we have
\[
(\partial A)_{+\tau} \subset \bigcup_{Q\in {\mathsf b}A} (\partial Q)_{+\tau}\,,
\]
and hence,
\[
\gamma ((\partial A)_{+\tau}) \le \sum_{Q\in {\mathsf b}A} \gamma \bigl( (\partial Q)_{+\tau} \bigr).
\]

For each boundary sector $Q$, we have
\[
\gamma \bigl( (\partial Q)_{+\tau} \bigr) \lesssim M\rho(w)^2\Gamma(w),
\quad w=w(Q),
\]
while for the sectors $Q'\sim Q$, $Q'\subset B\setminus A$, we have
\[
\gamma (Q') \simeq M^2 \rho(w)^2 \Gamma (w),
\]
and (again, by \eqref{dop5})
\begin{multline}
\label{dop6}
\gamma_1(Q') \ge \gamma(Q') - C\gamma\bigl( (\partial Q')_{+\tau} \bigr) \\
\gtrsim M^2\rho(w)^2\Gamma (w) - C_1 M\rho(w)^2\Gamma (w) = (M-C_1)M\rho(w)^2\Gamma (w).
\end{multline}
Recalling that each sector $Q'\subset B\setminus A$ has at most $9$ neighbouring squares $Q\subset A$, we conclude that the error term $C\gamma \bigl( (\partial A)_{+\tau} \bigr)$ is much smaller than
$\gamma (B\setminus A)$, as well as $\gamma_1(B\setminus A)$, provided that the constant $M$ is chosen to be much bigger than the constant $C_1$ in \eqref{dop6}. Thence,
\[
\gamma (K) \le \gamma_1(A) + \gamma_1(B\setminus A) = \gamma_1(B) \le \gamma_1(K_{+10M}),
\]
and similarly, $\gamma_1 (K) \le \gamma (K_{+10M})$, proving Lemma~\ref{LemmaE} with $\tau_1=10M$. 
\end{proof}

\subsubsection*{A.2.2 Completing the proof of Lemma~\ref{LemmaF}}

Since we deal with Borel measures in $\bC$, it suffices to verify that
conditions~\eqref{eq:Di} hold for arbitrary compact set $K\subset \bC$.
By Lemma~\ref{LemmaE}, they hold for any compact set $K$,
which is sufficiently far from the origin, and it remains to get rid of the
latter hurdle.

Let $r_1$, $\tau_1$ be the positive parameters from Lemma~\ref{LemmaE}, and let
\[
m = \max( \gamma(r_1\bD), \gamma_1 (r_1\bD) ).
\]
Choose $r_1'\ge r_1$ in such a way that, for any $z$ with
$|z|\ge r_1'$, we have $\min(\gamma(D_z), \gamma_1(D_z)) \ge m $.
Set $R_1=r_1'$, $R_{j+1}=R_j+6\tau_1\rho(R_j)$, $R_j' = R_j+3\tau_1\rho(R_j)$, and
$A_j = \{R_j \le |z| \le R_{j+1}\}$.

Consider the disks
$D_j$ centered at $R_j'$ of radius $\tau_1 \rho(R_j')$, and choose
$j_0$ so large that, for $j\ge j_0$,
\begin{itemize}
\item
$D_j \subset A_j$,
\item and moreover,
$(A_{j-1})_{+\tau_1} \bigcap D_j = \emptyset$ and $(A_{j+1})_{+\tau_1} \bigcap D_j = \emptyset$.
\end{itemize}
Then, choose a sufficiently large $d_0$ so that
\begin{itemize}
\item for every $z$ with $|z|\le r_1$,
$D_{j_0}\subset \{z\}_{+d_0}$,
\item for every $j\ge j_0$, ${\rm diam}_\rho (D_j \cup D_{j+1}) < d_0$.
\end{itemize}

Now, let $K\subset\bC$ be a compact set, and let $K_1 = K\cap \{|z|\le r_1\}$,
$K_2 = K\cap \{|z|\ge r_1\}$. Clearly, $\gamma (\{|z|=r_1\}) = 0$, and we can always assume that
$\gamma_1 (\{|z|=r_1\}) = 0$ (otherwise, we slightly increase the value $r_1$). Then,
$\gamma(K)=\gamma(K_1)+\gamma(K_2)$, and the same holds for $\gamma_1$.

The rest is clear. We apply Lemma~\ref{LemmaE} to the part of the mass, which lies
in $\{|z|\ge r_1\}$, moving it, at most, by $\tau_1$, and move the mass from $\{|z|\le r_1\}$
to $D_{j_0}$ (and then, if needed, from $D_{j_0}$ to $D_{j_0+1}$, from $D_{j_0+1}$ to
$D_{j_0+2}$, \ldots , from $D_{j_1}$ to $D_{j_1+1}$, with $j_1=j_1(K)$) transporting it,
at most, by $d_0$.

More formally, if $K_1=\emptyset$, then we just use Lemma~\ref{LemmaE}; if $K_1\not=\emptyset$ and 
$(K_2)_{+\tau_1}\cap\bar D_{j_0}=\emptyset$, then
\begin{align*}
\gamma (K) &= \gamma(K_1) + \gamma (K_2) \\
&\le m + \gamma(K_2) \\
&\le m + \gamma_1((K_2)_{+\tau_1}) \qquad ({\rm by\ Lemma~\ref{LemmaE}}) \\
&\le \gamma_1(D_{j_0})  + \gamma_1((K_2)_{+\tau_1}) \\
&= \gamma_1 (D_{j_0} \cup (K_2)_{+\tau_1}) \\
&\le \gamma_1 (K_{+(d_0+\tau_1)}),
\end{align*}
and similarly, $\gamma_1(K)\le \gamma (K_{+(d_0+\tau_1)})$.

If $(K_2)_{+\tau_1}\cap\bar D_{j_0} \ne \emptyset$, we
choose $j_1 = j_1(K) \ge j_0$ so that
$(K_2)_{+\tau_1}\cap \bar D_j \ne \emptyset$ for $j_0\le j \le j_1$, while
$(K_2)_{+\tau_1}\cap \bar D_{j_1+1} = \emptyset$. Then, arguing as above,
we get
\begin{align*}
\gamma (K) &\le m + \gamma_1((K_2)_{+\tau_1}) \\
&\le \gamma_1(D_{j_1+1}) + \gamma_1((K_2)_{+\tau_1})  \\
&=\gamma_1(D_{j_1+1} \cup (K_2)_{+\tau_1}) \\
&\le\gamma_1 (K_{+(d_0+\tau_1)}),
\end{align*}
together with a similar bound for $\gamma_1(K)$. This completes the proof of Lemma~\ref{LemmaF}.
\hfill $\Box$

\section*{Appendix B: Proof of Proposition~\ref{lemma-Nguyen-Vu}}
\renewcommand{\theequation}{B.\arabic{equation}}
\setcounter{equation}{0}
\renewcommand{\thelemma}{B.\arabic{lemma}}
\setcounter{lemma}{0}

As we have already mentioned, here, we will closely follow
the original proof given by Nguyen and Vu. They work not with the sum of
exponentials, as we do, but with the sum of cosines
$ \sum_{\lambda\in\Lambda} \xi_\lambda c_\lambda \cos(2\pi \lambda\theta)$.
Another difference is that we assume that the independent random variables $(\xi_\lambda)$
are identically distributed, while Nguyen and Vu only assume their independence, but
additionally assume that, for some $m>2$, the $m$th central moment of the $\xi_\lambda$s is uniformly bounded.

The proof is carried out in three steps. First, one considers the case
when $(\xi_\lambda)$ are Radema\-cher random variables $(r_\lambda)$, that is,
independent random variables taking the values
$\pm 1$ with equal probability $1/2$. This step is based on a variant of
Hal\'asz' anti-con\-cent\-ration result~\cite[Lemma~9.3]{NV} (we will recall it
below). In the second step, applying Kahane's reduction principle, one extends the result
to independent symmetric complex-valued random variables $(\xi_\lambda)$.
Finally, the general case will be reduced to the symmetric one using the symmetrization device, that is,
taking independent copies $\xi'_\lambda$ of $\xi_\lambda$ and applying the result of the second
step to the symmetric random variables $\xi_\lambda-\xi'_\lambda$.

\medskip 
Fix an integer $ \alpha \ge 1$. Assume that
$\min_{\lambda\in\Lambda} |c_\lambda| =1$.

\medskip\noindent{\em Step~1}.
{\em Assuming that $(r_\lambda)$ are Rademacher random variables, and letting
$S(\theta) = \sum_{\lambda\in\Lambda} r_\lambda c_\lambda e(\lambda\theta)$,
we show that, for every $\theta \in [-\tfrac12, \tfrac12]$, except for a set of Lebesgue measure
at most $O(n^{-\beta/(4\alpha) + 2\alpha + 1/4})$, we have}
$$ 
\sup_{Z\in\bC}\, \bP\bigl[\bigl| S(\theta) - Z\bigr| < n^{-\beta} \bigr] \lesssim n^{-\alpha}\,.
$$ 

Let $c_\lambda = c_\lambda' + {\rm i} c_\lambda''$, $\min_{\lambda\in\Lambda}
(|c_\lambda'|^2 +|c_\lambda''|^2) = 1$,
and let $c(\theta) = \cos(2\pi\theta)$, $s(\theta) = \sin (2\pi\theta)$, i.e.,
$e(\theta) = c(\theta) + {\rm i} s(\theta)$. Set $a_\lambda
= {\rm Re}\, \bigl[ c_\lambda e(\lambda\theta) \bigr] =
c_\lambda' c(\lambda\theta) -
c_{\lambda}'' s(\lambda\theta)$.
Since, for every $Z\in\bC$,
\[
\bP\bigl[\bigl| S(\theta) - Z \bigr| < n^{-\beta} \bigr]
\le
\bP\bigl[\bigl| {\rm Re\,} S(\theta) - {\rm Re\,} Z \bigr| < n^{-\beta} \bigr]\,,
\]
it suffices to estimate the size of the set of $\theta$s, for which
\begin{equation}\label{eq-NV2}
\sup_{X\in\bR}\, \bP\Bigl[\Bigl|
\sum_{\lambda\in\Lambda} r_\lambda a_\lambda  -  X \bigr| < n^{-\beta} \bigr]
\lesssim n^{-\alpha}\,.
\end{equation}
We will use a Hal\'asz-type lemma borrowed from~\cite[Lemma~9.3]{NV}:

\begin{lemma}\label{lemma-Halasz}
Let $(r_j)_{1\le j \le n}$ be Rademacher's random variables, and let $(b_j)_{1\le j \le n}\subset\bR$. Let $\alpha\in\bN$. Suppose that there exists $b>0$ such that,
for any two distinct sets $\{i_1, \ldots , i_{\alpha'}\}$ and $\{j_1,  \ldots j_{\alpha''}\}$, $\alpha'+\alpha''\le 2\alpha$, we have
\[
\Bigl|\, \sum_{t=1}^{\alpha'} b_{i_t} - \sum_{t=1}^{\alpha''} b_{j_t} \,\Bigr| \ge b\,.
\]
Then
\[
\sup_{X\in\bR}\,
\bP \Bigl[\, \Bigl| \sum_{j=1}^n r_j b_j - X \Bigr| < b n^{-\alpha}\, \Bigr]
\lesssim n^{-\alpha}\,.
\]
\end{lemma}

We call the value $\theta\in [-\tfrac12, \tfrac12]$ {\em normal} if, for every two distinct
sets $\{\lambda_1, \ldots , \lambda_{\alpha'}\}$, $\{\mu_1, \ldots , \mu_{\alpha''}\}$
in $\Lambda$ with $\alpha'+\alpha'' \le 2\alpha$, we have
\[
\Bigl|\,
\sum_{t=1}^{\alpha'} a_{\lambda_t} - \sum_{t=1}^{\alpha''} a_{\mu_t}
\,\Bigr| > n^{-\beta+\alpha}\,.
\]
By the Hal\'asz-type lemma with $b=n^{-\beta+\alpha}$, estimate~\eqref{eq-NV2} holds for normal values
of $\theta$, so we need to estimate the size of the set of $\theta$s that are not normal.

Fix the sets
$\{\lambda_1, \ldots , \lambda_{\alpha'}\}$, $\{\mu_1, \ldots , \mu_{\alpha''}\}$
in $\Lambda$, $\alpha'+\alpha'' \le 2\alpha$. To estimate the size of the set of abnormal
$\theta$s corresponding to these two sets of indices, we consider the trigonometric polynomial
\begin{align*}
P(\theta)
= \sum_{t=1}^{\alpha'} \bigl[ c_{\lambda_t}' c(\lambda_t\theta) - c_{\lambda_t}''s(\lambda_t\theta) \bigr]
&- \sum_{t=1}^{\alpha''} \bigl[ c_{\mu_t}' c(\mu_t\theta) - c_{\mu_t}''s(\mu_t\theta) \bigr] \\
&= \sum_{t=1}^{\alpha'} a_{\lambda_t} - \sum_{t=1}^{\alpha''} a_{\mu_t}
\end{align*}
with $2(\alpha'+\alpha'')=4\alpha$ frequencies, and let
$E=\{\theta\in [-\tfrac12, \tfrac12]\colon |P(\theta)|\le n^{-\beta+\alpha}\}$.
By Nazarov's version of the classical Tur\'an lemma~\cite[Section~1.1]{Nazarov},
\[
\max_{[-\frac12, \frac12]}\, |P| \le \Bigl( \frac{C}{|E|} \Bigr)^{4\alpha}\, \sup_E |P|\,.
\]
Recalling that $(c_\lambda')^2 + (c_\lambda'')^2 \ge 1$, we get
$ 1 \le \| P \|_{L^2[-\frac12, \frac12]}
\lesssim \bigl( C/|E| \bigr)^{4\alpha}\, n^{-\beta+\alpha} $, whence,
$ |E| \lesssim n^{-\beta/(4\alpha) + 1/4}$.

It remains to recall that there are at most
$n^{\alpha'+\alpha''} = O(n^{2\alpha})$ choices of the
distinct sets $\{\lambda_1, \ldots , \lambda_{\alpha'}\}$, $\{\mu_1, \ldots , \mu_{\alpha''}\}$
in $\Lambda$, with $\alpha'+\alpha'' \le 2\alpha$, which yields that the size of the set
of abnormal values of $\theta$ does not exceed
$O(n^{-\beta/(4\alpha) + 2\alpha + 1/4})$,
completing the first step.

\medskip\noindent{\em Step~2}. At this step, {\em assuming that
$(\xi_\lambda)$ are independent identically distributed non-degenerate
symmetric random variables,
we show that, for large enough $\beta$ and for every interval $I$ with $|I|\gtrsim 1/n$},
\[
\frac1{|I|}\, \int_I
\sup_{Z\in\bC}\, \bP\bigl[ \bigl| S(\theta) - Z \bigr| < n^{-\beta} \bigr]\,
{\rm d}\theta \lesssim n^{-\alpha}\,.
\]

\medskip
We take a collection of Rademacher random variables $(\ep_\lambda)$ independent of $(\xi_\lambda)$,
then the random variables $(\xi_\lambda)$ and $(\ep_\lambda\xi_\lambda)$ are equidistributed.
Hence,
\begin{align}\label{eq:star}
\nonumber \frac1{|I|}\,
\int_I \sup_{Z\in\bC}\, \bP^{(\xi_\lambda)}\bigl[ \bigl| S(\theta) &- Z \bigr| <  n^{-\beta} \bigr]\, {\rm d}\theta \\
\nonumber &= \frac1{|I|}\,
\int_I \sup_{Z\in\bC}\, \bE^{(\xi_\lambda)}\, \bP^{(\ep_\lambda)}\bigl[ \bigl| S(\theta) - Z \bigr|
< n^{-\beta} \bigr]\, {\rm d}\theta \\
&\le \bE^{(\xi_\lambda)} \Bigl[ \frac1{|I|}\,
\int_I  \sup_{Z\in\bC}\,  \bP^{(\ep_\lambda)}\bigl[ \bigl| S(\theta) - Z \bigr| < n^{-\beta} \bigr]\, {\rm d}\theta \Bigr]\,.
\end{align}
To apply Step~1, we need $|\xi_\lambda|\gtrsim 1$ for a positive proportion of $\lambda\in \Lambda$.
To get this, we fix a small $\delta>0$ so that $\bP[|\xi_\lambda|<\delta]\le 1-\delta$.
Then, applying the simplest form of the Bernstein--Chernoff--Hoeffding exponential concentration
to the independent identically distributed random variables
\[
X_\lambda =
\begin{cases}
1, &|\xi_\lambda|<\delta, \\
0, &|\xi_\lambda|\ge\delta,
\end{cases}
\]
we get $ \bP\bigl[ \sum_{\lambda\in\Lambda} (X_\lambda - \bE[X_\lambda]) \ge nt \bigr] \le e^{-cnt^2} $. Noting that $\bE[X_\lambda]\le 1-\delta$ and letting $t=\tfrac12\, \delta$,
we see that with probability at least $1-e^{-cn\delta^2/4}$ there are at least $\tfrac12\,   n$
indices $\lambda\in\Lambda$ for which $|\xi_\lambda|\ge \delta$. On the event that
this happens, denote by $\Lambda_1$ the set of $\lambda\in\Lambda$ for which $|\xi_\lambda|\ge \delta$, 
$S_1(\theta) = \sum_{\lambda\in\Lambda_1} r_\lambda c_\lambda e(\lambda\theta)$. 
Since $\Lambda_1$ has at least $\tfrac12\,   n$ elements, 
by Step~1, we get
\[
\sup_{Z\in\bC}\,  \bP^{(\ep_\lambda)}\bigl[ \bigl| S_1(\theta) - Z \bigr| < \delta n^{-\beta} \bigr]
\lesssim n^{-\alpha}\,,
\]
outside a set of $\theta$s of Lebesgue measure at most $O(n^{-\beta/(4\alpha) + 2\alpha + 1/4})$.
Since $S_1(\theta)$ and $S(\theta)-S_1(\theta)$ are independent, we obtain that again with probability at least $1-e^{-cn\delta^2/4}$ 
outside a set of $\theta$s of Lebesgue measure at most $O(n^{-\beta/(4\alpha) + 2\alpha + 1/4})$ we have
\[
\sup_{Z\in\bC}\,  \bP^{(\ep_\lambda)}\bigl[ \bigl| S(\theta) - Z \bigr| < \delta n^{-\beta} \bigr]
\lesssim n^{-\alpha}\,,
\]

Therefore,
\begin{multline*}
\bE^{(\xi_\lambda)}\Bigl[
\frac1{|I|}\,
\int_I  \sup_{Z\in\bC}\,  \bP^{(\ep_\lambda)}\bigl[ \bigl| S(\theta) - Z \bigr|
< \delta n^{-\beta} \bigr]\, {\rm d}\theta \Bigr] \\
\le
O(n^{-\alpha}) + O\Bigl(\frac1{|I|}\cdot n^{-\beta/(4\alpha) + 2\alpha + 1/4}\Bigr) + O(e^{-c\delta^2 n/4})
= O(n^{-\alpha})\,,
\end{multline*}
provided that $\beta\ge 20\alpha^2$.
Finally, we observe that
covering the disk centered at $Z$ of radius $n^{-\beta}$ by $O(\delta^{-2})$
disks of radius $\delta n^{-\beta}$, we get
\[
\sup_{Z\in\bC} \bP \bigl[ |S(\theta) - Z| < n^{-\beta} \bigr]
\le O(\delta^{-2}) \cdot \sup_{Z\in\bC} \bP \bigl[ |S(\theta) - Z| < \delta n^{-\beta} \bigr]\,,
\]
which, together with \eqref{eq:star}, completes the second step.

\medskip 
In particular, we get a value $\theta\in I$ such that
\[
\bP \bigl[ |S(\theta)| < n^{-\beta} \bigr]
\lesssim n^{-\alpha}.
\]

\medskip\noindent{\em Step~3}.
Here, we consider the general case. Let $\xi_\lambda'$ be independent copies of $\xi_\lambda$,
$\lambda\in\Lambda$. Then the random variables $\xi_\lambda'' = \xi_\lambda - \xi_\lambda'$
are symmetric and non-degenerate. Denote by $S_\xi$, $S_{\xi'}$ and $S_{\xi''}$ the corresponding
trigonometric sums, and note that
$|S_{\xi''}(\theta)| < 2n^{-\beta}$ provided that, for some $Z\in\bC$,
$|S_{\xi}(\theta)-Z| < n^{-\beta}$ and $|S_{\xi'}(\theta)-Z| < n^{-\beta}$.
Therefore,
\begin{multline*}
\Bigl( \bP \bigl[ |S_{\xi}(\theta)-Z| < n^{-\beta} \bigr] \Bigr)^2
= \bP \bigl[ |S_{\xi}(\theta)-Z| < n^{-\beta} \bigr] \cdot
\bP \bigl[ |S_{\xi'}(\theta)-Z| < n^{-\beta} \bigr] \\
\le \bP \bigl[ |S_{\xi''}(\theta)| < 2n^{-\beta} \bigr]\,.
\end{multline*}
By Step~2 (applied with $2\alpha$ instead of $\alpha$),
there exist $\beta$ and $\theta\in I$ such that the RHS is $O(n^{-2\alpha})$. This completes the proof of Proposition~\ref{lemma-Nguyen-Vu}.
\hfill $\Box$

\section*{Appendix C: Proof of Lemma~\ref{Lemma:sq-free}}
\renewcommand{\theequation}{C.\arabic{equation}}
\setcounter{equation}{0}

As we have already mentioned, we will closely follow Mirsky's paper~\cite{Mirsky}.
Since
$$
\mu^2(n)=\sum_{d^2|n}\mu(d),
$$
we have
\begin{align*}
\sum_{k\le x}\mu^2(k) \mu^2(k+h)
&=\sum_{\substack{a^2c-b^2d=h,\\ b^2d\le x}}\, \mu(a)\mu(b)\\
&=
\Bigl( \sum_{\substack{a^2c-b^2d=h, \\
b^2d\le x,\,ab\le y}} \
+ \sum_{\substack{a^2c-b^2d=h,\\
b^2d\le x,\, ab>y}}\, \Bigr) \mu(a)\mu(b)=I_1+I_2,
\end{align*}
for a large parameter $y$ to be fixed later on.

First,
$$
I_1 =
\sum_{\substack{\gcd(a^2,b^2)|h,\\ ab\le y}}\, \mu(a)\mu(b)\,
\sum_{\substack{a^2c-b^2d=h,\\1\le b^2d\le x}}1.
$$
Set
\[
t=\gcd(a,b), \ a'=\frac{a^2}{t^2}, \ a''=\frac{a}t, \ b'=\frac{b^2}{t^2}, \ b''=\frac{b}t, \ h'=\frac{h}{t^2},
\ x'=\frac{x}{t^2},
\]
and choose $c_0, d_0$ so that $a'c_0-b'd_0=h'$.
Then
$$
\sum_{\substack{a'c-b'd=h',\\ 1\le b'd\le x'}}\, 1
= \sum_{\substack{a'(c-c_0)=b'(d-d_0),\\1-b'd_0\le b'(d-d_0)\le x'-b'd_0}}\, 1
=\frac{x'}{a'b'}+q
$$
with $|q|<2$.
Therefore,
\begin{multline*}
I_1=\sum_{\substack{\gcd(a^2,b^2)|h,\\ ab\le y}}\,
\mu(a)\mu(b)\Bigl(\frac{x\gcd(a^2,b^2)}{a^2b^2}+O(1)\Bigr) \\
= x\, \sum_{\substack{\gcd(a^2,b^2)|h,\\ab\le y}}\,
\frac{\mu(a)\mu(b)\gcd(a^2,b^2)}{a^2b^2} + O(y\log y).
\end{multline*}

Since
\begin{multline*}
\Bigl|\, \sum_{\substack{\gcd(a^2,b^2)|h,\\ab> y}}\,
\frac{\mu(a)\mu(b)\gcd(a^2,b^2)}{a^2b^2}\, \Bigr| \le  \sum_{\substack{\gcd(a^2,b^2)|h,\\ab> y}}\,
\frac{\gcd(a^2,b^2)}{a^2b^2}\\ =
\sum_{\substack{t^2|h,\\ t^2a''b''> y}}\,
\frac{1}{t^2a''^2b''^2} 
\le \sum_{a'',b''\ge 1}\, \frac{1}{a''^2b''^2}\, \sum_{t>\sqrt{y/(a''b'')}}\, \frac{1}{t^2}
\\ \lesssim
\frac1{\sqrt{y}}\,
\sum_{a'',b''\ge 1}\, \frac{1}{(a''b'')^{3/2}} \lesssim \frac1{\sqrt{y}}\,,
\end{multline*}
we conclude that
$$
I_1= x\sum_{\gcd(a^2,b^2)|h}\frac{\mu(a)\mu(b)\gcd(a^2,b^2)}{a^2b^2}+O\Bigl(\frac{x}{\sqrt y}+y\log y\Bigr).
$$

Next,
\[
\sum_{\gcd(a^2,b^2)|h}\frac{\mu(a)\mu(b)\gcd(a^2,b^2)}{a^2b^2}\, =
\sum_{\substack{\gcd(a'',b'')=1,\,\gcd(a'',t)=1,\\
\gcd(b'',t)=1,\,t^2|h}}\, \frac{\mu(ta'')\mu(tb'')}{t^2a''^2b''^2}\, =:A.
\]
If $s^2$ is the largest square divisor of $h$, then
\begin{align*}
A &= \sum_{\mu^2(t)=1,\, t|s}\, \frac1{t^2}\, \sum_{\substack{\gcd(a'',b'')=1,\, \gcd(a'',t)=1,\\
\gcd(b'',t)=1}}\, \frac{\mu(a'')\mu(b'')}{a''^2b''^2} \\
&= \sum_{\mu^2(t)=1,\, t|s}\, \frac1{t^2}\, \sum_{\gcd(d,t)=1}\, \frac{\mu(d)\tau(d)}{d^2},
\end{align*}
where $\tau(d)$ is the number of the divisors of $d$.
Therefore,
\begin{multline*}
A = \sum_{\mu^2(t)=1,\,t|s}\frac1{t^2}\prod_{p\,\centernot\mid \,t}\Bigl(1-\frac2{p^2}\Bigr)=
\prod_p\Bigl(1-\frac2{p^2}\Bigr)\sum_{\mu^2(t)=1,\,t|s}\frac1{t^2}\prod_{p|t}\Bigl(1-\frac2{p^2}\Bigr)^{-1}\\=
D\, \prod_{p|s}\Bigl(1+\frac1{p^2(1-2/p^2)}\Bigr)
=D\, \prod_{p^2 | h} \Bigl(1+\frac1{p^2-2} \Bigr)=D(h).
\end{multline*}
Thus,
\begin{equation}
I_1= D(h)x+O\Bigl(\frac{x}{\sqrt y}+y\log y\Bigr).
\label{rtf17}
\end{equation}

Next,
$$
|I_2|=\Bigl|\sum_{\substack{a^2c-b^2d=h,\\ b^2d\le x,\, ab>y}}\mu(a)\mu(b)\Bigr|\le
\sum_{cd<x(x+h)/y^2}\,\, \sum_{\substack{a^2c-b^2d=h,\\b^2d\le x}}\, 1.
$$
Fix $\varepsilon>0$. By \cite[Lemma 4]{CilGar}, the equation $a^2cd-(bd)^2=hd$ has at most $C_\varepsilon x^\varepsilon$ solutions $(a,bd)$ if
$cd$ is not a perfect square. If $cd=m^2$, then the equation $(am)^2-(bd)^2=hd$ has at most $C_\varepsilon x^\varepsilon$ solutions $(am,bd)$. Therefore,
\begin{equation}
|I_2|\le C_\varepsilon x^\varepsilon\, \sum_{cd<2x^2/y^2}\, 1
\lesssim C_\varepsilon x^\varepsilon\, \frac{x^2}{y^2}\,\log(2x^2/y^2).
\label{rtf18}
\end{equation}

Finally, set $y=x^{2/3}$. Now, \eqref{rtf17} and \eqref{rtf18} yield
$$
\Bigl|\sum_{k\le x}\mu^2(k) \mu^2(k+h)-D(h)x\Bigr|\le C_\varepsilon x^{(2/3)+2\varepsilon},
\qquad 1\le h\le x,
$$
completing the proof. \hfill $\Box$

\section*{Appendix D: Proof of Lemma~\ref{Lemma:TM1}}

The following argument reproduces rather closely Mahler's original proof.
Set
$$
S(x,h)=\sum_{0\le k < x} \xi(k) \xi(k+h).
$$
Then
$$
S(x,h)=\sum_{0\le k < x/2} \xi(2k) \xi(2k+h)+\sum_{0\le k < x/2} \xi(2k+1) \xi(2k+h+1)+\Delta(x,h),
$$
where
$$
\Delta(x,h)=
\begin{cases}
-\xi(x)\xi(x+h) & x {\rm\ is\ odd}; \\
0 & {\rm\ otherwise}.
\end{cases}
$$
Therefore, we have
\begin{align*}
S(2x,2h)&=2S(x,h),\\
S(2x+1,2h)&=2S(x+1,h)+\Delta(2x+1,2h),\\
S(2x,2h+1)&=-S(x,h)-S(x,h+1),\\
S(2x+1,2h+1)&=-S(x+1,h)-S(x+1,h+1)+\Delta(2x+1,2h+1)
\end{align*}
for positive $x$ and $h$.
Since $S(x,0)=x$, $x\ge 0$, we obtain by induction in $h$ that
$$
\bigl|S(x,h)-\sigma(h)x\bigr|\le Ch\log(x+1),\qquad x\ge 1,
$$
for some positive numerical constant $C$. \hfill $\Box$

\bigskip
\medskip

\noindent J.B.:
School of Mathematics, Tel Aviv University, Tel Aviv, Israel
\newline {\tt benatar@mail.tau.ac.il}
\smallskip\newline\noindent{A.B.: Institut de Math\'ematiques de Marseille,
Aix Marseille Universit\'e, CNRS, Centrale Marseille, I2M, Marseille, France
\newline {\tt alexander.borichev@math.cnrs.fr}
\smallskip\newline\noindent M.S.:
School of Mathematics, Tel Aviv University, Tel Aviv, Israel
\newline {\tt sodin@tauex.tau.ac.il}
}


\begin{thebibliography}{A}

\bibitem{BBS} J. Benatar, A. Borichev, M. Sodin,
The ``pits effect" for entire functions of exponential type and the Wiener spectrum, to appear in J. Lond.\ Math.\ Soc.,  
{\tt arXiv:1908.09161}.

\bibitem{BNS} A.~Borichev, A.~Nishry, M.~Sodin, Entire functions of exponential type represented by pseudo-random and random Taylor series, J. Anal.\ Math.\ {\bf 133} (2017), 361--396.

\bibitem{CS} F.~Cellarosi, Ya.~Sinai,
Ergodic properties of square-free numbers,
J. Eur.\ Math.\ Soc.\ {\bf 15} (2013), 1343--1374.

\bibitem{CilGar} J.~Cilleruelo, M.~Garaev,
Concentration of points on two and three dimensional modular hyperbolas and applications,
Geom.\ Funct.\ Anal.\ {\bf 21} (2011), 892--904.

\bibitem{CL} Y.~W.~Chen, J.~E.~Littlewood,
Some new properties of power series, 
Indian J. Math.\ {\bf 9} (1967), 289--324.

\bibitem{EO}
A.~Eremenko, I.~Ostrovskii, On the ``pits effect'' of Littlewood and Offord, 
Bull.\ Lond.\ Math.\ Soc.\ {\bf 39} (2007), 929--939.

\bibitem{Halasz} G. Hal\'asz,
Estimates for the concentration function of combinatorial number theory and probability, 
Period.\ Math.\ Hungar.\ {\bf 8} (1977), 197--211.

\bibitem{HL} G.~H.~Hardy, J.~E.~Littlewood,
Some problems of Diophanitine approximation. II.
The trigonometric series associated with the elliptic $\vartheta$-functions, Acta Math.\ {\bf 37} (1914), 193--239. 

\bibitem{HW} G.~H.~Hardy, E.~M.~Wright,
An introduction to the theory of numbers.
Sixth edition. Revised by D.~R.~Heath-Brown and J.~H.~Silverman. With a foreword by Andrew Wiles.
Oxford University Press, Oxford, 2008.

\bibitem{Hayman} W. K. Hayman,
Angular value distribution of power series with gaps, Proc.\ Lond.\ Math.\ Soc.\ (3)
{\bf 24} (1972), 590--624.

\bibitem{HK} W. K. Hayman, P. B. Kennedy,
Subharmonic functions. Vol. I. Academic Press, London, 1976.

\bibitem{HR}
W. K. Hayman, J. F. Rossi,
Characteristic, maximum modulus and value distribution, 
Trans.\ Amer.\ Math.\ Soc.\ {\bf 284} (1984), 651--664.

\bibitem{HKPV}
B. Hough, M. Krishnpur, Y. Peres, B. Virag,
Zeros of Gaussian analytic functions and determinantal point processes.
University Lecture Series, {\bf 51}. American Mathematical Society, Providence, RI, 2009.

\bibitem{KZ} Z.~Kabluchko, D.~Zaporozhets,
Asymptotic distribution of complex zeros of random analytic
functions, Ann.\ Prob.\ {\bf 42} (2014), 1374--1395.

\bibitem{Kakutani} S. Kakutani, Strictly ergodic symbolic dynamical systems.
Strictly ergodic symbolic dynamical systems.
Proceedings of the Sixth Berkeley Symposium on Mathematical Statistics and Probability,
Vol. II: Probability theory, pp. 319--326. Univ. California Press, Berkeley, Ca.,
1972.

\bibitem{Khinchin} A. Ya. Khinchin, Continued fractions.  With a preface by B. V. Gnedenko.
Dover Publications, Inc., Mineola, NY, 1997.

\bibitem{Laczk} M. Laczkovich,
Uniformly spread discrete sets in $\bR^d$, J. Lond.\ Math.\ Soc.\ (2) {\bf 46} (1992), 39--57.

\bibitem{LMR} S. Lester, K. Matom\"aki, M. Radziwi{\l}{\l},
Small scale distribution of zeros and mass of modular forms, J. Eur.\ Math.\ Soc.\ 
{\bf 20} (2018), 1595--1627.

\bibitem{Levy} P.~L\'evy,
Sur la croissance des fonctions enti\`eres, 
Bull.\ Soc.\ Math.\ France, {\bf 58}  (1930),  29--59, 127--149.

\bibitem{Littlewood}
J.~E.~Littlewood,
A ``pits effect'' for all smooth enough integral functions with a coefficient factor
$ \exp(n^2 \alpha\pi{\rm i})$, $\alpha = \frac12 (\sqrt5-1)$.
J. Lond.\  Math.\ Soc.\ {\bf 43} (1968), 79--92.

\bibitem{Littlewood2}
J. E. Littlewood,
The ``pits effect'' for the integral function \\
$ f(z) = \sum \exp\{ -\vartheta^{-1}(n\log n - n) + \pi {\rm i}\alpha n^2\}z^n$,
$\alpha = \frac12\, (\sqrt{5}-1) $, 
1969 Number Theory and Analysis (Papers in Honor of Edmund Landau) pp. 193--215.
Plenum, New York.

\bibitem{LO}
J.~E.~Littlewood, A.~C.~Offord,
On the distribution of zeros and $a$-values of a random integral function. II,
Ann.\ Math.\ (2) {\bf 49} (1948), 885--952; errata {\bf 50} (1949), 990--991.

\bibitem{Mahler} K. Mahler,
The spectrum of an array and its application to the study of the translation
properties of a simple class of arithmetical functions. II: On the translation properties of a
simple class of arithmetical functions, 
J. Math.\ Phys.\ (MIT) {\bf 6} (1927), 158--163.

\bibitem{MaSa}
C. Mauduit, A. S\'ark\"ozy,
On Finite Pseudorandom Binary Sequences, II.
The Champernowne, Rudin--Shapiro, and Thue--Morse sequences.
A Further Construction, 
J. Number Th.\ {\bf 73} (1998), 256--276.

\bibitem{Mirsky} L.~Mirsky,
On the frequency of pairs of square-free numbers with a given difference,
Bull.\ Amer.\ Math.\ Soc.\ {\bf 55} (1949), 936--939.

\bibitem{Nassif} M.~Nassif,
On the behaviour of the function $ f(z)=\sum_{n=0}^\infty e^{\sqrt2\pi{\rm i} n^2}\, \frac{z^{2n}}{n!}$, 
Proc.\ Lond.\ Math.\ Soc.\ (2) {\bf 54} (1952), 201--214.

\bibitem{Nazarov} F. Nazarov,
Local estimates for exponential polynomials and their applications to inequalities of the uncertainty principle type, St.\ Petersburg Math.\  J. {\bf 5} (1994), 663--717.

\bibitem{NNS} F. Nazarov, A. Nishry, M. Sodin,
Distribution of zeroes of Rademacher Taylor series, 
Ann.\ Fac.\ Sci.\ Toulouse Math.\ (6) {\bf 25} (2016), 759--784.

\bibitem{NS-WhatIs} F. Nazarov, M. Sodin,
What is…a Gaussian entire function? Notices Amer.\ Math.\ Soc.\ {\bf 57} (2010),
375--377.

\bibitem{NSV} F. Nazarov, M. Sodin, A. Volberg,
The Jancovici--Lebowitz--Manificat law for large fluctuations of random complex zeroes, 
Comm.\ Math.\ Phys.\ {\bf 284} (2008), 833--865.

\bibitem{NV} O. Nguyen, V. Vu,
Roots of random functions: A general condition for local universality, {\tt arXiv:1711.03615}.

\bibitem{Off-Rand1} A. C. Offord,
 The distribution of the values of an entire function whose coefficients are independent random variables, Proc.\ Lond.\ Math.\ Soc.\ (3) {\bf 14a} (1965), 199--238.

\bibitem{Off-Lac1} A. C. Offord,
The pits property of entire functions, J. Lond.\ Math.\ Soc.\ (2) {\bf 44} (1991), 463--475.

\bibitem{Off-Lac2} A. C. Offord,
Lacunary entire functions, Math.\ Proc.\ Cambridge Philos.\ Soc.\ {\bf 114} (1993), 67--83.

\bibitem{Off-Rand2} A. C. Offord,
The distribution of the values of an entire function whose coefficients are independent random variables. II, Math.\ Proc.\ Cambridge Philos.\ Soc.\ {\bf 118} (1995), 527--542.

\bibitem{Sjostrand} J. Sj\"ostrand,
Non-self-adjoint differential operators, spectral asymptotics and random perturbations. Pseudo-Differential Operators. Theory and Applications, {\bf 14}. Birkh\"auser/Springer, Cham, 2019.

\bibitem{ST2} M. Sodin, B. Tsirelson,
Random complex zeroes. II. Perturbed lattice, Israel J. Math.\ {\bf 152} (2006), 105--124.

\bibitem{ST3} M. Sodin, B. Tsirelson,
 Random complex zeroes. III. Decay of the hole probability, Israel J. Math.\ {\bf 147} (2005),
 371--379.

\bibitem{ST-Transp} M. Sodin, B. Tsirelson,
Uniformly spread measures and vector fields, 
J. Math.\ Sci.\ (N.Y.) {\bf 165} (2010), 491--497.

\bibitem{Tims} S. R. Tims,
Note on a paper by M.~Nassif, Proc.\ Lond.\ Math.\ Soc.\ (2) {\bf 54} (1952), 215--218.

\bibitem{Wiener} N. Wiener,
The spectrum of an array and its application to the study of the translation
properties of a simple class of arithmetical functions. I: The spectrum of an array, 
J. Math.\ Phys.\ (MIT) {\bf 6} (1927), 145--157.

 

\end{thebibliography}
\end{document}